\documentclass{amsart}
\usepackage{amssymb,amsmath}
\usepackage{graphicx}
\usepackage{subfigure}
\usepackage{ upgreek }
\usepackage{tikz}
\usepackage{comment}
\usepackage{enumerate}
\usepackage[all]{xy}
\usepackage{url} 
\usetikzlibrary{matrix}

\newcommand{\bC}{{\mathbb C}}

\newcommand{\bR}{{\mathbb R}}
\newcommand{\bZ}{{\mathbb Z}}

\newcommand{\cF}{\mathcal F}
\newcommand{\cG}{\mathcal G}
\newcommand{\cA}{\mathcal A}

\newcommand{\cC}{\mathcal C}

\newcommand{\cY}{\mathcal Y}

\newcommand{\cO}{\mathcal O}

\newcommand{\cU}{\mathcal U}
\newcommand{\cL}{\mathcal L}
\newcommand{\cR}{\mathcal R}
\newcommand{\cT}{\mathcal T}
\newcommand{\cH}{\mathcal H}

\DeclareMathOperator{\val}{val}

\newtheorem*{question}{Question}
\newtheorem{theorem}{Theorem}[section]
\newtheorem{proposition}[theorem]{Proposition}
\newtheorem{lemma}[theorem]{Lemma}
\newtheorem{remark}[theorem]{Remark}

\newtheorem{definition}[theorem]{Definition}

\newtheorem{corollary}[theorem]{Corollary}
\newtheorem{conjecture}[theorem]{Conjecture}

\begin{document}

\title{Closed string mirrors of symplectic cluster manifolds}

\author{Yoel Groman 
 and Umut Varolgunes
}

\address{ Yoel Groman,
Hebrew University of Jerusalem,
Mathematics Department}
\address{Umut Varolgunes, School of Mathematics, Bo\u{g}azi\c{c}i University}

\begin{abstract}
For the base $B$ of a Maslov $0$ Lagrangian torus fibration with singularities consider the sheaf assigning to each $P\subset B$ the relative symplectic cohomology in degree $0$ of its pre-image. We compute this sheaf for nodal Lagrangian torus fibrations on four dimensional symplectic cluster manifolds. We show that it is  the pushforward of the structure sheaf of a certain rigid analytic space under a non-archimedean torus fibration. The  rigid analytic space is constructed in a canonical way from the relative SH sheaf and is referred as the \emph{closed string mirror}. The construction relies on computing relative SH for local models by applying general axiomatic properties rather than ad hoc analysis of holomorphic curves. These axiomatic properties include previously established ones such as the Mayer-Vietoris property and locality for complete embeddings; and new ones such as the Hartogs property and the holomorphic volume form preservation property of wall crossing in relative $SH$.  We indicate some higher dimensional settings where the same techniques apply. 

\end{abstract}

\maketitle
\setcounter{tocdepth}{1}
\tableofcontents

\section{Introduction}
Consider a symplectic manifold $M$ which is geometrically bounded \cite[Section 1.1]{groman}. To each compact set $K\subset M$ one can associate a closed string invariant denoted by $SH^*_M(K)$ known as \emph{relative symplectic cohomology}. This assignment is contravariantly functorial with respect to inclusions. We consider the case where $\Lambda^{top}_\mathbb{C}(TM,J)$ for some compatible almost complex structure $J$ is trivial, and we fix a trivialization. Then  the invariant $SH^*_M(K)$ is $\bZ$-graded and is defined over the Novikov ring $$\Lambda_{\geq 0}:=\left\{\sum_{i\geq 0} a_iT^{\alpha_i}\mid a_i\in\mathbb{F}, \alpha_i\in\mathbb{R}_{\geq 0}, \text{ where } \alpha_i\to\infty \text{, as } i\to\infty\right\},$$ where $\mathbb{F}$ is a commutative ring. We briefly review the construction in \S\ref{SecSHreview}.

Suppose now that $M$ is equipped with a proper smooth map $\pi:M\to B,$
for $B$ a smooth manifold of half the dimension of $M$, such that there exists open $U\subset B$ with $\pi^{-1}(U)$ dense in $M$ and the fibers above the points of $U$ are Lagrangian tori. Note that $\pi$ is therefore an involutive map as in \cite[Definition 3]{varolgunes}. For the elementary proof, including the definition, see Lemma \ref{lem-lag-inv}. Then, we may consider the contravariant functor $P\mapsto SH^*_M(\pi^{-1}(P))$ on the compact sets of $B$. It is shown in \cite{varolgunes} that this presheaf satisfies a Mayer Vietoris property, see Definition 1 and Theorem 1.3.4 from \cite{varolgunes}. Then, assuming that $SH^*_M(\pi^{-1}(P))$ are all non-negatively graded, if we restrict attention to degree $0$, that is $P\mapsto SH^0(\pi^{-1}(P))$, this presheaf is in fact a sheaf provided one replaces the usual notion of a topology with that of a Grothendieck topology, or $G$-topology, in which the compact sets are declared as open and only finite covers are considered. A brief review of this notion in the generality relevant for us is given in Section \ref{ss-eigenray}, see in particular Remark \ref{rem-Hausdorff-G-top}. The proof of the stated sheaf property is recalled in Theorem \ref{thm-mv-recall}.



In this paper, barring Section \S\ref{secFutureWork}, we restrict attention to the $4$-dimensional setting of \emph{symplectic cluster manifolds}. That is, $M$ is a symplectic $4$-manifold, $\pi$ is an almost toric fibration with only focus-focus type singularities \cite[Definition 2.2]{leung}, and a base diagram \cite[Section 5.3]{symington} for $\pi$ is given by an \emph{eigenray diagram $\cR$}. A brief review of the latter notion and how it gives rise to a nodal integral affine manifold that we denote by $B_\cR$ is given in \S\ref{s-integral-affine}. Such $\pi$ always has a Lagrangian section whose image is disjoint from the focus-focus singularities. In fact, the eigenray diagram determines \cite[Section 7.1]{locality} an almost toric fibration $\pi_{\cR}:M_{\cR}\to B_{\cR},$ which reconstructs $\pi$ up to an appropriate notion of isomorphism \cite[Proposition 3.5]{leung}, \cite[Proposition 7.8]{locality}. Therefore, we will think of our symplectic cluster manifolds as arising from eigenray diagrams. It is well-known that $\Lambda^{top}_\mathbb{C}TM_\mathcal{R}$ admits a trivialization such that the non-negatively graded property from the previous paragraph holds for $\pi_\cR$ (see \cite[Section 7.6]{locality} and Section 5 of this paper).

Rather then the $G$-topology of all compact sets in $B_\cR$, we will always consider a weaker $G$-topology on $B_\cR$, the $G$-topology of admissible polygons. Admissible opens in this topology are (roughly speaking) finite unions of convex, possibly degenerate, rational polygons such that if the polygon contains a node, then it contains only one node which is an interior point in the standard topology of $B_\cR$. Admissible coverings are finite covers of admissible polygons by admissible polygons. See Section \ref{s-integral-affine} for definitions and more details. 
 
 The aim of this paper is to fully describe the sheaf of algebras over the Novikov field $\Lambda:=\left\{\sum_{i\geq 0} a_iT^{\alpha_i}\mid a_i\in\mathbb{F}, \alpha_i\in\mathbb{R}, \text{ where } \alpha_i\to\infty \text{, as } i\to\infty\right\}$ defined by:$$\mathcal{F}(P):=SH^0_M(\pi^{-1}(P);\Lambda)=SH^0_M(\pi^{-1}(P))\otimes_{\Lambda_{\geq 0}}\Lambda$$ and use it to reconstruct in a canonical way a non-archimedean analytic mirror of $M$ and some of its related structures. We could compute the sheaf over the Novikov ring with our methods but this adds a layer of complexity to the statements. Since we will be working in the framework of rigid analytic spaces for the construction of the mirror (see Theorem \ref{thm-mirror-cons}), we decided to simplify the discussion by passing to the Novikov field already at this point.
\\


In our discussions involving non-archimedean geometry, we will assume for simplicity that $\mathbb{F}$ is a characteristic $0$ algebraically closed field, which implies that $\Lambda$ is also an algebraically closed field \cite[Appendix 1]{FOOO}. Let $\mathcal{Y}$ be a rigid analytic space over $\Lambda$ in the sense of Tate and $B$ a nodal integral affine manifold. We call a map $p:\mathcal{Y}\to B$ \emph{continuous} if  the preimage of each admissible polygon is an admissible open of $\mathcal{Y}$ and the preimages of admissible coverings are admissible coverings in $\mathcal{Y}$. We call such a continuous map \emph{Stein} if around every $b\in B$ there is an admissible convex polygon $P_b$ which contains $b$ in its interior such that for any admissible convex $Q\subset P_b$, the rigid analytic space $p^{-1}(Q)$ is isomorphic to an affinoid domain. The definitions of the \emph{nodal slide} and \emph{branch move} operations on eigenray diagrams can be reviewed from \cite[Section 7.2]{locality}.

\begin{theorem}\label{thm-mirror-cons}
Let $\pi_\cR:M_{\mathcal{R}}\to B_{\mathcal{R}}$ be a Lagrangian fibration associated to an eigenray diagram $\mathcal{R}$. There exists a rigid analytic space $\mathcal{Y}_{\mathcal{R}}$ over $\Lambda$ with a Stein continuous map $p_\mathcal{R}: \mathcal{Y}_{\mathcal{R}}\to B_{\mathcal{R}}$ such that the push-forward under $p_\mathcal{R}$ of the structure sheaf of $\mathcal{Y}_{\mathcal{R}}$ is canonically isomorphic as a sheaf of $\Lambda$-algebras to the sheaf $$\cF_\mathcal{R}(P):=SH^0_{M_\mathcal{R}}(\pi_{\cR}^{-1}(P);\Lambda).$$ Moreover,

\begin{enumerate}
\item $p_\mathcal{R}: \mathcal{Y}_{\mathcal{R}}\to B_{\mathcal{R}}$ is an affinoid torus fibration in the complement of the set of nodal points $N_\cR\subset B_\cR$.
\item The space $\mathcal{Y}_{\mathcal{R}}$ is smooth if and only if the nodes of $\cR$ all have multiplicity $1$, which is equivalent to $\pi_\cR$ having at most one critical point in each fiber.
\item There is a canonical volume form $\Omega_\cR$  on the smooth locus of $\mathcal{Y}_{\mathcal{R}}$. 

\item Let $\cR_1$ and $\cR_2$ be eigenray diagrams related by a branch move. Then, we have an analytic isomorphism $\cY_{\cR_1}\simeq \cY_{\cR_2}$ intertwining the torus fibrations and the volume forms.
\item \label{thm-mirror-cons:it4}Let $\cR_1$ and $\cR_2$ be eigenray diagrams related by a nodal slide that preserve all the multiplicities. Then, we have an analytic isomorphism $\cY_{\cR_1}\simeq \cY_{\cR_2}$ between the associated mirrors. The isomorphism intertwines the volume forms, but not the torus fibrations.

\item The symplectic invariant $SH_{M_{\mathcal{R}}}^0(M_{\mathcal{R}};\Lambda)$ is isomorphic as a $\Lambda$-algebra to the algebra of entire functions on $\mathcal{Y}_{\mathcal{R}}$.
\end{enumerate}

\end{theorem}
\begin{remark}
In Theorem \ref{thm-mirror-cons} item (\ref{thm-mirror-cons:it4}) it is important that we consider only paths where the multiplicities of nodes are not changed. See Remark \ref{rmkMultiplicityResolution} for what happens when multiplicities are modified by nodal sliding.
\end{remark}

\begin{remark}
We do not claim that the preimage under $p_\mathcal{R}$ of every admissible convex polygon $P$ is an affinoid domain, only when it is sufficiently small (see the definition of a \emph{small admissible polygon} in Section \ref{ss-eigenray} and Proposition \ref{prpLocalSmall}). Tony Yue Yu informed us that a straightforward induction based on \cite[Proposition 6.5]{yu} shows that when $P$ does not contain a node, then indeed the preimage is always an affinoid domain.
\end{remark}

\subsection{The construction}

We will abbreviate $MaxSpec$ by $M$ in this paper. For $A$ an affinoid algebra, we will denote by $M(A)$ the corresponding affinoid domain.

The underlying set of the rigid analytic space $\cY_{\cR}$ of Theorem \ref{thm-mirror-cons} has the following simple description. For each admissible convex polygon denote by $M(P):=M(\cF_{\cR}(P))$ the set of maximal ideals of $\cF_{\cR}(P)$. Given a cover $\cU$ of $B_\cR$ by admissible convex polygons let $\cY_{\cU}:=\coprod_{P\in\cU} M(P)/\sim$ where we quotient by the equivalence relation generated by pairs $(\rho_1(z),\rho_2(z))\in M(P)\times M(Q)$  where $\rho_1,\rho_2$ are the inclusion maps $M(P\cap Q)\to M(P)$ and $M(P\cap Q)\to M(Q)$ induced respectively by the restriction maps $\cF_{\cR}(P)\to\cF_{\cR}(P\cap Q)$ and $\cF_{\cR}(Q)\to\cF_{\cR}(P\cap Q)$ and $z\in M(P\cap Q)$. If $\cU'$ is a refinement of $\cU$ there is a natural map $\cY_{\cU'}\to\cY_{\cU}$. We define $\cY:=\varprojlim_{\cU}\cY_{\cU}$.

We refer to this as the \emph{closed string mirror}. Note that it is not a priori clear that this construction gives rise to a rigid analytic space. For this we need to know first of all that the gluing relation is actually an equivalence relation. In addition, we need to know that for $Q\subset P$ small enough $\cF_{\cR}(P)$ is affinoid and the restriction map $\cF_{\cR}(P)\to\cF_{\cR}(Q)$ induces an inclusion of an affinoid subdomain. 


The following theorem gives the necessary properties. In the statement we refer to  the notion of a \emph{small admissible polygon} whose definition appears in Section \ref{ss-sheaf-G}.  Every admissible open in $B_\cR$ is a finite union of small admissible polygons and the intersection of a pair of small admissible polygons is a small admissible polygon.

\begin{theorem}\label{thm-four-prop}
For $\mathcal{R}$ an eigenray diagram, the sheaf $\mathcal{F}_\cR$ satisfies the following properties

\begin{itemize}
\item (Affinoidness) If $P$ is a small admissible polygon, $\mathcal{F}_\cR(P)$ is a reduced affinoid algebra.
\item (Subdomain) If $Q\subset P$ are small admissible polygons, then the morphism of affinoid domains induced from the restriction map $\mathcal{F}_\cR(P)\to \mathcal{F}_\cR(Q)$: $$M(\mathcal{F}_\cR(Q))\to M(\mathcal{F}_\cR(P))$$ has image an affinoid subdomain and it is an isomorphism of affinoid domains onto its image.
\item (Strong cocycle condition) For small admissible polygons $Q,Q'\subset P$, we have $$im(M(\mathcal{F}_\cR(Q)))\cap im(M(\mathcal{F}_\cR(Q')))=im(M(\mathcal{F}_\cR(Q\cap Q')))$$ where $im$ denotes the image in $M(\mathcal{F}_\cR(P))$ induced by the relevant restriction map. \footnote{An important non-trivial special case is that if $Q$ and $Q'$ are disjoint, then so are $im({M}(Q))$ and $ im({M}(Q'))$.} 
\item (Independence) Let $Q_1,\ldots, Q_N\subset P$ be small admissible polygons such that $P=\bigcup Q_i$. Then $$\bigcup  im(M(\mathcal{F}_\cR(Q_i)))=  M(\mathcal{F}_\cR(P)).$$

\item (Separation) Let $Q\subset P$ be an inclusion of small admissible polygons. For any $p\in M(\cF_\cR(P))\setminus im(M(\cF_{\cR}(Q))$ there is a small admissible polygon $Q'\subset P$ which is disjoint of $Q$ such that $p\in im(M(\cF_{\cR}(Q')))$. 

\end{itemize}

\end{theorem}

The above theorem allows us to endow the closed string mirror with a  $G$-topology and a structure sheaf turning it into a rigid analytic space. This is done in Section \ref{s-cons}. In addition we show that  any rigid analytic space $X$ constructed from a sheaf $\cF$ on $B_{\cR}$ satisfying the above properties is endowed with a  Stein continuous map $p:X\to B_{\cR}$ so that there is a canonical isomorphism of sheaves of $\Lambda$-algebra $\cF\simeq p_*\cO_X$. 

\begin{remark}
The proof of Theorem \ref{thm-four-prop} relies on local computations and results of \cite{locality} concerning locality of relative SH for complete embeddings. More details are given in the next subsection. We pose the following 
\begin{question}
Is it possible to prove the properties in Theorem \ref{thm-four-prop} as a priori properties of the sheaf $\cF_{\cR}$ without first computing?
\end{question}
We expect the answer is positive for singular Lagrangian torus fibrations satisfying some natural hypotheses. See Section \ref{secFutureWork} for one class of examples. This is pursued in other work. 
\end{remark}

\begin{remark}
There is an alternative approach to constructing the closed string mirror. 
Namely, we could define the underlying set of the mirror as 
$$
\cY=\coprod_{b\in B_{\cR}}M(\cF_{\cR}(\{b\})).
$$
We expect this construction is correct for more general Lagrangian torus fibrations provided neighborhoods of singular fibers are modeled on Liouville domains equipped with a radiant Lagrangian torus fibration whose Liouville field lifts the Euler vector field in the base. The construction would have a number of conceptual advantages. Among other things, there is no need to establish the cocycle condition, and the function $p_{\cR}:\cY\to B_{\cR}$ is obvious from the construction. However, in order to endow $\cY$ with the structure of a rigid analytic space we still need to show that for sufficiently small polygons $P$ the natural restriction maps $\cF(P)\to \cF(\{b\})$ for $b\in P$ give rise to a bijection $M(\cF(P))\simeq p_{\cR}^{-1}(P)$. The proof of this statement ends up amounting to proving Theorem \ref{thm-four-prop} and the additional arguments going into the reconstruction of $p_{\cR}$ from the gluing construction.
\end{remark}








\subsection{Local computations}\label{ss-local-comp}

We now review the local computations underlying Theorem \ref{thm-four-prop}. 
For an integer $k\geq 0$ we denote by $B_k$ the nodal integral affine structure on $\bR^2$ which is induced by a Lagrangian torus fibration $M\to \mathbb{R}^2$ with a single possibly non-smooth fiber at the origin with $k$ focus-focus singularities. The complete definition of this nodal integral affine manifold is in Section \ref{ss-integral-affine-local}. We denote by $M_k$ the corresponding symplectic manifold and by $\cF_k$ the sheaf of relative $SH^0$ on $B_k$ with its $G$-topology (see Section \ref{ss-6.1}).

Consider the affine variety $$Y_k=Spec(\Lambda [x,y,u^{\pm}]/(xy-(u+1)^k)$$ for $k\geq 0$ and its rigid analytification $Y_k^{an}$. A review of this notion is given in \S \ref{s-analytification}. Note that $Y_k$ is smooth if and only if $k=0,1$. Since $\Lambda$ is algebraically closed, as a set $$Y_k^{an}=M(\Bbbk[x,y,u^{\pm}]/(xy-(u+1)^k).$$ In Section \ref{Sec-Yk-analyt}, we show that $Y_k^{an}$ admits a Stein continuous map $p_k$ to $\bR^2$. A construction by \cite{koso} associates with such a map an integral affine structure with singularities on the base.  For $k=0$, we have $Y_0^{an}\simeq (\Lambda^*)^2$ as a set and $p_0$ is nothing but the standard tropicalization map. We show that for any $k$ this nodal integral affine structure is isomorphic to $B_k$.  This is a slight generalization of \cite[Section 8]{koso} which addresses the case $k=1$. We can now state our main local computation.

\begin{theorem}\label{thm-local-comp}
There is an isomorphism of sheaves of $\Lambda$-algebras on $B_k$
\begin{equation}\label{eqthm-local-comp}
\cF_k\simeq (p_k)_*\cO_{Y_k^{an}}. 
\end{equation}
\end{theorem}

\begin{remark}\label{rmkMultiplicityResolution}
Note that for $k>1$,  using nodal slides, we can find a multiplicity $1$ eigenray diagram $\mathcal{R}$ such that $M_\cR$ is symplectomorphic (fiber preserving outside a compact set) to $M_k$ \cite[Proposition 7.8]{locality}. The  mirror we produce for this $M_{\cR}$ is the analytification of the resolution of the $A_k$ singularity of $Spec(\Lambda[x,y,u^{\pm}]/(xy-(1+u)^k).$ Compare with \cite{pomerleano}. This reflects the fact that a Lagrangian section does not globally generate but it does so locally if the cover is chosen appropriately. For global generation one needs $k$ distinct Lagrangian sections.
\end{remark}

 Theorem \ref{thm-local-comp} for $k=0$ is not too difficult to prove. Let us note that in this case it is possible to compute the relative symplectic cohomology over all admissible convex polygons, see Theorem \ref{thm-full-BV-comp}. 

\begin{remark}
This brings us very close to computing the relative symplectic cohomology over all admissible polygons using the spectral sequence displayed in \eqref{eq-spectral}.
\end{remark}

We briefly explain how Theorem \ref{thm-local-comp} is proven in the case $k=1$, which is more difficult. Let us focus on computing the isomorphism for sections over an admissible convex polygon containing the node. Let us first note  that both sides of \eqref{eqthm-local-comp} satisfy a Hartogs property  (Propositions \ref{prop-Hartogs-A} and \ref{prop-Hartogs-B}) which reduces the problem to comparing the two sheaves on an annular region surrounding the nodal point. Furthermore, using the main results of  \cite {locality},  properties of the map $p_1$, and the $k=0$ case, we know we have an isomorphism of sheaves between the restrictions of both $\cF_1$ and $(p_1)_*\cO_{Y_1^{an}}$ to the complement of each eigenray in $B_1$ and the restrictions of $p_*\cO_{(\Lambda^*)^2}:=(p_0)_*\cO_{Y_0^{an}}$.  

In particular, letting $\cG$ denote ether of the sheaf $\cF_1$ or the sheaf 
$(p_1)_*\cO_{Y_1^{an}}$ and letting $U\subset \bR^2$ denote either the upper or the lower open half plane, we get two different identifications $\cG|_U\simeq p_*\cO_{(\Lambda^*)^2}|_U$. This reduces the claim to a comparison of the transition map between these identifications. We refer to these as the \emph{wall crossing maps}. We then show the wall crossing map on the A-side is uniquely determined by the following considerations 
\begin{itemize}
\item The extra grading by $H_1(M_1;\mathbb{Z})\simeq \mathbb{Z}$.
\item Wall crossing preserves the logarithmic volume form on $(\Lambda^*)^2$. This is a consequence among other things of the fact that the wall crossing map on relative $SH^*$ is a map of BV algebras. See Proposition \ref{prop-WC-BV}.
\item
The wall crossing map preserves norms  and it is identity in the zeroth order. See Lemma \ref{lmWallCrossing}. 
\item 
The image of the restriction map in $\cF_1$ from an admissible convex polygon containing the node to an admissible polygon intersecting at most one of the eigenrays contains sufficiently many elements.  See Proposition \ref{prop-max-action}.

\item
 We can therefore relate upper and lower wall-crossing maps by chasing the images of these elements around the node, see Proposition \ref{thm-monodromy}.
\end{itemize}
We find that this wall crossing map matches the one on the B-side, which is easy to compute, see Proposition \ref{prop-KS-cover}. 
\subsection{Relationship with other work involving mirror symmetry and some explicit conjectures}
\subsubsection{Relation with Kontsevich-Soibelman's work}
First, let us relate our work with Kontsevich-Soibelman's work from \cite{koso} to orient the reader who is familiar with this influential paper. We will not give proofs but they are immediate from the construction. Consider the setup and the construction in Theorem \ref{thm-mirror-cons} and assume that the multiplicities of the nodes are all $1$. The map $p_\cR: \mathcal{Y}_{\mathcal{R}}\to B_{\mathcal{R}}$ is a singular affinoid torus fibration in the sense of \cite{koso}. The set of smooth points is precisely $B_{\mathcal{R}}^{reg}$ and the induced integral affine structure  agrees with the given one. Moreover, $(Y_\cR,p_\cR,\Omega_\cR)$ solves the lifting problem for the $\Lambda$-affine structure on $B^{reg}_\cR$ given by the group homomorphism $\mathbb{R}\to \Lambda^*$, $x\mapsto T^x$.

Assume that in Kontsevich-Soibelman's framework we choose the ``lines" from Section 9.1 of \cite{koso} to be as shown in Figure \ref{fig-koso}. These satisfy the Axioms listed in \cite[Section 9.2]{koso}, notably Axiom 2. In particular, there is no scattering (or composed lines in the language of Kontsevich-Soibelman). That our solution to the lifting problem is the same as theirs follows from our computation of wall-crossing (see \cite{koso}'s second highlighted equation on page 55 and Section 11.4). In fact taking Kontsevich-Soibelman's construction as the definition of $p_\cR: \mathcal{Y}_{\mathcal{R}}\to B_{\mathcal{R}}$, one could think of Theorem \ref{thm-mirror-cons} as a computation, or in other words an actual symplectic meaning for Kontsevich-Soibelman construction (see Seidel's \cite[Section 3]{seidel} for the impetus of this idea). 

\begin{figure}
\includegraphics[width=0.6\textwidth]{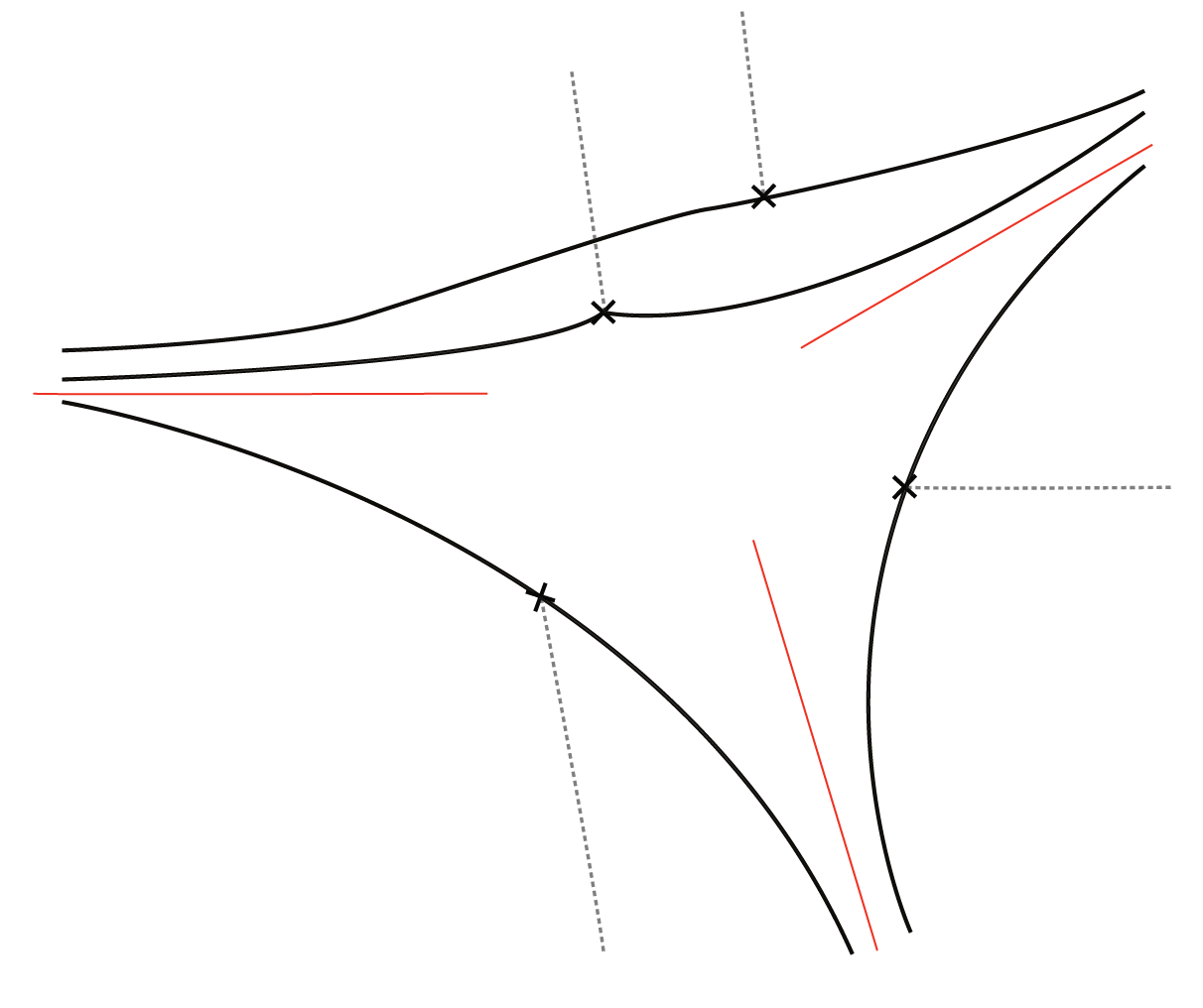}
\caption{The lines are the rays in solid black. The slopes of red rays and the eigenrays are cyclically ordered with respect to the cyclic order at infinity, where the slopes of a consecutive red ray and eigenray are not equal.}
\label{fig-koso}
\end{figure}
\subsubsection{Relation with Family Floer theory}\label{sss-ff}

Let us also make a comparison with Family Floer theory. Starting with $\pi_\cR:M_{\mathcal{R}}\to B_{\mathcal{R}}$, Hang Yuan constructs a rigid analytic space $\mathcal{Z}$ with an analytic torus fibration $\mathcal{Z}\to B_{\mathcal{R}}^{reg}$ in \cite[Theorem 1.3]{yuan}.

\begin{conjecture}
$\mathcal{Z}$ is isomorphic as a rigid analytic space to $p_\mathcal{R}^{-1}(B_{\mathcal{R}}^{reg})$ in a fiber preserving way.
\end{conjecture}

A conceptual proof of this conjecture would involve setting up closed-open string maps that realize the idea that sections of the sheaf $\mathcal{F}_\cR$ give rise to functions on the Family Floer mirror.

\subsubsection{Further expectations for closed string invariants}

Let's continue with closed string consequences. The first step to this is to extend the local computations to show that there is an isomorphism of algebras \begin{equation}\label{eq-der} SH^*_{M_{\mathcal{R}}}(\pi^{-1}(P);\Lambda)\to \Lambda^{*}Der(\mathcal{F}_\cR(P),\mathcal{F}_\cR(P)),\end{equation} whenever $P$ is a small admissible polygon.  The right hand side means the exterior algebra over $\mathcal{F}_\cR(P)$ of the module of algebra derivations of $\mathcal{F}_\cR(P)$. We proved this when $P$ does not contain a node in this paper. We are also confident that this is true when $P$ does not contain a node with multiplicity more than $1$. Otherwise we do not know the answer.

\begin{remark}
An element $a\in SH^1_M(K)$ defines the derivation $[a,\cdot]$ on $SH^0_M(K)$, so at least in degree $1$ the map can be defined naturally.
\end{remark}

Recall that given compact subsets $P_1,\ldots ,P_n$ of $B_\cR$, we have a convergent spectral sequence:\begin{align}\label{eq-spectral} 
\bigoplus_{0\neq I\subset [n], |I|=p } SH_M^q\left(\bigcap_{i\in I} \pi^{-1}
(P_i)\right) \Rightarrow SH_M^{p+q}\left(\bigcup_{i=1}^n \pi^{-1}(P_i)\right).
\end{align}

The differentials in the first page of the spectral sequence are precisely the Cech differentials. We believe that it is reasonable to conjecture that the spectral sequence degenerates after that page when $P_i$ are small admissible polygons. This would allow us to conclude the first part of the following.

\begin{conjecture}If the multiplicites of nodes in $\cR$ are all $1$, there exists a ring isomorphism $$SH_{M_{\mathcal{R}}}^k(M_{\mathcal{R}},\Lambda)\to\bigoplus_{p+q=k} H^p(\mathcal{Y}_{\mathcal{R}},\Lambda^qT\mathcal{Y}_{\mathcal{R}}),$$ where the right hand side is the Cech cohomology with respect to the defining affinoid cover. 

Using $\Omega_\cR$ we equip the polyvector fields with a differential by transporting the deRham differential. 
This differential is compatible with the BV operator.\end{conjecture}
%
%
%
%
%
\begin{remark}\label{rem-local-gen-closed} By construction $\pi_\mathcal{R}$ has a Lagrangian section (see the two paragraphs before Remark 7.5 in \cite{locality}), in fact one that is the fixed point set of an anti-symplectic involution \cite[Proposition 3.3]{Solomon}. Let us call its image $L$. Assuming that all multiplicities of nodes are $1$ and $P\subset B$ is a small admissible polygon, we conjecture that the local generation criterion is satisfied at $K=\pi^{-1}(P),$ namely the open-closed map $$HH_2(CF_{M_\mathcal{R}}^*(K;L;\Lambda))\to SH_{M_\mathcal{R}}^0(K;\Lambda)$$ hits the unit. For a brief introduction to the open string invariant used on the left hand side see \cite[Section 2.3]{tonkonog}.

Another conjecture is that our locality results (based on the existence of complete embeddings) hold for $HF_{M_\mathcal{R}}^*(\pi^{-1}(P);L;\Lambda)$ as well and imply that it is supported in degree $0$. Then, we would obtain that the closed open map is an isomorphism $$SH_{M_\mathcal{R}}^*(\pi^{-1}(P);\Lambda)\to HH^*(HF_{M_\mathcal{R}}^0(\pi^{-1}(P);L;\Lambda)).$$ Using commutativity of $HF_{M_\mathcal{R}}^0(L;\pi^{-1}(P);\Lambda)$ and HKR, we get the isomorphism from \eqref{eq-der}.

We expect that the mentioned degeneration at the $E_2$ page of the spectral sequence in \eqref{eq-spectral} would follow from the construction of a chain level closed-open map (which is a quasi-isomorphism by the local generation criterion) and the chain level lift of HKR. Assuming that the closed open map and HKR respect the $BV_{\infty}$ structures in a suitable way, the comparison between the $BV$-operator and the deRham differential is also a direct consequence of local generation and HKR.\end{remark}
\subsubsection{Homological mirror symmetry}
In terms of open string mirror symmetry we expect

\begin{conjecture}\label{con-open} There exists a canonical $A_\infty$-functor $$Fuk(M_\mathcal{R})\to Coh_{dg}(\mathcal{Y_\cR}).$$ If we assume that multiplicities of nodes are all $1$, it is a quasi-equivalence.\end{conjecture}

Let us also briefly mention the approach to Conjecture \ref{con-open}. Here we are omitting any discussion of the serious question of what we mean by $Fuk(M_\mathcal{R})$. Let $L$ be as in Remark \ref{rem-local-gen-closed} and assume that the mirror $\mathcal{Y}_\cR$ is equivalently constructed using the relative Lagrangian Floer homology of $L$ as suggested by the discussion in Remark \ref{rem-local-gen-closed}.  To construct the functor we send each Lagrangian $L'$ to the complex of $\mathcal{O}$-modules $CF_{M_\mathcal{R}}^*(\pi_\mathcal{R}^{-1}(P); L,L';\Lambda)$. Coherence will follow from the local generation property of $L$, using the unitality of restriction maps. On morphisms, we use the product $$CF_{M_\mathcal{R}}^*(K;L,L_1;\Lambda)\otimes CF_{M_\mathcal{R}}^*(K;L_1,L_2;\Lambda)\to CF_{M_\mathcal{R}}^*(K;L,L_2;\Lambda)$$and so on for higher terms of the morphisms. Finally, a version of the local-to-global property will allow us to deduce the full and faithfulness of the functor from the local generation of $L$. 

Essential surjectivity is less clear and here we believe that it would follow from specific results in this case inspired by Hacking-Keating homological mirror symmetry \cite{hacking}, which we now turn to -- in general we don't know whether we should expect essential surjectivity.

\subsubsection{Relation with Hacking-Keating mirror symmetry}

It is possible to understand the mirror in Theorem \ref{thm-mirror-cons} concretely in general, not just in the local version. Namely, we strongly believe that it is the (rigid analytification of the) interior of a log CY surface defined over the Novikov field. 

Let us be more precise. We call an eigenray diagram \emph{exact} if the lines containing all of the rays pass through the same point. If $\mathcal{R}$ is exact, then $M_{\mathcal{R}}$ is a complete finite type Liouville manifold \cite[Remark 1.10]{locality}.

One can isotope any eigenray diagram $\cR$ through eigenray diagrams with the same number of nodes into an exact eigenray diagram $\cR'$. To each node $n$ of $\cR'$ we can associate a real number $$x_n:=m_n (\Delta_n \times v_n)$$ where $m_n$ is the multiplicity, $\Delta_n$ is the vector from the new location of the node to the old one and $v_n$ is the primitive integral tangent vector in the direction of the ray emanating from $n$. 

We now think of the rays of $\cR'$ as forming a fan defining a toric variety $V_\cR$ over $\Lambda$. By construction we  can choose an identification of each one dimensional toric orbit with $\Lambda^*$ (uniquely up to an overall action by element of the two dimensional torus). We then do one Kaliman modification for each node on the corresponding irreducible toric divisor of $V_\cR$ at the point $T^{x_n}$. We denote by $Y_\cR$ the resulting non-proper scheme over $\Lambda$. 

\begin{conjecture}
$\mathcal{Y}_\cR$ is the analytification of $Y_\cR$.
\end{conjecture}
 In Section 7.7 of \cite{locality} (see Remark 7.38 and the discussion preceeding it), we gave another description of $Y_\cR$ by gluing certain affine schemes, which is the one that would actually be useful in the proof. The proof also requires a Zariski analogue of \cite{varolgunes} to justify this gluing construction, which is actually straightforward.
%
%

In the case where $\mathcal{R}$ is exact, our $M_\cR$ corresponds to an $M$ from \cite[Theorem 1.1]{hacking}. If we in addition use $\mathbb{C}$ as our base field, our $Y_\cR$ is the base change to the Novikov field of the corresponding $Y\setminus D$ from the same statement. We refer the reader to the cited paper's introduction for many interesting special cases.  When $\mathcal{R}$ is not exact, which we rigorously cover in this paper, $M_\cR$ is a non-exact deformation of one of the Liouville manifolds covered in Hacking-Keating's theorem.

\subsection{Outline} In Section \ref{s-integral-affine}, we give a brief account of eigenray diagrams and nodal integral affine manifolds. We then introduce the $G$-topology we will be using and the notion of small admissible polygons. Section \ref{s-local-analytic} is where we discuss the $B$-side local models $Y_k^{an}$ and their extra structures. In Section \ref{s-rel-sh}, we go over some generalities regarding relative symplectic cohomology and its relation to symplectic cohomology of Liouville manifolds. In Section \ref{s-noray}, we compute the relative symplectic cohomology presheaf in the base of $T^*T^n\to \mathbb{R}^n$ in all degrees relying on Viterbo's theorem. We then go on to the proof of Theorem \ref{thm-local-comp}, which takes up the next four sections. In Section \ref{s-6}, we construct   sheaf isomorphisms between $\cO_k$ and $\cF_k$ in the complement of each of the eigenrays. 
In the next section, Section \ref{s-7}, we prove a slight extension of our locality for complete embeddings result showing that the locality isomorphisms for different complete embeddings are the same in the zeroth order. Section \ref{s-8} is where we finish the computation of $A$-side wall-crossing. In the final section of local computations, Section \ref{s-hartogs}, we extend the isomorphism of $\cO_k$ and $\cF_k$ from the punctured plane to the entire plane by proving an $A$-side Hartogs extension theorem for relative symplectic cohomology in degree $0$. In Section \ref{s-cons}, we first prove Theorem \ref{thm-four-prop} relying on the locality for complete embeddings theorem and the local results. Then we go on to prove our main result Theorem \ref{thm-mirror-cons} using the properties listed in Theorem \ref{thm-four-prop}. We end our paper with a short section that discusses some higher dimensional situations to which our methods extend straightforwardly.

\subsection*{Acknowledgements} 
We thank the anonymous referees whose numerous comments have helped us greatly improve the paper. 
Y.G. was supported by the ISF (grant no. 2445/20). 
 U.V. was supported by the T\"{U}B\.{I}TAK 2236 (CoCirc2) programme with a grant numbered 121C034.

\section{Some integral affine geometry}\label{s-integral-affine}
\subsection{Local models}\label{ss-integral-affine-local}
Let us denote by $B_0^n$ the Euclidean space $\mathbb{R}^n$ as an integral affine manifold with its standard integral affine structure. We abbreviate $B_0:=B_0^2$.

For $k> 0$ and a primitive vector $e\in\mathbb{Z}^2$, consider the linear map $A_{k,e}:\mathbb{R}^2\to\mathbb{R}^2$ defined by \begin{equation} v\mapsto v-k\cdot det(e,v)\cdot e.\end{equation} We define the integral affine manifold $B_{k,e}^{reg}$ by gluing $$B_{e}^{+}:= \mathbb{R}^2\setminus \{c\cdot e\mid c\geq 0 \}\subset B_0\text{ and }B_{e}^{-}:= \mathbb{R}^2\setminus \{c\cdot e\mid c\leq 0 \}\subset B_0$$ with the gluing map $\phi: B_{e}^{+}\cap B_{e}^{-}\to B_{e}^{+}\cap B_{e}^{-} $ from $B_{e}^{-}$ to $B_{e}^{+}$ defined by $$\phi(w)=\begin{cases}w,& det(e,w)>0  \\ A_{k,e}^{-1}(w),& det(e,w)<0.\end{cases}$$ Notice that we can replace $B_{e}^{-}$ with an arbitrary neighborhood of $\{c\cdot e\mid c> 0 \}$ inside $\mathbb{R}^2$ disjoint from $\{c\cdot e\mid c\leq 0 \}.$ For convenience, when $e=e_1$ is the first standard basis vector we simply omit the subscript $e$'s and write $B_k^{reg}$. Clearly $B_{k,e}^{reg}$ is integral affine isomorphic to $B_k^{reg}$ but this involves choices and we prefer to not identify them.

Note that $B_{k,e}^{reg}$ can be embedded into $\mathbb{R}^2$ as a topological manifold by the map that sends $B_{e}^{+}$ to $B_{e}^{+}$ with the identity map. We define $B_{k,e}$ as the topological manifold $\mathbb{R}^2$ along with the integral affine structure on $\mathbb{R}^2\setminus \{0\}$ induced from this embedding. Let $B_{0,e}^{reg}:=B_{0,e}:=B_0$. Let us denote the canonical embeddings $B_{e}^{\pm}\to B_{k,e}$ by $E^\pm_k$.  When $e=e_1$, we omit it from notation.

Inside $B_{k,e}$, we define \begin{itemize}
\item a \emph{rational line} to be the image of a complete affine geodesic $\mathbb{R}\to B_{k,e}^{reg}$ with a rational slope,
\item a \emph{rational halfspace} to be the closure of a connected component of $B_{k,e}\setminus l$ with $l$ a rational line,
\item a \emph{model polygon} to be a compact subset that is a finite intersection of rational halfspaces and that contains the node if $k>0.$
\end{itemize} 

We can similarly define model polytopes on $B_0^n$ as well. We omit spelling this out. This is the same notion as $\mathbb{R}$-rational convex polyhedra from \cite[Convention 1.3]{kapranov}. 

\subsection{Eigenray diagrams}\label{ss-eigenray}

A \emph{ray} in $\mathbb{R}^2$ is the image of a map of the form $[0,\infty)\to \mathbb{R}^2$, $t\mapsto x+vt$, where $x,v\in \mathbb{R}^2$. Let us define an \textit{eigenray diagram} $\mathcal{R}$ to consist of the following data:

\begin{enumerate}
    \item A finite set of pairwise disjoint rays $l_i$, $i=1,\ldots, k,$   in $\bR^2$ with rational slopes.
        \item A finite set of points on each ray including the starting point. Let us call the set of these points $N_{\mathcal{R}}$. 
    \item A map $m_{\mathcal{R}}: N_{\mathcal{R}}\to \mathbb{Z}_{\geq 1}$.
\end{enumerate}
For any $n\in N_\cR,$ we denote by $l^n$ the subray of the ray that contains $n$ starting at $n.$

\begin{definition}\label{def-nodal-int-aff}
Let $B$ be a two dimensional topological manifold with a finite number of special points $N\subset B$ and let $B^{reg}:=B-N$ be equipped with an integral affine structure. We call $B$ a \emph{nodal integral affine manifold} if each point $n$ in $N$ admits a neighborhood $U$ such that $U\setminus\{n\}$ is integral affine isomorphic to a punctured neighborhood of the origin in $B_k^{reg}$ with $k\geq 1$. The elements of $N$ are called \emph{nodes}, and the ones of $B^{reg}$  \emph{regular points}. We call the positive integer $k$  the \emph{multiplicity} of $n$. 
\end{definition}

By definition $B_k$ is a nodal integral affine manifold.  An \emph{embedding of a nodal integral affine manifold into another} is defined to be a topological embedding which sends nodes to nodes and is integral affine on the regular locus.

\begin{definition}\label{def-eigenray} Let $B$ be a nodal integral affine manifold and $n\in N$ a node. We call the oriented image of a continuous map $\phi: [0,E)\to B$ for some $E\in (0,\infty]$ a \emph{local eigenray of $n$} if \begin{itemize}
\item $\phi(0)=n$,
\item $\phi$ restricted to each connected component of $\phi^{-1}(B^{reg})$ is a geodesic,
\item For any $a\in [0,E)$ such that $\phi(a)\in N$, there exists an open neighborhood $U$ of $a$ such that for $t\in U\setminus \{a\}$, $\phi(t)\in B^{reg}$ and the tangent vector $\phi'(t)$ is fixed under the linear holonomy $T_{\phi(t)}B\to T_{\phi(t)}B$ around a loop which bounds a disk $D$ with $D\cap N=\phi(a)$ and $\phi(t)\in \partial D$.
\end{itemize}

A local eigenray that is not contained in any other one is called an eigenray. 
\end{definition}

Each eigenray diagram $\cR$ gives rise to a nodal integral affine manifold $B_\cR$ with a preferred homeomorphism $\psi_\cR: \mathbb{R}^2\to B_{\mathcal{R}} $ such that
\begin{itemize}
\item $\psi_\cR(N_{\mathcal{R}})$ is the set of nodes of $B_{\mathcal{R}}$
  \item $\psi_\cR$ restricted to the complement of the rays in $\cR$ is an integral affine isomorphism onto its image.
    \item The multiplicity of a node $\psi_\cR(n)$ is $m_{\mathcal{R}}(n)$.
    \item For any $n\in N_\cR$, $\psi_\cR(l^n)$ is an eigenray of $\psi_\cR(n)$.
\end{itemize}  The construction involves  starting with the standard integral affine $\mathbb{R}^2$ and doing a modification similar to the one in the definition of $B_{k,e}$ for each element $n$ of $N_\cR$ by removing $l_n$ and then re-gluing an arbitrarily small product neighborhood of $l_n$. The order in which these surgeries are made does not change the resulting nodal integral affine manifold. 

We refer the reader to  \cite[\S7.2]{locality} for more details. In particular one finds there a detailed description the two operations of \emph{nodal slide} and \emph{branch move} that are mentioned in Theorem \ref{thm-mirror-cons}. 

\begin{remark}
    We think of the domain of $\psi_\cR$ as the place where the combinatorial data of the rays, nodes and multiplicities of an eigenray diagram lives. This is a special case of the base diagrams of \cite{symington}. The target of $\psi_\cR$ on the other hand is thought of as a geometric object, a nodal integral affine manifold. 
\end{remark}

We will now define a $G$-topology on $B_\cR$. For the purposes of this paper a $G$-topology on a set $X$ is a set
$\mathfrak{U}$ of subsets $U \subset X$ and a set of set-theoretic
coverings $Cov(\mathfrak{U})$ of each $U \in \mathfrak{U}$ by members of $\mathfrak{U}$ contained in $U$ such that 
\begin{enumerate}
\item $\{U\} \in Cov(U)$ for all $U \in \mathfrak{U},$
\item if $U, V\in\mathfrak{U}$, then $U\cap V\in\mathfrak{U}$,
\item if $\{U_i\} \in Cov(U)$ and $V \subset  U$ with $V \in \mathfrak{U}$, then $\{V \cap U_i\}\in Cov(V),$ and
\item if $\{V_{ij}\}_{j\in J_i} \in Cov(U_i)$ for $\{U_i\} \in Cov(U)$, then
$\{V_{ij}\}_{i,j}\in Cov(U)$.
\end{enumerate}We note that this is the same definition as \cite[Definition 1 of Section 9.1]{bgr}. 

\begin{remark}\label{rem-Hausdorff-G-top}
    When $X$ is a Hausdorff topological space, taking compact subsets as the admissible opens and finite covers of compact subsets by compact subsets as the admissible covers, we obtain a $G$-topology. Note that under the Hausdorff assumption intersection of two compact subsets is compact.
\end{remark}

An \emph{admissible convex polygon} inside $B_\cR$ is a subset that is the image of a model polygon $P\subset B_k$ for some $k\geq 0$ under an embedding $U\to B_\cR$ of nodal integral affine manifolds with $U$ an open neighborhood of $P$ inside $B_k$. 

It immediately follows that inside $B_0,$ the class of model polygons and admissible convex polygons coincide.
    In the case $B_\cR=B_k$ with $k\geq 1,$ an admissible convex polygon is a model polygon if and only if it contains the node.


\begin{proposition}\label{prop-convex-closed-under-intersection}
    Let $\mathcal{R}$ be an eigenray diagram.
    The intersection of two admissible convex polygons inside $B_\mathcal{R}$ is a finite (possibly empty) union of admissible convex polygons.
\end{proposition}
\begin{proof}
    For $i=1,2$, let $P_i$ be our admissible convex polygons. We fix $K_i\subset U_i\subset B_{k_i}$ and embedding $\psi_i: U_i\to B_\cR$ such that $\psi_i(K_i)=P_i$ as in the definition above. 
    
    The case where $\cR$ has no rays is trivial. Now assume that $N_\mathcal{R}$ has only one element, i.e. $B_\cR=B_k$. If both $P_1$ and $P_2$ are model polygons in $B_k$, the claim is again obvious. Assume that $P_1$ is not a model polygon in $B_k$, which implies that it does not contain the node. We can find a finite open cover of $B_k$ by an open neighborhood of the node that is disjoint from $P_1$ and interiors of rational halfspaces that do not contain the node. Then, by subdividing $K_1$ and using a Lebesgue number argument, we can represent $P_1$ as a finite union of admissible convex polygons $P_1=\bigcup Q_j$ each of which is contained in one of the open rational halfspaces $H_j$ inside $B_k$. Since $\bar{H_j}\cap P_2$ is also an admissible convex polygon, by the previous case, $Q_j\cap P_2=Q_j\cap \bar{H_j}\cap P_2$ is an admissible convex polygon. This finishes the proof when $|N_\mathcal{R}|=1$, since  $P_1\cap P_2=\bigcup (Q_j\cap P_2).$

    Next, we consider the case where $\cR$ has one ray. We do an induction on the number of nodes. The base of the induction was just covered. For the induction step, we find a rational line that transversely intersects the open line segment between the starting point of the defining eigenray and and the next node on it. Then, we chop up both $P_1$ and $P_2$ by this line and prove the induction step.
    
    Finally, we are in the general case. Consider the open cover of $B_\cR$ by the complement of the defining eigenrays $W$ and pairwise disjoint neigborhoods of the rays $W_j.$ By subdividing the admissible convex polygons (using the Lebesgue number argument), we can reduce to the case where each admissible convex polygon is contained in one of the members of this open cover. Since $W_j$'s are disjoint, we reduce in fact to the previous case with only one ray. This finishes the proof.
\end{proof}

A finite (possibly empty) union of admissible convex polygons is called an \emph{admissible polygon}. An \emph{admissible covering} is any finite covering of admissible polygons by admissible polygons. We define a $G$-topology on $B_\cR$ whose opens are the admissible polygons and whose allowed covers are the admissible covers. The property (2) for being a $G$-topology follows from Proposition \ref{prop-convex-closed-under-intersection} and the other properties follow easily.

\subsection{Sheaves over $B_\cR$}\label{ss-sheaf-G}
To prove our main theorem turns out to be easier to deal with a more restricted class of polygons. Namely, we fix an eigenray presentation of $B_{\cR}$ and consider only certain polygons which intersect at most one ray. We will prove the following proposition.

\begin{proposition}\label{prpSmall}
    Let $l_1,\ldots,l_n$ be the rays of $\mathcal{R}$. There exists a class of admissible convex polygons, called \emph{small admissible polygons} satisfying the following conditions.
    \begin{enumerate}
    \item\textbf{(Smallness)} a small admissible polygon can intersect $\psi_\cR(l_i)$ for at most one ray $l_i$. 
    \item\textbf{(Intersection)} intersection of two small polygons is a small polygon.
    \item\textbf{(Covering)} every admissible open has a finite cover by small admissible polygons. 
    \end{enumerate}
\end{proposition}

Before proving Proposition \ref{prpSmall} we say how we use it. 
 
\begin{lemma}\label{lmSheafGluing}
Let $\mathcal{G}$ and $\mathcal{G}'$ be two sheaves of $\Lambda$-algebras on the $G$-topology of $B_\cR$. Assume that we are given isomorphisms of algebras $$\mathcal{G}(P)\to\mathcal{G}'(P)$$for every small admissible polygon $P$ and that they are compatible with restriction maps. Then there is a unique isomorphism of sheaves of $\Lambda$-algebras $\mathcal{G}\to \mathcal{G}'$ extending the given isomorphisms. 
\end{lemma}

\begin{proof}
     Consider the $G$-topology on $B_\cR$ that has its admissible opens the small polygons and its covers their finite unions. In the language of Definition 1 of \cite[Section 9.1.2]{bgr} the conditions in Proposition \ref{prpSmall} guarantee that the $G$-topology of admissible polygons in $B_\cR$ is \emph{slightly finer} than the $G$-topology of small polygons.   Namely, for a pair of G-topologies $\mathcal{T}, \mathcal{T'}$ we say $\mathcal{T'}$ is slightly finer than $\cT$ if 
    \begin{itemize}
        \item Every $\cT$-open is $\cT'$-open and each $\cT$-cover is also a $\cT'$-cover.
        \item The $\cT$-opens form a basis for $\cT'$. This means every $\cT'$ open has a $\cT'$-admissible cover by $\cT$-opens. 
        \item Each $\cT'$-covering of a $\cT$-open has a $\cT$-covering which refines it. 
    \end{itemize}
     The claim is then Proposition 1 of \cite[Section 9.2.3]{bgr}.
\end{proof}

\begin{remark}
The notion of a small admissible polygon is an ad-hoc notion that depends on the eigenray diagram presentation of the nodal integral manifold $B_\cR$. Note that a given nodal integral affine manifold can have many different eigenray diagram presentations.
\end{remark}

To prove Proposition \ref{prpSmall}, we turn to specify the notion of small polygons. 
For each $i$ fix a half infinite strip $S_i$ containing $l_i$ and assume the $S_i$ are pairwise disjoint. Let $\tilde{l}_i\subset S_i$ be the ray containing $l_i$ and disconnecting it. Finally, for each node $p_{ij}$ of $l_i$ consider a piecewise affine segment $\sigma_{ij}$ through $p_{ij}$ disconnecting $S_i$. We take the $\sigma_{ij}$ to be pairwise disjoint. See  Figure \ref{fig-slits}.
\begin{figure}
\includegraphics[width=\textwidth, trim = 0 450 0 0 ]{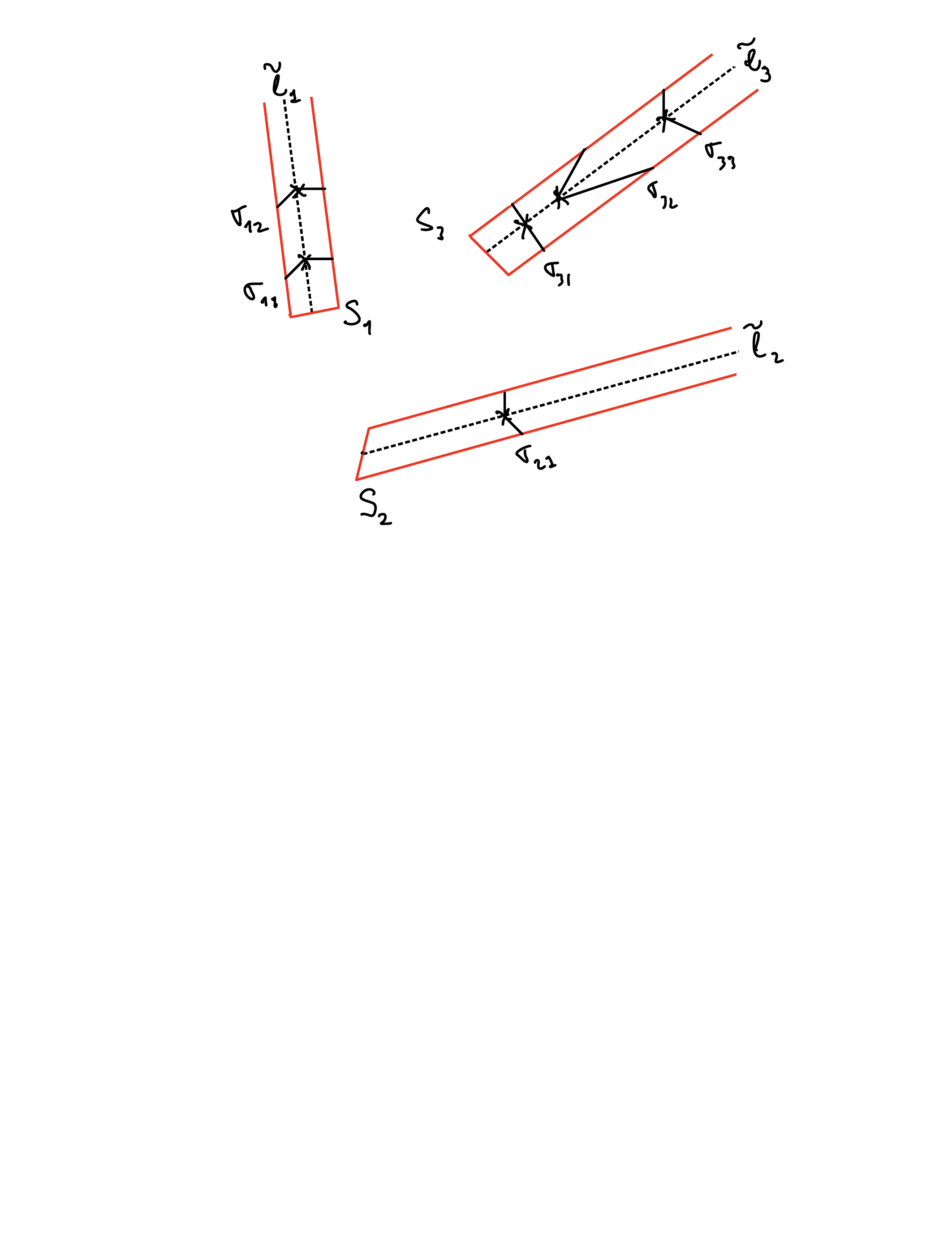}
\caption{An eigeray diagram sith a choice of strips $S_i$ and piecewise affine geodesics $\sigma_{ij}.$}
\label{fig-slits}
\end{figure}
A \emph{small admissible polygon} is an admissible convex polygon which satisfies \emph{one} of the following conditions:
\begin{enumerate}[(A)]
    \item It does not intersect $\tilde{l}_i$ for any $i$, or,
    \item it is contained in $S_i$ for some $i$ and contains a node $p_{ij}$ and does not intersect $\sigma_{ij'}$ for $j'\neq j$, or,
    \item it is contained in $S_i$ for some $i$ and does not intersect $\sigma_{ij}$ for any $j$.
\end{enumerate}

We will need the following lemma. 
\begin{lemma}\label{lem-bk-intersection}
Let $P_1$ and $P_2$ be admissible convex polygons inside $B_k$ with $k>0.$ Assume that $P_1$ contains the node. Then $P_1\cap P_2$ is also an admissible convex polygon.    
\end{lemma}

\begin{proof}
    First recall that if an admissible convex polygon in $B_k$ contains the node, then it must be a model convex polygon. Therefore, if $P_2$ also contains the node, the result immediately follows. If $P_2$ does not contain the node, then it can intersect only one of the eigenrays. It follows that there is a rational half space not containing the node that contains $P_2$ in its interior since inside standard $\mathbb{R}^2$ there is a rational line separating any convex polygon and a point not on the polygon. The intersection of this rational half space with $P_1$ is also an admissible convex polygon and clearly so is its intersection with $P_2.$
\end{proof}

\begin{proof}[Proof of Proposition \ref{prpSmall}]
Smallness is immediate from the construction.

We now verify the intersection property. For any pair $P_1,P_2$ of small admissible polygons, it is clear that $P_1\cap P_2$ satisfies the set theoretic condition defining one of the types. What remains is to verify that $P_1\cap P_2$  is an admissible  convex polygon. This is not automatic in $B_{\cR}$. We spell out how this does follow for the more restricted type of polygons we are considering.  

The intersection of two polygons of type A occurs in a simply connected region of $B^{reg}_{\cR}$ containing them both and is thus convex since the integral affine structure is isomorphic to that of $\bR^2$. The same applies to the intersection of two polygons of type C since the intersection necessarily occurs inside a region bordered by $\sigma_{i,j}$, $\sigma_{i,j+1}$ inside $S_i$ for some $i,j$. The intersection of any polygon $P_1$ with a polygon $P_2$ of type B can be presented, by making $P_1$ smaller without changing the intersection $P_1\cap P_2$, as an intersection of two polygons inside a region integral affine isomorphic to a subset of $B_k$. The convexity now follows from  Lemma \ref{lem-bk-intersection}. Finally, if $P_1$ is of type A and $P_2$ of type C meeting $\tilde{l}_i$ then $P_1$ meets only one component of  $P_2\setminus\tilde{l}_i$, which is then convex in the usual sense.  This intersection again occurs in a region integral affine isomorphic to $\bR^2$. 


To verify the covering property note that given any admissible convex polygon we can cut it into small admissible polygons by considering a subdivision by the boundaries of the $S_i$, and further partition of polygons contained in $S_i$ into polygons of type B and C. 

\end{proof}

\section{Local models from non-archimedean geometry}\label{s-local-analytic}

 We will try to keep the required background in non-archimedean geometry to a minimum. We refer the reader to Sections 2 and 3 of Bosch's \cite{bosch} for this minimum background, making further references when necessary. Let us note here for convenience that a subset $S\subset M(A)$ of the affinoid domain corresponding to an affinoid algebra $A$ is called a \emph{Weierstrass subdomain} if it is given by inequalities $$|f_1(x)|\leq 1, \ldots |f_m(x)|\leq 1$$ for some $f_1,\ldots, f_m\in A$ and a \emph{Laurent subdomain} if it is given by inequalities $$|f_1(x)|\leq 1, \ldots |f_m(x)|\leq 1, |g_1(x)|\geq 1, \ldots |g_l(x)|\geq 1$$ for some $f_1,\ldots, f_m, g_1,\ldots ,g_l\in A.$ 

\subsection{Analytifications of affine varieties}\label{s-analytification}
Let $\Bbbk$ be an algebraically closed non-archimedean field whose valuation map surjects onto $\mathbb{R}\cup\{\infty\}$. For $\alpha\in \mathbb{Z}_{\geq 0}^n$, we denote by $|\alpha|$ the sum of its components. For any $r\geq 0$, let us define the non-archimedean Banach $\Bbbk$-algebra $$T_n^{(r)}:=\left\{\sum_{\alpha\in \mathbb{Z}_{\geq 0}^n}a_\alpha x^\alpha\mid a_\alpha\in\Bbbk \text{ such that } val(a_\alpha)-r|\alpha|\to \infty \text{ as } |\alpha|\to \infty\right\}.$$  Picking an arbitrary  $c\in \Bbbk$ with valuation $-r$ and writing down formally the definition of the Tate algebra (synonymous to restricted power series) in the variables $c^{-1}x_1, \ldots, c^{-1}x_n$ suggests the notation $$T_n^{(r)}=\Bbbk\langle c^{-1}x_1, \ldots, c^{-1}x_n\rangle.$$ 

We usually think of the affinoid domain $M(T_n^{(r)})$ as the $n$-dimensional non-archimedean ball of radius $e^r.$ The points of $M(T_n^{(r)})$ are identified with $n$-tuples $(a_1,\ldots,a_n)\in\Bbbk^n$ such that $|a_i|\leq e^r.$ For any $r<r',$ we have the canonical inclusion of a Weierstrass subdomain $M(T_n^{(r)})\subset M(T_n^{(r')})$.

There is an important functorial operation in non-archimedean geometry called \emph{rigid analytification} that takes as input an algebraic variety over $\Bbbk$ and outputs a rigid analytic space over $\Bbbk$ (see \cite[Section 5.4]{bosch}). In this paper, we will only use the local version where the input is an affine variety. For example the rigid analytification of the affine space $\mathbb{A}^n_\Bbbk$ is defined by taking the union of all $n$-dimensional non-archimedean balls using their canonical inclusions. This union is canonically identified as a set with $\Bbbk^n.$ Using the formalism of $G$-ringed spaces, we equip it with a $G$-topology and a structure sheaf so that it is a rigid analytic space. 

More generally, let $Y:=Spec(A)$ with $A:=\Bbbk[x_1,\ldots,x_n]/(f_1,\ldots, f_m)$ be our affine variety. Its rigid analytification $Y^{an}$ is a rigid analytic space constructed so that there is a canonical bijection of sets $M(A)\to Y^{an}$ and there is an exhaustion by nested inclusions of affinoid domains $$\ldots\subset M\left(\Bbbk\langle c^{-1}x_1, \ldots, c^{-1}x_n\rangle/(f_1,\ldots, f_m)\right)\subset\ldots \subset Y^{an}$$ with $val(c)=-r$ and $r\to \infty.$ We note that the $r$-sublevels of this exhaustion correspond precisely to the intersections of $M(A)\subset \Bbbk^n$ with the non-archimedean ball of radius $e^r.$ 

\begin{remark}
    Even though the construction of $Y^{an}$ that we alluded to in our incomplete summary depends on the embedding $Y\subset \mathbb{A}^n_\Bbbk$ one can prove that different affine embeddings lead to isomorphic results as shown in Definition and Proposition 3 of \cite[Section 5.4]{bosch}.
\end{remark}

\subsection{Analytification of $(\Bbbk^*)^n$ and the tropicalization map}\label{ss-ana-no-ray}

Let us define $\mathcal{Y}_0^n$ to be the rigid analytification of the affine variety $Spec(\Bbbk[(\mathbb{Z}^n)^{\vee}]) $. As a set $\mathcal{Y}_0^n$ is canonically identified with the closed points of $Spec(\Bbbk[(\mathbb{Z}^n)^{\vee}]) $ (see \cite[Proposition 4 of Section 5.4]{bosch}), which we can in turn canonically identify with $(\Bbbk^*)^n$ using the standard basis of $\mathbb{Z}^n$ (and its dual).  Define the \emph{tropicalization} map $p_0:\mathcal{Y}_0^n\to B_0^n=\mathbb{R}^n$ via $$(y_1,\ldots ,y_n)\mapsto (val(y_1),\ldots,val(y_n)).$$ Let us denote by $\cO_{\mathcal{Y}_0^n}$ the structure sheaf of $\mathcal{Y}_0^n.$

\begin{proposition}\label{prpTropTorus}Let $P\subset B_0^n$ be an admissible convex polytope. Then
\begin{enumerate}
\item $p_0^{-1}(P)\subset \mathcal{Y}_0^n$ is an affinoid domain\footnote{This is short for $p_0^{-1}(P)\subset \mathcal{Y}_0^n$ is an admissible open, $\mathcal{O}_{\mathcal{Y}_0^n}(p_0^{-1}(P))$ is an affinoid algebra and with the induced rigid analytic structure $p_0^{-1}(P)$ is isomorphic to the affinoid domain $M(\mathcal{O}_{\mathcal{Y}_0^n}(p_0^{-1}(P)))$.}.
\item $\mathcal{O}_{\mathcal{Y}_0^n}(p_0^{-1}(P))$ is the completion of $\Bbbk[(\mathbb{Z}^n)^{\vee}]$ with respect to the valuation \begin{align}val_P(f= \sum a_jx^{j}):= \inf_{m\in P}\inf_j \left(val(a_j)+ j(m)\right).\end{align} 
\item If $Q\subset P$ is also an admissible convex polytope, then $$p_0^{-1}(Q)\subset p_0^{-1}(P)$$ is a Weierstrass (and in particular affinoid) subdomain.
\end{enumerate}
\end{proposition}

\begin{proof}
These statements are well-known, see e.g., \cite[Section 3]{kapranov} and \cite[Section 4]{gubler}. 
\end{proof}

Let us explain part (2) a little more concretely. We introduce the vector space of formal Laurent power series $$\Bbbk[[(\mathbb{Z}^n)^{\vee}]].$$ To be explicit, these are collections of coefficients $a_j$ indexed by $j\in(\mathbb{Z}^n)^\vee$ which we write as $\sum a_jx^{j}$. Note that we do not have a well defined multiplication operation here. We say that a formal sum $\sum a_jx^{j}$ \emph{converges} at a point $m\in B_0^n$, if for every real number $R$, there are only finitely many $j\in(\mathbb{Z}^n)^\vee$ such that $val(a_j)+j(m))<R$. 

To any admissible convex polytope $P\subset B_0^n$ and $f\in \mathcal{O}_{\mathcal{Y}_0^n}(p_0^{-1}(P))$, we can canonically\footnote{This is canonical only because we have defined $\mathcal{Y}_0^n$ as the analytification of $Spec(\Bbbk[N])$ and specified a basis for the lattice $N$. As a rigid analytic space $\mathcal{Y}_0^n$ can be expressed in this form in different ways which lead to different formal expressions. One should think of these formal expressions as analogous to Taylor expansions with respect to fixed set of coordinate functions.} associate a formal Laurent power series, in other words we have a canonical injection $$\mathcal{O}_{\mathcal{Y}_0^n}(p_0^{-1}(P))\to \Bbbk[[(\mathbb{Z}^n)^{\vee}]].$$ The image of this map is precisely the formal Laurent power series which converge on every point $m\in P$ in the sense spelled out in the previous paragraph. From now on we call this image the \emph{formal expression of an element} of $\mathcal{O}(p_0^{-1}(P))$

Let us define a sheaf $\mathcal{O}_0$ on the $G$-topology of $B_0^n$ by first setting $$\mathcal{O}_0(P):=\mathcal{O}_{\mathcal{Y}_0^n}(p_0^{-1}(P))$$ on an admissible convex polytope $P$ and then defining $\mathcal{O}_0(\cup P_i)$ as the equalizer

\begin{equation}  \label{sites-equation-sheaf-condition} \xymatrix{ \mathcal{O}_0(\cup P_i) \ar[r] &  \prod^{k}_{i=1} \mathcal{O}_0(P_i) \ar@<1ex>[r] \ar@<-1ex>[r] &  \prod^{k}_{i,j=1} \mathcal{O}_0(P_i\cap P_j) } \end{equation} for admissible convex polytopes $P_1,\ldots ,P_k$. Tate's acyclicity theorem guarantees that this indeed defines the desired sheaf $\mathcal{O}_0$. We of course have $(p_0)_*\mathcal{O}=\mathcal{O}_0$ but we chose to be more concrete for once.

\begin{proposition}\label{prop-basic-hartogs}
Let $P_1,\ldots ,P_k$ be admissible convex polytopes in $B_0^n$. 
Assume that $\cup P_i$ is connected and the convex hull $P$ of $\cup P_i$ is an admissible convex polytope, then the canonical map $$ \mathcal{O}_0(P)\to \mathcal{O}_0(\cup P_i)$$ is an isomorphism.
\end{proposition}

\begin{proof}
We identify each function with its formal expression. Because of the connectedness assumption an element of $\mathcal{O}_0(\cup P_i)$ is simply a single element of $\Bbbk[[(\mathbb{Z}^n)^{\vee}]]$ which converges at all points of $\cup P_i$. Now suppose a formal expression $f=\sum a_jx^j$ converges at a pair of points $m_0,m_1\in B_0^n$. For any  $t\in[0,1]$ and for any real $R$ the set of $j$ so that $val(a_j)+j(tm_1+(1-t)m_2)<R$ is contained in the union of the corresponding sets for  $val(a_j)+j(m_1)$ and $ val(a_j)+j(m_2)$. In particular, convergence of $f$ at $m_1,m_2$ implies convergence of $f$ at $tm_1+(1-t)m_2$ for all $t\in [0,1]$. This finishes the proof.
\end{proof}
\begin{remark}
This proposition is related to results from the theory of several complex variables which relate the holomorphic convex hulls with usual convex hulls. Since this is a much simpler result, explaining the connection seems pointless.
\end{remark}

Let $H$ be a co-oriented rational hyperplane in $B_0^n$. Denote by $l_H\in (\mathbb{Z}^n)^\vee$ the positive (with respect to the co-orientation) primitive generator of the annihilator of the vectors tangent to $H$. Then 
\begin{equation}\label{eq-hyper-fun} x^{l_H}\in \Bbbk[(\mathbb{Z}^n)^\vee] \end{equation} 
defines an element of each $\mathcal{O}_0(P)$, where $P$ is an admissible polytope. The formal expression of $x^{l_H}$ is independent of $P$ and is nothing but $x^{l_H}\in \Bbbk[[(\mathbb{Z}^n)^\vee]].$

Given a hyperplane $H$ and an  $S\subset B_0^n$ intersecting exactly one of the components of $B_0^n\setminus H$ we refer to the co-orientation for which $l_H$ increases as we go towards the side containing $S$, the \emph{$S$-co-orientation}.

To orient the reader let us work out an example. Let $P$ be the triangle with vertices in $(0,0),(1,0),(0,1)$. Let $H$ be the line containing the hypotenuse of $P$ with its $P$-co-orientation. Let $e_1^\vee$ and $e_2^\vee$ be the standard basis of $(\mathbb{Z}^2)^\vee$ and let $z_i:=x^{e_i^\vee}$ for $i=1,2$. Then $$x^{l_H}=z_1^{-1}z_2^{-1}.$$
\\

 The following is a version of the uniqueness of analytic continuation.

\begin{proposition}\label{prpUnique}
Assume that $P\subset Q\subset B_0^n$ are admissible polytopes and let $Q$ be connected. Then, the restriction map $\mathcal{O}_0(Q)\to \mathcal{O}_0(P)$ is injective.
\end{proposition}
\begin{proof}
    It suffices to prove the case where $P$ is a singleton $\{p\}$. This is immediate from Proposition \ref{prop-basic-hartogs} and Proposition \ref{prpTropTorus}(2).
\end{proof}

Let  $\cT^*_{\mathcal{Y}^n_0}$ (resp. $\Omega^*_{\mathcal{Y}^n_0}$) be the sheaf of polyvector fields (resp. differential forms) on $\mathcal{Y}^n_0$. An easy computation shows that\footnote{The reader should feel free to take these as the definition.} for $P$ an admissible convex polygon we have an isomorphism of graded normed algebras \begin{equation}\label{eq-BV-comp-B-side} \cT^*_{\mathcal{Y}^n_0}(p_0^{-1}(P))\simeq \mathcal{O}_0(P)[\partial_1,\ldots ,\partial_n]\end{equation} where the degree $1$ super-commutative variables $\partial_i$ correspond to the valuation zero derivations $y_i\frac{\partial}{\partial y_i};$ and an isomorphism of differential graded algebras \begin{equation}\label{eq-BV-diff}\Omega^*_{\mathcal{Y}^n_0}(p_0^{-1}(P))\simeq \mathcal{O}_0(P)[dy_1,\ldots ,dy_n],\end{equation} where $dy_i$ are degree $1$ super-commutative variables and the differential on the right hand side is given by the usual exterior derivative $d$. Let us also note that given a a map of algebras $g: \mathcal{O}_0(P)\to \mathcal{O}_0(P')$, we obtain a unique dga map $$g_*: \Omega^*_{\mathcal{Y}^n_0}(p_0^{-1}(P))\to \Omega^*_{\mathcal{Y}^n_0}(p_0^{-1}(P'))$$ whose degree $0$ part is equal to $g.$

Consider the nowhere vanishing analytic volume form 
$$\Omega_0:=\frac{dy_1}{y_1}\wedge\ldots\wedge\frac{dy_n}{y_n}$$ 
on $\mathcal{Y}^n_0$. 

Observe that $\Omega_0$ gives rise to a degree reversing isomorphism, $i_{\Omega_0}:\Omega^*_{\mathcal{Y}^n_0}\to \cT^{n-*}_{\mathcal{Y}^n_0}$, between the poly-vector fields and the differential forms. We denote by $div_{\Omega_0}:\cT^*_{\mathcal{Y}^n_0}\to \cT^{*-1}_{\mathcal{Y}^n_0}$ the degree $-1$ operator on  given by 
\begin{equation}
div_{\Omega_0}\omega:=(-1)^{|\omega|}\iota_{\Omega_0}\circ d\circ\iota_{\Omega_0}^{-1}\omega
\end{equation}
for a homogeneous $\omega$, where $d$ is the exterior derivative. A basic observation is that $(\cT^*_{\mathcal{Y}^n_0},div_{\Omega_0})$ is a sheaf of BV algebras (see \cite[Section 2.1]{barannikov} with the sign correction from \cite[Section 5]{mandel}). We thus obtain a sheaf of BV algebras on the $G$-topology of $B_0^n$ given by
$$\cA^*(P):=\left(\cT^*_{\mathcal{Y}^n_0}(p_0^{-1}(P)),div_{\Omega_0}\right).$$


\begin{proposition}\label{prpCYBV}

Consider an isomorphism of unital BV algebras $f^*:\cA^*(P)\to \cA^*(P')$ with $P$ and $P'$ convex admissible polytopes in $B_0^n$. Then, the map $$(f^0)_*:\Omega^*_{\mathcal{Y}^n_0}(p_0^{-1}(P))\to \Omega^*_{\mathcal{Y}^n_0}(p_0^{-1}(P'))$$ sends $\Omega_0$ to $c\Omega_0$, for some $c\in \Bbbk^*.$
\end{proposition}
\begin{proof}
We have two maps
\begin{equation}
(f^0)_* \text{ and }g=i_{\Omega_0}^{-1}\circ f\circ i_{\Omega_0}:\Omega^*_{\mathcal{Y}^n_0}(p_0^{-1}(P))\to \Omega^*_{\mathcal{Y}^n_0}(p_0^{-1}(P')).
\end{equation} 
Considering $\Omega^*_{\mathcal{Y}^n_0}(p_0^{-1}(P))$ as an $\cA^0(P)=\mathcal{O}_0(P)$ module where $a$ acts by multiplication with $f^0(a)$, we see that both of these are maps of $\cA^0(P)$ modules. In addition, both maps are chain maps with respect to the exterior derivative.

Note that $g(1)$ has to be a non-zero closed element, i.e., an element $c\in \Bbbk^*$. Then we immediately see that $g$ restricted to $\cA^0(P)$ equals $cf^0$ for $c=g(1)$. 
Observe that $\Omega^k\cA^0(P)$ for $k\geq 1$ is generated as an $\cA^0(P)$ module by the forms $dy^I$ for $|I|=k$. Each of these generators is exact. From this we deduce inductively that $g=c(f^0)_*$ in all degrees. 

Finally, note that  $g\Omega_0=\Omega_0$ since $i_{\Omega_0}(\Omega_0)=1$. The claim follows

 \end{proof}

\subsection{Analytification of $Y_k$}\label{Sec-Yk-analyt}

Let us now consider the affine variety $$Y_k=Spec(\Bbbk[x,y,u^{\pm}]/(xy-(u+1)^k),$$ for $k\geq 1$ and its rigid analytification $Y_k^{an}$. Note that $Y_k$ is smooth if and only if $k=1$. On the other hand $Y_k$ is normal for all $k\geq 1$ as it is an open subset of a surface $\{f=0\}\subset \mathbb{\Bbbk}^3$ with only isolated singularities. As a set $$Y_k^{an}=M(\Bbbk[x,y,u^{\pm}]/(xy-(u+1)^k).$$

In this section we prove that $Y_k^{an}$ admits a Stein continuous map $p_k$ to $B_k$, which induces the integral affine structure of Section \ref{ss-integral-affine-local} on $B_k^{reg}$. In particular all points other than the origin are regular. The notion of an integral affine structure induced by a continuous map as above is explained in \cite[Theorem 1]{koso}. This section is based on Section 8 of \cite{koso}.

\begin{remark}
Suitably interpreted the constructions extend to the $k=0$ case and recover the previous section in dimension $2$. 
\end{remark}

The rigid analytic space $Y_k^{an}$ is embedded inside the analytification of $\mathbb{A}^2_{\Bbbk}\times \Bbbk^*$ as a set by functoriality of analytification. We consider the map \begin{equation}\label{eq-KS-proj} ((x,y,u)\mapsto (\min{(0,val(x))},\min{(0,val(y))},val(u)).\end{equation} The restriction of this map to $Y_k^{an}$ defines a map $$p_k:Y_k^{an}\to \mathbb{R}^3.$$

\begin{remark}
Note that we have a map $\mathbb{A}_{\Bbbk}^1\to \mathbb{R}\cup\{\infty\}$ which sends each element to its valuation. To orient ourselves, this map sends $1$ to $0$ and $0$ to $\infty$. The maps in  Equation \eqref{eq-KS-proj} are obtained by composing this map with the map that collapses $\mathbb{R}\cup\{\infty\}$ to $\mathbb{R}_{\leq 0}$ in the non-negative side. Now the entire ball of radius $1$ (and nothing else), including $0$ and $1$, maps to $0$.  Here the ball is defined with respect to the norm given by $e^{-\val}$.
\end{remark}

Let us consider a copy of $(\Bbbk^*)^2$ with coordinates $\xi^+,\eta^+$. We have the tropicalization map $p_0:(\Bbbk^*)^{2}\to \mathbb{R}^2$ given by $(\xi^+,\eta^+)\mapsto (v^+,u^+)=(val(\xi^+),val(\eta^+))$. Let $B^+$ be the complement of the ray $\{(v^+,0)\mid v^+\geq 0\}$ in $\mathbb{R}^2$, and let $T^+$ be the preimage of $B^+$ under the tropicalization map.

Define an embedding $$g^+:(\Bbbk^*)^{2}\to Y_k^{an}$$  by $$x\mapsto \xi^+, y\mapsto\frac{1}{\xi^+}\left(1+\eta^+\right)^k\text{ and } u\mapsto \eta^+.$$ We can restrict this embedding to $T^+$. There then exists a map $f^+:B^+\to \mathbb{R}^3$ fitting into the diagram
\begin{align}\label{eqT+diagram}
\xymatrix
   { T^+ \ar[d]_{p_0} \ar[r]^{g^+}& Y_k^{an} \ar[d]^{p_k} \\
     B^+ \ar[r]^{f^+} & \mathbb{R}^3
   }.
\end{align}
$f^+$ can be computed explicitly:
\begin{proposition}\label{prop-B-plus-intertwine}
$$
f^+(v^+,u^+)= 
\begin{cases}
\left(v^+, \min{(0,ku^+-v^+)},u^+\right),& \qquad v^+\leq 0\\
\left(0, -v^++\min{(ku^+,0)},u^+\right), &\qquad v^+\geq 0.
\end{cases}
$$
\end{proposition}
\begin{proof}
To see this substitute $\val(y)=\begin{cases}ku^+-v^+,&u^+\leq 0\\v^+,&u^+\geq 0\end{cases}$ into \eqref{eq-KS-proj} and verify separately for the two cases of $v^+$. 
\end{proof}

Analogously consider a $(\Bbbk^*)^2$ with coordinates $\xi^-,\eta^-$ with the same tropicalization map. Let $B^-$ be the complement of the ray $\{(v^+,0)\mid v^+\leq 0\}$ in $\mathbb{R}^2$, and let $T^-$ be the preimage of $B^-$ under the tropicalization map.

Define an embedding $$g^-:(\Bbbk^*)^{2}\to Y_k^{an}$$
by $$x\mapsto  \xi^-\left(1+\eta^-\right)^k , y\mapsto\frac{1}{\xi^-}\text{ and } u\mapsto \eta^-.$$ We can restrict this embedding to $T^-$.

As before, define a map $f^-:B^-\to \mathbb{R}^3$
$$
f^-(v^-,u^-)= 
\begin{cases}
(v^-+\min{(0,ku^-)},0,u^-), v^-\leq 0\\
(\min{(v^-+ku^-,0)},-v^-,u^-), v^-\geq 0.
\end{cases}
$$
Then we have a commutative diagram 
\begin{align}\label{eqT-diagram}
\xymatrix
   { T^- \ar[d]_{p_0} \ar[r]^{g^-}& Y_k^{an} \ar[d]^{p_k} \\
     B^- \ar[r]^{f^-} & \mathbb{R}^3.
   }
\end{align}


\begin{remark}
One can also analyze the image of $p_0^{-1}(\mathbb{R}^2\setminus B^+)$ under $g^+$. The set $$Z:=\{\eta^+=-1, val(\xi^+)\geq 0\}$$ all maps to the fiber above the origin of $p_1$. One can also check that if we use the continuous extension of $B^+\to \mathbb{R}^3$ to $\mathbb{R}^2\to \mathbb{R}^3,$ the diagram \begin{align}\xymatrix
   { (\Bbbk^*)^{2}\setminus Z \ar[d]_{p_0} \ar[r]^{g^+}& Y_k^{an} \ar[d]^{p_k} \\
     \mathbb{R}^2\ar[r] & \mathbb{R}^3
   }
\end{align} commutes. One could imagine $Z$ being lifted to be above the origin.
\end{remark}

\begin{proposition}\label{prop-KS-cover}The images of $g^-$ and $g^+$ cover $Y_k^{an}\setminus \{(0,0, -1)\}.$The intersection of the images is $Y_k^{an}\setminus \{xy=0\}.$ The corresponding transition map from the $-$ chart to the $+$ chart is given by the analytification of the map $(\Bbbk^{*})^2\setminus \{\eta^-=-1\}\to (\Bbbk^{*})^2\setminus \{\eta^+=-1\}$ given by $$(\xi^-,\eta^-)\mapsto (\xi^-(1+{\eta^-})^k,\eta^-).$$
\end{proposition}
\begin{proof}
Direct computation.
\end{proof}

Note also that the images of $g^-|_{T^-}$ and $g^+|_{T^+}$ cover $Y_k^{an}-p_k^{-1}(\{(0,0)\})=p_k^{-1}(B_k^{reg}).$ The intersection of the images is $Y_k^{an}-p_k^{-1}(\{u=0\}).$
\begin{proposition}\label{PrpBsideMonodromy} The restriction of the transition map from Proposition \ref{prop-KS-cover} $$(\Bbbk^{*})^2\setminus \{val(\eta^-)=0\}\to (\Bbbk^{*})^2\setminus \{val(\eta^+)=0\}$$
covers the map $B_0\setminus \{u^-=0\}\to B_0\setminus \{u^+=0\}$ :
$$
(v^-,u^-)\mapsto 
\begin{cases}
(v^-,u^-), u^-> 0\\
(v^-+ku^-,u^-), u^-< 0.
\end{cases}
$$
\end{proposition}

Consider the piecewise affine subspace $P_k(Y_k^{an})\subset \bR^3$ as a nodal integral affine manifold with charts the maps $f^{\pm}$ from the diagrams \eqref{eqT+diagram} and \eqref{eqT-diagram}. Proposition \ref{PrpBsideMonodromy}  guarantees there is a nodal integral affine isomorphism $i_k:p_k({Y_k^{an}})\to B_k$ fitting into the diagram
\begin{align}\label{EqIdentification}\xymatrix
   { B^{\pm} \ar[rd]_{E^{\pm}_k} \ar[r]^{f^{\pm}}& p_k({Y_k^{an}}) \ar[d]^{i_k} \\
      & B_k
   }
\end{align} 
where $E^\pm_k$ defined in Section \ref{ss-integral-affine-local}. From now on we will omit $i_k$ from the notation and will not care to distinguish between $B_k$ and $p_k(Y_k^{an})$.

Before we move further let us also concretely analyze how the structure sheaf on $Y_k^{an}$ is constructed. It is more helpful to think of $Y_k$ as embedded inside $\mathbb{A}^4_{\Bbbk}$ with the equations $$xy-(u+1)^k=0\text{, }uu'=1.$$ 

\begin{remark}\label{rmkPktoR4} We get a more symmetric picture of the map $\mathbb{A}^2_{\Bbbk}\times \Bbbk^*\to \mathbb{R}^3$ from Equation \eqref{eq-KS-proj} using the inclusion into $\mathbb{A}^4_{\Bbbk}$. Note that the image is in fact contained in $\mathbb{R}^2_{\leq 0}\times\mathbb{R}$. We have the commutative diagram:
\begin{align}\label{eqDiagEmb}\xymatrix
   { \mathbb{A}^2_{\Bbbk}\times \Bbbk^* \ar[d]_{p_0} \ar[r]& \mathbb{A}^4_{\Bbbk} \ar[d] \\
     \mathbb{R}^2_{\leq 0}\times\mathbb{R} \ar[r]^j & \mathbb{R}^4_{\leq 0}
   }
\end{align} where the right vertical maps simply applies $\min{(0,val(\cdot))}$ to each coordinate and the  map $j$ is $(a,b,c)\mapsto (a,b,\min{(0,c)},\min{(0,-c)}).$
\end{remark}

We then consider the defining exhaustion of $\mathbb{A}^4_{\Bbbk}$ by the balls of radius $e^{r}>0$. $$\mathbb{B}^4(r):=M(\Bbbk\langle c^{-1}x,c^{-1}y, c^{-1}u, c^{-1}u'\rangle),\text{for  some }val(c)=-r.$$ By construction $\mathbb{B}^4(r)\cap Y_k^{an}$ is identified with $$M(\Bbbk\langle c^{-1}x,c^{-1}y, c^{-1}u, c^{-1}u'\rangle/(uu'-1, xy-(u+1)^k) ).$$ We also have induced inclusions of affinoid subdomains $$\mathbb{B}^4(r)\cap Y_k^{an}\subset \mathbb{B}^4(r')\cap Y_k^{an},$$ for any $r<r'$, which gives by definition an exhaustion of $Y_k^{an}$ by taking $r\to \infty$.

\begin{definition} Recall the identification $i_k$ of $B_k$ with $p_k(Y_k^{an})$ in \eqref{EqIdentification}. For $a>0,$ let $P(a) \subset B_k$ be the set mapping under the map $j$ of \eqref{eqDiagEmb} to the intersection of the interval $[-a,0]^4$ with the image $j(p_k(Y_k^{an}))=j(B_k)$. 
\end{definition}

Unwinding definitions we have 
$$p_k^{-1}(P(a))=\mathbb{B}^4(a)\cap Y_k^{an}.$$ Observe $P(a)$ 
is an admissible polygon. To see this note the embedding of $p_k(Y_k^{an})$ into $\bR^4$ is piecewise affine. Indeed, it is given by the composition of $j$ with the maps $f^{\pm}$ from \eqref{eqT+diagram}, \eqref{eqT-diagram}. The latter are piecewise affine by Proposition \eqref{prop-B-plus-intertwine}. Thus the intersection with a cube in $\bR^4$ is polygonal. 
We refer the reader to Figure \ref{fig-P} for what the polygons $P(a)$ look like.

\begin{figure}
\includegraphics[width=\textwidth]{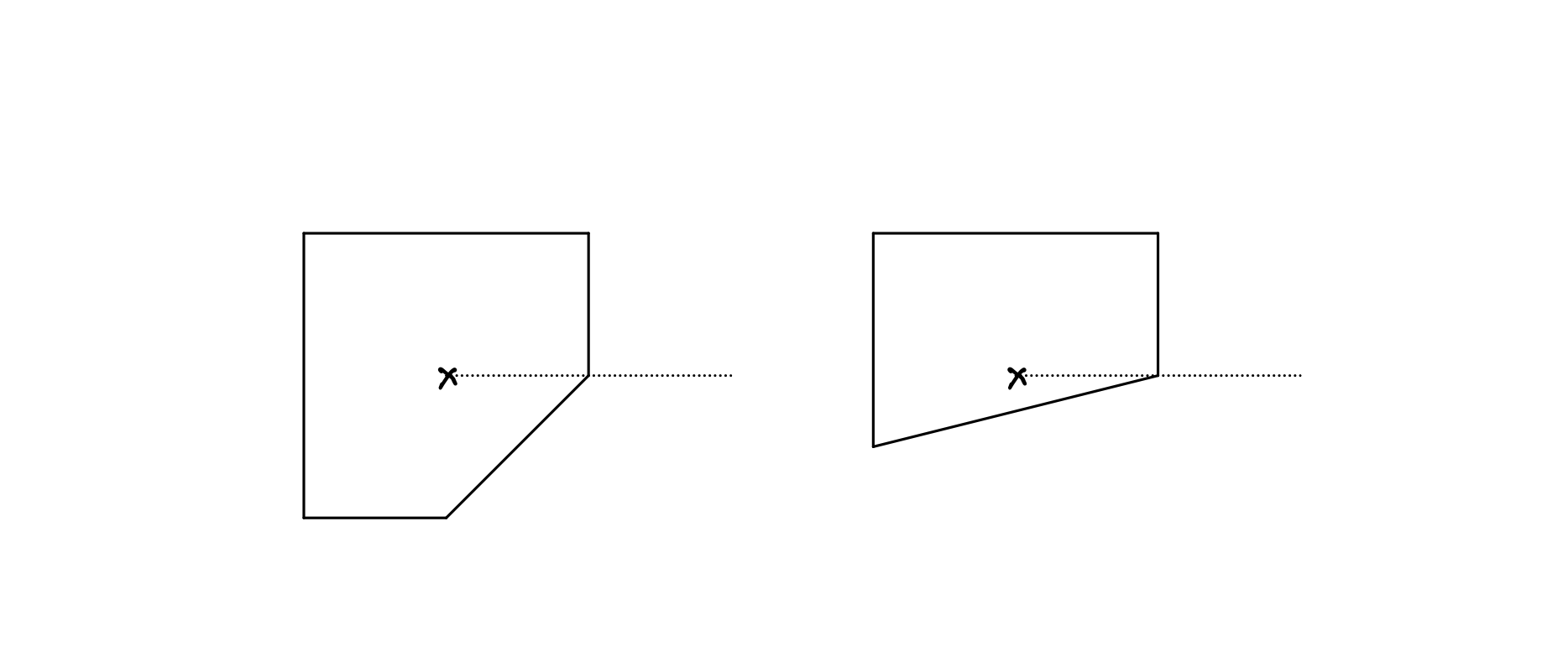}
\caption{Drawn inside $B_k$ with the dashed ray representing the defining monodromy invariant ray. The left figure is the $k=1$, and the right one is $k=4$.}
\label{fig-P}
\end{figure}

Recall that for $\mathcal{Y}$ a rigid analytic space a map $\mathcal{Y}\to B$ is continuous if  the preimage of each admissible polygon is an admissible open of $\mathcal{Y}$ and the preimages of admissible coverings are admissible coverings in $\mathcal{Y}$. We call such a map strongly Stein if the preimage of an admissible convex polygon is isomorphic to an affinoid domain as a rigid analytic space. 

\begin{proposition}\label{prop-Stein-k}
The map $p_k: Y_k^{an}\to B_k$ is a strongly Stein continuous map.
\end{proposition}
\begin{proof}
For admissible convex polygons not containing the origin Propositions \ref{prop-B-plus-intertwine} and \ref{prop-KS-cover} allow us to  intertwine $p$ over $P$ with the tropicalization map $p_0:\mathcal{Y}_0^2\to B_0^n=\mathbb{R}^2$.  In particular, $p^{-1}(P)$ is affinoid by Proposition \ref{prpTropTorus}. The Stein property similarly holds for $P(a)$ with $a>0$ since it is an affinoid domain by construction. 

Suppose we have established the Stein property for some admissible polygon $P$ containing the node, i.e., that $p_k^{-1}(P)$ is an affinoid domain. We will show it follows for any other such admissible polygon $Q$  obtained from $P$ by intersecting $P$ with a rational halfspace containing the node in its interior. This will finish the proof of the Stein property since we can obtain each admissible polytope containg a node by successively chopping off some $P(a)$ by rational halfspaces. 

Let $U$ be a rational halfspace containing the node in its interior. Now we claim that $$p_k^{-1}(P\cap U)\subset p_k^{-1}(P)$$ is a Weierstrass subdomain (cf. \cite[Proposition 3.3.2]{kapranov}). Let $B^{\pm}\subset B_0$ be the complement of the rays generated by $(\pm1,0)$.   Here we are specializing the notation $B^{\pm}_e$ introduced in Section \ref{ss-integral-affine-local} to the case $e=(1,0)$. Recall we have integral affine embeddings $E^{\pm}:B^{\pm}\to B_k$ mapping whose image is the complement of the corresponding eigenray. Consider the rational line $H$ that is the boundary of $U$ and use its $U$-co-orientation. Note $H$ does not pass through the node. Thus it is contained in the complement of an eigenray. Assume  $H$ is contained in the complement $E^+(B^+)$ of the $+$ eigenray in $B_k$ (if it is contained in $E^-(B^-)$ the same argument works). Then as in Equation \eqref{eq-hyper-fun} and using Proposition \ref{prop-B-plus-intertwine} we obtain a function $F\in \mathcal{O}(p_k^{-1}(Q))$ for any admissible polytope $Q\subset E^+(B^+)$. Under the embedding $g^+$ the function $F$ is equal to $(\xi^+)^a(\eta^+)^b$ with $a\in\mathbb{Z}_{\geq 0}$ and $b\in\mathbb{Z}.$ Since $\xi$ and $(\eta^+)^{\pm 1}$ are global functions, we obtain that $F$ in fact extends to functions in $\mathcal{O}(p_k^{-1}(Q))$ for any admissible polytope $Q\subset B_k, $ particularly for $Q=P.$\footnote{This can also be seen as a result of the higher dimensional version of the removable singularity theorem. We omit further discussion here} To finish we notice that $$p_k^{-1}(P\cap U)=\{x\in p_k^{-1}(P)\mid val(F(x))\geq b\}$$ for some $b\in \mathbb{R}$.

The Stein property is now established. To finish, it suffices to show that if $Q\subset P\subset B_k$ are admissible polygons such that $P$ contains the node but $Q$ does not, then $p_k^{-1}(Q)\subset p_k^{-1}(P)$ is an affinoid subdomain. 

It suffices to deal with the case where $Q=P\cap U$ with $U$ a rational halfspace which does not contain the node. The same argument as above works, but this time we obtain $$p_k^{-1}(P\cap U)=\{x\in p_k^{-1}(P)\mid val(F(x))\leq b\}$$ for some $F\in \mathcal{O}(p_k^{-1}(P))$. This means that we indeed have an affinoid subdomain, but this time a Laurent subdomain.
\end{proof}

%
%

\begin{definition} Let us define a sheaf $\mathcal{O}_k$ on the $G$-topology of $B_k$ by $$\mathcal{O}_k(P):=\mathcal{O}(p_k^{-1}(P)),$$ for an admissible polygon $P$. \end{definition}

\begin{corollary}\label{cy4PropertiesOk}
For $k\geq 0$, the sheaf $\mathcal{O}_k$ on the $G$-topology of $B_k$ satisfies the following properties

\begin{itemize}
\item (Affinoidness) If $P$ is an admissible convex polygon, $\mathcal{O}_k(P)$ is a reduced affinoid algebra.
\item (Subdomain) If $Q\subset P$ are admissible convex polygons, then the morphism of affinoid domains induced from the restriction map $\mathcal{O}_k(P)\to \mathcal{O}_k(Q)$: $$M(\mathcal{O}_k(Q))\to M(\mathcal{O}_k(P))$$ has image an affinoid subdomain and it is an isomorphism of affinoid domains onto its image.
\item (Strong cocycle condition) For admissible convex polygons $Q,Q'\subset P$, we have $$im(M(\mathcal{O}_k(Q)))\cap im(M(\mathcal{O}_k(Q')))=im(M(\mathcal{O}_k(Q\cap Q')))$$ inside $M(\mathcal{O}_k(P))$
\item (Independence) Let $Q_1,\ldots, Q_N\subset P$ be admissible convex polygons such that $P=\bigcup Q_i$. Then $$\bigcup  im(M(\mathcal{O}_k(Q_i)))=  M(\mathcal{O}_k(P)).$$
\item (Separation) Let $Q\subset P$ be an inclusion of admissible convex polygons. For any $p\in M(\cO_k(P))\setminus im(M(\cO_k(Q))$ there is a small admissible polygon $Q'\subset P$ which is disjoint of $Q$ such that $p\in im(M(\cO_k(Q')))$. 
\end{itemize}

\end{corollary}
\begin{proof}
We have already proved the first two properties in the proof of Proposition \ref{prop-Stein-k}. The last three are also automatic because for $Q\subset P$ admissible polygons, we have $$im(M(\mathcal{O}_k(Q))=p_k^{-1}(Q)\subset M(\mathcal{O}_k(P))$$
under the canonical identification of $M(\mathcal{O}_k(Q))$ with $p_k^{-1}(Q).$
\end{proof}


We will also need a version of the  Hartogs property. For this we quote a general extension property in rigid analytic geometry from \cite{muenster}. First, recall that for any reduced affinoid algebra $A$, we have a canonical supremum (or spectral) norm \cite[\S3.1]{bosch}. We define the reduction of a non-archimedean normed field or algebra to be the quotient of norm $1$ elements by the norm less than $1$ elements. We denote reductions by a tilde. With this in mind, we know that $\tilde{A}$ is a finite type algebra over $\tilde{\Bbbk}.$

For the affinoid domain $X=M(A)$, we say that a subset $U\subset X$ is \emph{$n$-small}\footnote{We deviate here from the terminology of \cite{muenster}. There, the complement $X\setminus U$ is referred to as a ball figure of dimension $n$.}  if there is a collection $f_1,\dots f_k\in A$ that we are thinking of as functions $X\to\Bbbk $ such that 
    \begin{itemize}
        \item $|f_i(x)|\leq 1$, for all $i=1,\ldots ,k$ and $x\in X$ (or equivalently $|f_i|_{sup}\leq 1$)
        \item the vanishing locus $V(\tilde{f}_1,\ldots,\tilde{f}_k)\subset Spec(\tilde{A})$ is $n$-dimensional,
    \end{itemize}
such that $$U=\cap_{j=1}^k\{x\in X:|f_j(x)|<1\}.$$ 

Then \cite[Proposition 2.5]{muenster} states the following Hartogs property. 
\begin{proposition}\label{prpHartogsGeneral}
   Suppose $X$ is a normal affinoid domain of dimension $n+2$ and $U\subset X$ is $n$-small . Then the restriction map 
   $$\cO(X)\to\cO(X\setminus U)$$ is an isomorphism.
\end{proposition}
\begin{proof}
    This is a rephrasing of part of \cite[Proposition 2.5]{muenster}. 
\end{proof}

\begin{proposition}\label{prop-Hartogs-B} Let $P\subset B_k$ be an admissible convex polygon containing the node. Then, restriction map $$\mathcal{O}_k(P)\to \mathcal{O}_k(\partial P)$$ is an isomorphism.

\end{proposition} 

\begin{proof}
We first show the set $p_k^{-1}(P\setminus\partial P)$ is $0$-small for $P:=P(a)$ with $a>0.$ We take $c'x, c'y, c'u, c'u'$ with $val(c')=a$ as the functions $f_i$ in the definition. The reduction of the affinoid domain $p_k^{-1}(P)$ is the affine variety over the ground ring defined by the equations $xy=0,uu'=0$ in $\tilde{\Bbbk}[x,y,u,u']$. The vanishing locus of the reductions is given by the equations $x=y=u=u'=0$, and so is indeed $0$ dimensional. For more general $P$ the same follows as in  the proof of Proposition \ref{prop-Stein-k} by the fact they are all obtained by imposing additional constraints. $Y_k^{an}$ is normal as analytification preserves normality \cite[Proposition 1.3.5]{schoutens}. We now rely on Proposition \ref{prpHartogsGeneral}.

\end{proof}

The Hartogs principle from Proposition \ref{prop-Hartogs-B} allows us to prove another uniqueness of analytic continuation result. 

\begin{proposition}\label{prop-analytic-cont-B} Let $P\subset B_k$ be an admissible convex polygon containing the node and $p\in \partial P$. Then the restriction map $$\mathcal{O}_k(P)\to \mathcal{O}_k(\{p\})$$ is injective.
\end{proposition} 

\begin{proof}
    By Proposition \ref{prop-Hartogs-B} and functoriality of restriction maps it suffices to prove the injectivity of the restriction map $\mathcal{O}_k(\partial P)\to \mathcal{O}_k(\{p\})$.   Let $E\subset\partial P$ be one of the edges of $P$. By Proposition \ref{prpUnique} and the Diagrams \eqref{eqT+diagram} and \eqref{eqT-diagram}, the map $\mathcal{O}_k(E)\to \mathcal{O}_k(\{p\})$ is injective. So it remains to prove the map $\mathcal{O}_k(\partial P)\to \mathcal{O}_k(E)$ is injective. Now observe that by the sheaf property, $\mathcal{O}_k(\partial P)$ consists of sections over each of the edges of $\partial P$ which agree on overlaps which are the vertices of $\partial P$.  Using Proposition \ref{prpUnique} and Diagrams \eqref{eqT+diagram} and \eqref{eqT-diagram} again, vanishing of the restriction to a vertex implies vanishing of the restriction to any edge containing that vertex. By connectedness of $\partial P$ we deduce vanishing over all of $\partial P.$



\end{proof}

\begin{remark}
    The uniqueness of analytic continuation property is much more general. In a MathOverflow thread \cite{Mathoverflow} it is shown that given an affinoid domain $X$ and a non-empty affinoid subdomain $Y\subset X$, the map $\mathcal{O}_X(X)\to \mathcal{O}_X(Y)=\mathcal{O}_Y(Y)$ is injective as long as $X$ is irreducible. With the stronger assumption that $X$ is smooth and connected, this result is proved in \cite[Proposition 4.2]{ardakov}.
\end{remark}


Recall that in the previous section we defined on $Y_0$ (resp. $Y_0^{an}$) the algebraic (resp. analytic) volume form $\Omega_0$. 

\begin{proposition}\label{prop-hol-vol-form}
By taking the residue of the meromorphic form $$Res_{Y_k\subset \mathbb{A}^3_\Bbbk} \left(-\frac{dxdydu}{u(xy-(u+1)^k)}\right),$$ we can define an algebraic (resp. analytic) volume form $\Omega_k$ on the smooth points of $Y_k$ (resp. $Y_k^{an}$). The pullback of $\Omega_k$ under the embeddings $g^{\pm}$ is $\Omega_0$.
\end{proposition}
\begin{proof}
Note that $\Omega_k$ is defined by the equality: $$-\frac{dxdydu}{u(xy-(u+1)^k)}=\frac{uydx+uxdy-(xy-(u+1)^{k-1})du}{u(xy-(u+1)^k)}\wedge\Omega_k.$$ This implies the desired properties. For example when $x\neq 0$, we have $\Omega_k=\frac{dxdu}{xu}.$
\end{proof}

\section{Symplectic cohomology type invariants}\label{s-rel-sh}
In this section, we deal with the relation between various notions of symplectic cohomology that appear in the literature. In particular, we study the relation between 
\begin{itemize}
\item relative symplectic cohomology of a compact set $K$ as typically defined over the Novikov ring,
\item a version of relative symplectic cohomology defined over the ground ring when the ambient manifold is exact,
\item symplectic cohomology a l\`{a} Viterbo associated with Liouville manifolds of finite type.
\end{itemize}
The material of Sections \ref{ss-filt-map} and  \ref{SecSHreview} is a review. The content of Sections \ref{ss-exact-review} and \ref{ss-liouville-review} is mostly known to experts, but to our knowledge some of the statements have not explicitly appeared in the literature. We alert the reader that Propositions \ref{prpBoundaryDepthComparison} and \ref{tmSHComparison} are central to the argument in the next section. 
\subsection{Filtrations, torsion and boundary depth}\label{ss-filt-map}

A \emph{filtration map} on an Abelian group $A$ is a map $\rho: A\to \mathbb{R}\cup\{\infty\}$ satisfying the inequality $$\rho(x+y)\geq \min{(\rho(x),\rho(y))},$$ the equality $\rho(x)=\rho(-x)$, and sending $0$ to $\infty$.  If $\rho^{-1}(\infty)=\{0\}$ the filtration map is called Hausdorff. A \emph{non-archimedean valuation} on a $\Lambda$-vector space satisfies these assumptions along with multiplicativity for scalar multiplication. 

 Note that if $(V_{i}, \rho_{i})$ are Abelian groups equipped with filtration maps indexed by a set $i\in I$, then $\bigoplus_{\alpha\in I} V_{\alpha}$ is equipped with a filtration map given by $$\rho\left(\sum v_i\right):=\min_{i\in I}{\left(\rho_i(v_i)\right)}.$$ Let us call this the $min$ construction.
 
 Let us call an Abelian group with a filtration map a \emph{filtered Abelian group}. Let us call the value of an element under the filtration map the \emph{filtration value}. We call a map from one filtered Abelian group to another \emph{bounded} if the map decreases filtration values at most by a uniform constant real number.

 We can define a pseudometric on an abelian group $A$ with a filtration map $\rho$ by $d(a,a'):=e^{-\rho(a-a')}$. This defines a topology on $A$ as well. The completion $\widehat{A}$ of $A$ is defined by taking the abelian group of Cauchy sequences in $A$ and modding out by the subgroup of sequences which converge to $0$. $\widehat{A}$ is equipped with a canonical filtration map: $$\rho((a_i)_{i\in \mathbb{N}})=\varinjlim \rho(a_i).$$ We call $A$ complete,  if the natural map $A\to \widehat{A}$ is bijective. If $A$ is complete, it is automatically Hausdorff. Note that bounded maps are automatically continuous but the converse is not necessarily true.

 A filtration map defines an exhaustive filtration by the subgroups $$F_{\geq \rho_0}A:=\{ a\in A\mid \rho(a)\geq \rho_0\}.$$ We can equivalently describe the completion of $A$ as the inverse limit $$\varprojlim_{\rho_0\in \mathbb{Z}} A/F_{\geq \rho_0}A.$$

We call a graded Abelian group with a filtration map in each degree and a differential that does not decrease the filtration values a \emph{filtered chain complex}. The homology of a filtered chain complex is naturally a filtered graded Abelian group  by taking the supremum of the filtration values of all representatives for each homology class.

For $(C,d,\rho)$ a filtered chain complex and real numbers $a<b$, we define the subquotient complexes $$C_{(a,b]}:=\frac{\{\rho > a\}}{\{\rho > b\}} .$$

\begin{lemma}\label{lem-filt-qis}
Let $(C,d,\rho)$, $(C',d',\rho')$ be filtered chain complexes whose underlying graded Abelian groups are degreewise complete. Let $f:C\to C'$ be a chain map which does not decrease the filtration values. Assume that for every $a<b$ real numbers, the induced map  $$C_{(a,b]}\to C'_{(a,b]}$$ is a quasi-isomorphism, then $f$ is a quasi-isomorphism.
\end{lemma}
\begin{proof}Consider the filtrations by $F_nC:={\{\rho > n\}}$ and $F_nC':={\{\rho' > n\}}$ for $n\in \mathbb{Z}$. The map $f$ induces a map on the corresponding spectral sequences. The condition that we are given implies that the map on the first pages is an isomorphism.
Now we apply the Eilenberg--Moore Comparison Theorem \cite[Theorem 5.5.11]{Weibel} relying on completeness.
\end{proof}

A Novikov ring $\Lambda_{\geq 0}$-module $A$ can be equipped with a filtration map defined by $$\rho(a):=\sup \{r\in\mathbb{R}\mid \text{ there exists }a'\in A \text{ with }a=T^ra'\}.$$ Moreover, $A\otimes \Lambda$ is naturally equipped with a non-archimedean valuation as follows. Consider the natural map $\iota: A\to A\otimes \Lambda$. For $a\in A\otimes \Lambda$, we define $$val_A(a):=-\inf\{r\in \mathbb{R}\mid T^ra\in \iota(A)\}.$$

 In fact, we can describe $A\otimes \Lambda$ more concretely as $A\times \mathbb{R}$ modulo the equivalence relation $(a,s)\sim (T^ra, s+r).$ Here $(a,s)$ is thought of as $\frac{a}{T^s},$ which makes the $\Lambda$-module structure transparent. We can construct a $\Lambda_{\geq 0}$-bilinear map $$A\times \Lambda\to A\times \mathbb{R}$$ by sending $(a,c)$ to $(c'a, -val(c))$ for $c=T^{val(c))} c'.$ This induces a map $A\otimes \Lambda\to A\times \mathbb{R}$ sending $a\otimes c$ to $(c'a, -val(c))$ for $c=T^{val(c))} c'$. We also have a map $A\times \mathbb{R}\to A\otimes \Lambda $ sending $(a,s)$ to $a\otimes T^{-s}.$ These are easily seen to be inverses of each other.

\begin{lemma}\label{lem-triv-comp}
Let $A$ be a free Novikov ring module, and equip $A\otimes \Lambda$ with the induced non-archimedean valuation. We can complete $A$ as a Novikov ring module $\hat{A}$ and complete $A\otimes \Lambda$ as a normed space $\widehat{A\otimes \Lambda}$. Then, the natural map $$\hat{A}\otimes \Lambda\to \widehat{A\otimes \Lambda}$$ is a valuation preserving functorial isomorphism.
\end{lemma}
\begin{proof}
We choose a basis $\{\gamma_i\}_{i\in I}$ of $A.$ Since $\{\gamma_i\otimes 1\}_{i\in I}$ forms a basis of $A\otimes \Lambda,$ it follows immediately that $val_A$ is nothing but the min construction with respect to this basis. We can also easily show that $\hat{A}$ is isomorphic to the module of infinite $\Lambda_{\geq 0}$-linear combinations $\sum a_i \gamma_i$ with the property that for every $r\geq 0,$ there are only finitely many $i\in I$ with $val(a_i)\leq r.$ From the concrete description of $\hat{A}\otimes \Lambda,$ we easily see that it is isomorphic to the module of infinite $\Lambda$-linear combinations $\sum a_i \gamma_i$ with the property that for every $r\in\mathbb{R},$ there are only finitely many $i\in I$ with $val(a_i)\leq r$ with the obvious inclusion of $\hat{A}$ under these identifications. This immediately shows the second statement.
\end{proof}

\begin{definition}\label{def-tor} Let $A$ be a module  over the Novikov ring $\Lambda_{\geq 0}$. For any element $v\in A$ we define the torsion $\tau(v)$ to be the infimum over $\lambda$ so that $T^{\lambda}v=0$. We take $\tau(v)=\infty$ if $v$ is non-torsion. We define the maximal torsion
\begin{equation}
\tau(A):=\sup_{v:\tau(v)<\infty}\tau(v).
\end{equation}
If there are no torsion elements we take $\tau(A)=-\infty$.
\end{definition}

\begin{lemma}
The natural map $\iota: A\to A\otimes \Lambda$ is injective if and only if $A$ has no torsion elements.
\end{lemma}
\begin{proof}
 From the concrete description of $A\otimes \Lambda$ above it follows that the element $a\otimes 1\in A\otimes \Lambda$ is equal to zero if and only if $(a,0)\sim (0,0)$, which is in turn equivalent to $T^\lambda a=0$ for some $\lambda\in\mathbb{R}$ as desired.
\end{proof} 

The importance of the following definition was explained thoroughly in \cite{locality}.

\begin{definition}For any $C$ chain complex over $\Lambda_{\geq 0}$, we define the \emph{$i$- homological torsion of $C$} as the maximal torsion of $H^i(C)$. If the $i$-homological torsion of $C$ is less than $\infty$, we say that it has \emph{homologically finite torsion at degree $i$}.\end{definition}

Now we relate it to the notion of finite boundary depth which to the best of our knowledge was introduced to symplectic topology by Usher \cite{usher}.

\begin{definition}
Let $(C,d,\rho)$ be a Hausdorff filtered chain complex. The boundary depth in degree $i\in \mathbb{Z}$ is defined as the infimum of $\beta\geq 0$ satisfying  $$\text{ for every }x\in im(d)\cap C^i, \text{there exists y}\in C^{i-1}\text{ s.t.  }dy=x\text{ and }\rho(x)-\rho(y)\leq \beta.$$ If this is not $+\infty$ we say that the boundary depth is finite in degree $i$. If the boundary depth is finite in all degrees we say that $C$ has finite boundary depth.
\end{definition}

\begin{proposition}\label{prop-tor-eq-bdry}
Assume that $(C,d,\rho)$ is obtained from a chain complex over the Novikov ring $C'$, i.e., $C=C'\otimes_{\Lambda_{\geq 0}}\Lambda$, where the underlying module of $C'$ is torsion free. Then the boundary depth of $C$ in degree $i$ is equal to the maximal torsion of $H^i(C')$. 
\end{proposition}
\begin{proof}

For an element $\alpha\in C$ and $\delta>0$, there is a unique element $\alpha_{\delta}\in C'$ such that $$\alpha=T^{\rho(\alpha)-\delta}\alpha_{\delta}\otimes 1.$$ 

Let us first show that $b_i$, the boundary depth in degree $i$, is at least $t_i$, the maximal torsion in degree $i$. Let $[z]\in H^i(C')$ be an arbitrary torsion element. Then, it follows that $z\otimes 1\in C^i$ is exact. For every $\epsilon >0$, there exists $$y\in C^{i-1}\text{ s. t. }dy=z\otimes 1\text{ and }-\rho(y)\leq b_i+\epsilon-\rho(z\otimes 1)\leq b_i+\epsilon.$$ Since $d$ does not decrease $\rho$ values, $T^{-\rho(y)+\epsilon}z$ defines an element in $C'$. It immediately follows that $T^{-\rho(y)+\epsilon}z= dy_{\epsilon}$ and that $T^{b_i+2\epsilon}[z]=0$.  We proved that $b_i+2\epsilon\geq t_i$, which implies the claim.

Conversely, let us consider an exact element $x\in C^i$. Then, it follows that $x_{\delta}$ is torsion for all $\delta>0$. Therefore, there exists $y_\delta$ such that $dy_\delta=T^{t_i+\delta}x_\delta$. But, then $$T^{\rho(x)-t_i-2\delta}d(y_\delta\otimes 1)=T^{\rho(x)-2\delta-t_i}T^{t_i+\delta}x_\delta\otimes 1=x.$$ Note that $$\rho(x)-\rho(T^{\rho(x)-t_i-2\delta}y_\delta\otimes 1)\leq t_i+2\delta$$This shows that $b_i\leq t_i+2\delta$, finishing the proof.

\end{proof}

\begin{corollary}\label{cor-fin-bd-tor}
Using the notation and assumptions of Proposition \ref{prop-tor-eq-bdry}, the complex $C$ has finite boundary depth in degree $i$ if and only if the complex $C'$ has homologically finite torsion in degree $i$. 
\end{corollary}

We will also need the following lemma.


\begin{lemma}\label{lem-completion-quotient}
Let $X$ be a Hausdorff filtered abelian group and $Y$ a subgroup. Then the completion of $X/Y$ is isomorphic to $\widehat{X}/\bar{Y},$ where $\bar{Y}$ is the closure of $Y\subset \widehat{X}$.
\end{lemma}
\begin{proof}
First, note that $\widehat{X}/\bar{Y}$ is complete. We also have a canonical map $$q:  X/Y\to \widehat{X}/\bar{Y},$$ because $X\to\widehat{X}\to \widehat{X}/\bar{Y}$ factors through $X/Y$. We claim that $q$ satisfies the universal mapping property (UMP) of the completion of $X/Y$ in the category of filtered abelian groups with bounded group homomorphisms, which finishes the proof.

Let $X/Y\to Z$ be any map in this category with $Z$ complete. By pre-composing with the quotient map we obtain a map $X\to Z$, which factors through $X\to \widehat{X}$ to give $\widehat{X}\to Z$ by the UMP of completion. This last map sends $Y$ to zero, which implies that it sends $\bar{Y}$ to zero by continuity. Therefore, we get the desired factorization through $q$. The uniqueness also follows because otherwise we contradict the uniqueness part of the UMP for $X\to\widehat{X}$ noting that $\widehat{X}\to \widehat{X}/\bar{Y}$ is surjective.
\end{proof}

\subsection{Relative symplectic cohomology}\label{SecSHreview}
A symplectic manifold $M$ is said to be \emph{geometrically bounded} if there exists an $\omega$-compatible almost complex structure $J$ such that the associated metric $g_J$ is equivalent to a complete Riemannian metric $h$ whose sectional curvature is bounded from above and below and whose radius of injectivity is bounded away from $0$.  Equivalence between $g_J$ and $h$ means there is a constant  $C>1$ so that
$$
\frac1{C}g_J(v,v)<h(v,v)<Cg_J(v,v),
$$
for all tangent vectors $v$. An $\omega$-compatible $J$ is said to be geometrically bounded if it is a witness for $\omega$ being geometrically bounded. Note that a given $\omega$ may have geometrically bounded almost complex structures whose associated Riemannian metrics are mutually inequivalent. Nevertheless, the structures we construct below are independent up to quasi-isomorphism (which itself is specified up to contractible choice) of the choice of $J$ as well as any other choices.

It is often useful to consider a more restricted notion from \cite[Section 3]{locality}. A function $f$ on a geometrically bounded symplectic manifold is called \emph{admissible} if it is proper, bounded below, and there is a constant $C$ such that with respect to a geometrically bounded almost complex structure $J$ we have $\|X_f\|_{g_J}<C$ and $\|\nabla X_f\|_{g_J}<C.$ We say $M$ is \emph{geometrically of finite type if} it admits an admissible function all of whose critical points are contained in a compact set. The symplectic manifolds manifold $M_{\cR}$ we consider in this paper are all geometrically of finite type.

Let $M$ be a geometrically bounded symplectic manifold such that $c_1(M)=0$. Let us also fix a homotopy class of trivializations of $\Lambda^{top}_\mathbb{C}(TM,J)$ for some compatible almost complex structure $J$. We refer to this choice as the grading of $M$ in the future. To any compact $K\subset M$ we associate  a $\Lambda_{\geq 0}$ module $SH^*_K(M)$, the \emph{symplectic cohomology of $K$ relative to $M$} over the Novikov ring $\Lambda_{\geq 0}$. Before briefly discussing its construction we describe its properties.
\begin{itemize}
    \item $SH^*_K(M)$ carries the structure of a $\mathbb{Z}$-graded $BV$-algebra over $\Lambda_{\geq 0}$.
    \item For any inclusion $K_1\subset K_2$ there a restriction map $SH^*_M(K_2)\to SH^*_M(K_1). $
    \item The restriction maps are contravariantly functorial with respect to nested inclusions and are morphisms of  BV algebras over $\Lambda_{\geq 0}$.
\end{itemize}

We now briefly review the construction of $K\mapsto SH^*_K(M)$ first as a presheaf of $\Lambda_{\geq 0}$-modules and then of BV algbras. For this we mention a class of Floer data $(H,J)$ and paths between them called \emph{regular dissipative Floer data} introduced in \cite{groman} for whom Floer's equation satisfies the necessary $C^0$-estimates for the definition of the Floer complex.  Describing these in detail will take us too far afield. Suffice it to spell out the properties we need and refer to \cite{groman} for details.

\begin{itemize}
    \item There is a sufficient supply of regular dissipative Floer data in the sense that a lower semi-continuous function $H$ on $M$ can be approximated by a monotone sequence of regular dissipative Floer data $(H_i,J_i)$ so that $H_i$ converges to $H$ on compact subsets. 
    \item Any regular dissipative Floer datum $(H,J)$ gives rise to a Floer complex $CF^*(H,J)$.
    \item Given any pair of regular dissipative Floer data $(H_1,J_1)$, $(H_2,J_2)$ with $H_1\leq H_2$ there exists a continuation map $CF^*(H_1,J_1)\to CF^*(H_2,J_2)$. 
    \item The continuation map itself depends on a choice of a regular dissipative path, but any two such choices are connected by a dissipative homotopy which gives rise to a homotopy of the corresponding chain maps. In particular, the homology level continuation map is independent of any choices.
    \item When $H_1=H_2$ the continuation map is an isomorphism on homology. When also $J_1=J_2$, it is the identity.
    
\end{itemize}

To construct $SH^*_M(K)$ in a manifestly choice free manner takes a bit of preparation. Consider sequences $(H_i,J_i)$ of dissipative Floer data such that $$\lim_{i\to\infty }H_{i}(x)=\begin{cases}0,&\quad x\in K\\
\infty,&\quad x\not\in K\end{cases}.$$
We refer to such a sequence, together with choices of dissipative data for continuation maps $CF^*(H_i,J_i)\to CF^*(H_{i+1},J_{i+1}),$ as an \emph{ acceleration datum}. 

From an acceleration datum we obtain a 1-ray of chain complexes
$$
\mathcal{C}:=CF^*(H_1)\to CF^*(H_2)\dots
$$
over $\Lambda_{\geq 0}$ where we drop $J$ from the notation when there is no need to specify it.

We construct a model for the completed homotopy colimit of $\mathcal{C}$ called the \emph{completed telescope} $\widehat{tel}(\cC)$  as follows. Abbreviate $C_i:=CF^*(H,J)$ and let $$tel(\cC(A))=\left(\bigoplus_{i=1}^\infty C_i[q]\right)$$ with $q$ a degree $-1$ variable satisfying $q^2=0$.  The differential is as depicted below \begin{align}\label{teles}
\xymatrix{
C_1\ar@{>}@(ul,ur)^{d }  &C_2\ar@{>}@(ul,ur)^{d} &C_3\ar@{>}@(ul,ur)^{d}\\
C_1[1]\ar@{>}@(dl,dr)_{-d} \ar[u]^{\text{id}}\ar[ur]^{\kappa_1} &C_2[1]\ar@{>}@(dl,dr)_{-d} \ar[u]^{\text{id}}\ar[ur]^{\kappa_2}&\ldots\ar[u]^{\text{id}}_{\ldots} }.
\end{align}
As motivation for the construction, note there is a projection map from $tel(\cC)$ to the ordinary colimit of $\cC$ inducing an isomorphism on homology. $tel(\cC)$ should be seen as a better behaved model for the colimit. The chain complex $tel(\mathcal{C})$ carries the $T$-adic norm, and we define $\widehat{tel}(\cC)$ as the Cauchy completion of $ {tel}(\cC)$ with respect to this norm. 



Given an inclusion $K'\subset K$ and pairs $A, A'$ of dissipative acceleration data for $K,K'$, we say $A\leq A'$  if  $H_i\leq H'_i$ for all $i$. For $A\leq A'$ one can construct a continuation map $\widehat{tel}(\cC)\to \widehat{tel}(\cC')$ by making auxiliary choices of regular monotone dissipative homotopies from $H_i\to H'_i$ and additional homotopies between the concatenations $H_i\to H'_i\to H'_{i+1}$ and  $H_i\to H_{i+1}\to H'_{i+1}$. Let us refer to such a choice as a \emph{continuation datum} from $A$ to $A'$. A continuation datum $h$ from $A$ to $A'$ gives rise to a chain map $\kappa_h:\widehat{tel}(\cC)\to \widehat{tel}(\cC')$ which we refer to the \emph{continuation map}. Details for the construction of this continuation map are in  \cite[\S8.3]{groman} and in \cite[Section 3.3.2]{varolgunes}.

\begin{proposition}\label{prpFloerFunctor}
    \begin{enumerate}
        \item For a pair $A\leq A'$ of acceleration data, and any choice $h_1,h_2$ of continuation data from $A$ to $A'$, the corresponding chain maps $\kappa_{h_1},\kappa_{h_2}$ are homotopic. 
        \item Given a triple $A\leq A'\leq A''$ and continuation data $h_{12},h_{23}, h_{13}$ respectively from $A\to A',A'\to A'',A\to A''$ the chain maps $\kappa_{h_{23}}\circ\kappa_{h_{12}}$ and $\kappa_{h_{13}}$ are chain homotopic.
        \item If $A\leq A'$ are acceleration data for $K$, any continuation datum $h$ from $A$ to $A'$ induces an isomorphism on homology. 
        \item Given any pair $A,A'$ of acceleration data from $K'\subset K$ there is a third datum $A''$ for $K$ such that $A''\leq A$ and $A''\leq A'$
    \end{enumerate}
\end{proposition}
\begin{proof}
  Details for all these statements can be found in \cite[\S8.3]{groman} and \cite[Section 3.3.2]{varolgunes}. 
  \end{proof}
  

A particular consequence is that given ordered acceleration data $A\leq A'$ for $K'\subset K$ \emph{the induced map on homology $H^*(\widehat{tel}(\cC))\to H^*(\widehat{tel}(\cC'))$ is uniquely determined independently of all the auxiliary choices for the continuation datum}. In general, given an inclusion $K\subset K'$ and acceleration data $A,A'$ we obtain a map $H^*(\widehat{tel}(\cC))\to H^*(\widehat{tel}(\cC'))$ by picking a third acceleration datum $A''$ for $K$ and composing the map
$H^*(\widehat{tel}(\cC''))\to H^*(\widehat{tel}(\cC'))$ with the inverse of the isomorphism $H^*(\widehat{tel}(\cC''))\to H^*(\widehat{tel}(\cC))$. \emph{The induced map is again independent of all choices} as an immediate consequence of Proposition \ref{prpFloerFunctor}.

To define  $SH^*_M(K)$ recall that a \emph{contractible groupoid} is a category with a unique morphism between any two objects. Necessarily, this morphism is an isomorphism. For a compact set $K$ let $\mathcal{G}_K$ the contractible groupoid whose object set is the set of all admissible acceleration data for $K$.  We then have a commutative diagram of shape $\mathcal{G}_K$ assigning to an object $A$ the homology of the completed telescope $H^*(\widehat{tel}(\cC(A)))$. The unique arrow from $A$ to $A'$ is the unique 
homology level continuation isomorphism defined in the previous paragraph.

We thus define the $\Lambda$-module $SH^*_M(K)$ for each compact set \emph{uniquely up to unique isomorphism} as 
\begin{equation}\label{EqDefSH}
        SH^*_M(K)=\varinjlim_{A\in \mathcal{G}_K} H^*(\widehat{tel}(\cC(A))).
\end{equation}

Given an inclusion of compact sets $K_1\subset K_2$ we obtain a restriction map $SH^*_M(K_2)\to SH^*_M(K_1)$ by picking acceleration data $A_1,A_2$ for $K_1,K_2$, and considering the corresponding map $H^*(\widehat{tel}(\cC_2))\to H^*(\widehat{tel}(\cC_1))$. This induces a map $SH^*_M(K_2)\to SH^*_M(K_1)$ which is again independent of the auxiliary choice. Let us call this uniquely defined map the \emph{restriction map}. The upshot of the above discussion is the following proposition.
\begin{proposition}
The assignment $\mathcal{SH}$ which assigns to each $K$ the $\Lambda_{\geq 0}$-module $SH^*_M(K)$ and to any inclusion $K_1\subset K_2$ the canonical restriction map $SH^*_M(K_2)\to SH^*_M(K_1)$, is a presheaf over the category of compact sets whose morphisms are the inclusions. 
\end{proposition}
\begin{proof}
    In light of the discussion above, all that remains is to establish contravariant functoriality, which follows immediately from the second item of Proposition \ref{prpFloerFunctor} and unwinding definitions. 
\end{proof}

\begin{remark}

    The discussion so far specifies the presheaf $\mathcal{SH}$ uniquely up to unique isomorphism of pre-sheaves. We typically do not bother going through this procedure of specifying the presheaf $\mathcal{SH}$ up to  \emph{unique} isomorphism. An alternative approach is to make a choice of acceleration datum for each compact set. Proposition \ref{prpFloerFunctor} still guarantees we obtain a pre-sheaf $\cF$ of $\Lambda$-modules by considering the homology level maps between the corresponding homology groups since they are uniquely defined and functorial. \emph{Remarkably, there is no consistency to worry about concerning any choices.} Moreover, by the same token, the sheaf $\cF$ produced in this way is canonically isomorphic to the presheaf $\mathcal{SH}$. We are typically satisfied with specifying the presheaf $\mathcal{SH}$ only up to \emph{canonical} isomorphism. This makes it appear like the construction depends on choices when in fact it does not. This is quite common in Floer theory. The idea of defining an invariant by specifying a diagram parameterized by a contractible groupoid and then making a choice goes back to Conley \cite{Conley1978}. 
\end{remark}
\begin{remark}

     An alternative approach to specifying $\mathcal{SH}$ uniquely up to unique isomorphism and  \emph{which works at the chain level} is taken in \cite[Section 1.5]{AGV}. Namely, we construct for each $K$ a chain complex which realizes the homotopy colimit over all Floer complexes $CF^*(H,J)$ for $H\leq H_K$ via a bar construction. The restriction maps are then just inclusions of chain complexes. 
    
\end{remark}

So far we have discussed the structure of $\Lambda_{\geq 0}$ modules. The BV algebra structure involves  further choices of monotone Floer data  on punctured Riemann surfaces equipped with cylindrical ends and
is constructed in  \cite{tonkonog, AGV}. The resulting structure is again independent of any choices, and, moreover, all the continuation maps preserve the BV algebra structure at the homology level. For more details see \cite[Section 6.2]{locality}.

Passing to the Novikov \emph{field}
$$SH^*_M(K,\Lambda):=SH^*_M(K)\otimes \Lambda,$$ a unit which is respected by the restriction maps is constructed in  \cite{tonkonog}.
From now on we will only consider unital and $\bZ$-graded BV algebras, and we will omit mentioning this. Note that using 
\begin{equation}\label{eq-ray-nov-ring}
\mathcal{C}:=CF^*(H_1;\Lambda_{\geq 0})\to CF^*(H_2;\Lambda_{\geq 0})\ldots 
\end{equation}as the defining Floer one ray, we can equip $tel(\mathcal{C})\otimes \Lambda$ with its natural filtration map. We then have by Lemma \ref{lem-triv-comp} that $$H^*(\widehat{tel(\mathcal{C})\otimes \Lambda})$$ is canonically isomorphic as filtered BV algebras with $SH^*_M(K,\Lambda)$.

We can also define an invariant for open sets. Let $U$ be an open subset of $M$. Let $K_1\subset K_2\subset \ldots $ be an exhaustion of $U$ by compact subsets. We define $$SH^*_M(U,\Lambda):=\varprojlim SH^*_M(K_i,\Lambda).$$ It is easy to see by functoriality of $SH^*_M(\cdot)$ with respect to inclusions that  $SH^*_M(U)$ is independent of the choice of exhaustion in the sense that different exhaustions of $U$ give rise to canonically isomorphic $SH^*_M(U).$ 
\begin{remark}
In the present work we will only mention the case $U=M$. In the case considered here, $SH^0_M(M)$ will be the algebra of entire functions of the mirror constructed in Theorem \ref{thm-mirror-cons} (see part (6) there).
\end{remark}

\subsubsection{The Mayer Vietoris and sheaf property}\label{SecSheafProperty}
Let us start with an elementary geometric lemma.
\begin{lemma}\label{lem-lag-inv}
    Let $M$ be a symplectic manifold, $B$ a smooth manifold and $\pi: M\to B$ a smooth map such that there is an open subset $U\subset B$ with the property that $\pi^{-1}(U)\subset M$ is dense and for every $p\in U,$ $\pi^{-1}(p)$ is a Lagrangian submanifold of $M$. Then, $\pi$ is an involutive map in the sense that for every $f,g\in C^\infty(B),$ the Poisson bracket $\{\pi^*f,\pi^*g\}$ of $\omega$ vanishes identically.
\end{lemma}
\begin{proof}
Since $\{\pi^*f,\pi^*g\}\in C^\infty(M)$, by the density assumption, it suffices to show that $\{\pi^*f,\pi^*g\}(x)=0$ for $\pi(x)\in U.$ Let us fix such an $x$ and let $p:=\pi(x).$ First, we notice that due to the Lagrangian condition, the Hamiltonian vector fields of $\pi^*f$ and $\pi^*g$ at $x$, denoted by $X_{\pi^*f}(x)$ and $X_{\pi^*g}(x)$ are both tangent to $\pi^{-1}(p)$. This is because for all $v\in T_x(\pi^{-1}(p)),$ we have $$\omega(X_{\pi^*f}(x), v)=df_x(v)=v(f)=0,$$ and similarly for $g.$ Again using the Lagrangian condition, we see that $$\omega(X_{\pi^*f}(x),X_{\pi^*g}(x))=0,$$ as desired.
\end{proof}

Here is the main result that we will use.
\begin{theorem}\label{thm-mv-recall}
    Let $(M,\omega)$ be a geometrically bounded symplectic manifold with grading datum, $B$ a smooth manifold, and $\pi: M\to B$ be an involutive, proper and smooth map. Assume that for every compact $P\subset B$, $SH_M^*(\pi^{-1}(P))$ is non-negatively graded. Then, the presheaf $P\mapsto SH_M^0(\pi^{-1}(P))$ on the $G$-topology of compact subsets of $B$ (as in Remark \ref{rem-Hausdorff-G-top}) is in fact  a sheaf.
\end{theorem}
\begin{proof}First of all, by an elementary induction (recall that we are only considering finite covers) we can reduce to showing the sheaf property for two element covers $P_1\cup P_2.$ When $M$ is closed, combining Lemma 2.4.3 and Theorem 1.3.4 of \cite{varolgunes} shows that there is an associated Mayer-Vietoris sequence for relative symplectic cohomology. Using the non-negatively graded assumption, we see that the following is an exact sequence by looking at the first three non-zero terms of the Mayer-Vietoris sequence: 

\begin{align*}
0\to SH_M^0(\pi^{-1}(P_1\cup P_2))\to SH_M^0(\pi^{-1}(P_1))&\oplus SH_M^0(\pi^{-1}(P_2))\\&\to SH_M^0(\pi^{-1}(P_1\cap P_2))
\end{align*} This is precisely the sheaf property. 

The extension of the results of \cite{varolgunes} to the geometrically bounded case is straightforward. Let $D$ be a compact set containing all our compact subsets in its interior. For the proof of the cited Mayer-Vietoris property one constructs the part of the  acceleration datum that lies inside $D$ exactly as in \cite{varolgunes}. This is done so that the data is the same in some neighborhood of the boundary of $D$ for all compact subsets.
Using the results of \cite[Section 7]{groman}, also see \cite[Appendix C]{AGV}, we can extend the data from $D$ to the entire manifold so as to satisfy the dissipativity requirements, and so the extension agrees for all compact subsets. The main point in the proof of the Mayer-Vietoris sequence is an analysis of the topological energy zero (in particular constant) solutions to the continuation map Floer equations. The same analysis goes through in this setup. 
\end{proof}


\subsection{Exact manifolds}\label{ss-exact-review}

Let $(M,\theta)$ be an exact graded symplectic manifold that is geometrically of finite type and assume that $K\subset M$ is compact. Then we can work over the base commutative ring $\mathbb{F}$ (coefficient ring for the elements of $\Lambda$), e.g., $\mathbb{F}=\mathbb{Z}$. We consider $\mathbb{F}$ as a trivially valued ring.  

We choose a dissipative acceleration datum for $K$ whose underlying Hamiltonians $H_i$ each have finitely many $1$-periodic orbits. Define a Floer $1$-ray over $\mathbb{F}$ by signed counts without weights: $$\mathcal{C}_{\mathbb{F}}:=CF^*(H_1;\mathbb{F})\to CF^*(H_2;\mathbb{F})\ldots$$ and obtain $$SH^*_{M,\theta}(K;\mathbb{F}):=H^*(\widehat{tel}(\mathcal{C})),$$which is a BV algebra over $\mathbb{F}$. Here the filtration map on the telescope comes from taking the actions of generators $$A(\gamma)=\int \gamma^*\theta+\int H(\gamma(t)) dt,$$ and equipping the telescope with the min-filtration (by taking the minimum of the filtration values of the basis elements in the linear combination). The differential does not decrease the filtration values, making $\widehat{tel}(\mathcal{C})$ a degreewise complete filtered chain complex. We obtain a filtration map on $SH^*_{M,\theta}(K;\mathbb{F})$ by taking supremum of chain representatives. The operations do not decrease the filtration map in the appropriate sense. 

\begin{theorem}\label{SH-exact}
$SH^*_{M,\theta}(K;\mathbb{F})$ is well-defined as a filtered graded Abelian group in the sense that another choice of dissipative acceleration data gives rise to a filtered graded Abelian group which is canonically isomorphic to it in a way that preserves filtration maps.
\end{theorem}

\begin{proof}
Given a pair of a acceleration data $A_1,A_2$ it is shown in \cite[\S8.3]{groman} there is a third acceleration datum $A_0$ which is dominated by both and such that the monotone continuation map, which is filtration preserving and is defined up to contractible choice, induces an isomorphism on all the truncated homologies. The claim now follows from Lemma \ref{lem-filt-qis}.
\end{proof}

\begin{remark}
This statement would not be true if we did not complete the telescope. This can be seen by comparing $S$-shaped and $J$-shaped acceleration data for Liouville manifolds.

\end{remark}

In many situations relevant to this paper (see Proposition \ref{prop-action}), we observe that the acceleration data can be chosen to satisfy the extra property that the actions of the $1$-periodic orbits are uniformly bounded above (e.g., non-positive). The completion to the telescope of $\mathcal{C}_{\mathbb{F}}$ does nothing in this case! Let us call this the \emph{bounded action} property. The following is immediate.

\begin{lemma}
Under the bounded action property, the induced filtration map on $SH^*_{M,\theta}(K;\mathbb{F})$ is complete and Hausdorff.
\end{lemma}
\begin{proof}
    Since the filtration values of all possible representatives of a homology class are bounded above, the Hausdorff property follows. By the same token, the distance between two different homology classes is bounded below by a positive uniform constant. Therefore,  the only Cauchy sequences are the ones that are eventually constant, which automatically converge.
\end{proof}

Here is the main application of boundary depth considerations for our purposes. Define $\mathcal{C}:=CF^*(H_1;\Lambda_{\geq 0})\to CF^*(H_2;\Lambda_{\geq 0})\ldots$ as in Equation \eqref{eq-ray-nov-ring}. We can equip $tel(\mathcal{C}_{\mathbb{F}})\otimes \Lambda$ with the tensor product filtration map, which means that the filtration value of an arbitrary element $\sum \alpha_n\otimes {\gamma_n}$ is the minimum of the sum of the valuations of $\alpha_n$ and ${\gamma_n}.$ We have a filtered isomorphism of chain complexes 
$$tel(\mathcal{C}_{\mathbb{F}})\otimes \Lambda\to tel(\mathcal{C})\otimes_{\Lambda_{\geq 0}} \Lambda,$$ given by sending $$\gamma\otimes 1\mapsto \gamma \otimes T^{A(\gamma)}.$$ We call this the \emph{rescaling map}. Observe the rescaling map commutes on the nose with all the continuation maps. It thus induces a homology level map 
\begin{equation}\label{eqRescaling}
 SH^i_{M,\theta}(K;\mathbb{F})\hat{\otimes} \Lambda\to SH_M^i(K;\Lambda)
 \end{equation}
which moreover commutes with all restriction maps. That is, given an inclusion $K_1\subset K_2$ we have a commutative diagram
$$
\xymatrix{
  SH^*_{M,\theta}(K_2;\mathbb{F})\hat{\otimes}\Lambda\ar[d]\ar[r]& SH_M^*(K_2,\Lambda)\ar[d]\\
 SH^*_{M,\theta}(K_1;\mathbb{F})\hat{\otimes}\Lambda\ar[r]&SH_M^*(K_1,\Lambda).
}
$$

\begin{proposition}\label{prpBoundaryDepthComparison}
Let us assume that $tel(\mathcal{C}_{\mathbb{F}})$ has finite boundary depth at degree $i$.
 Then the the rescaling map is a filtered isomorphism.
\end{proposition}
\begin{proof}

Let us define $(C^*,d^*)$ to be the $\Lambda$-chain complex $\widehat{tel}(\mathcal{C}_{\mathbb{F}})\otimes \Lambda$ with the tensor product filtration map and $(\widehat{C}^*,\widehat{d}^*)$ be the degree-wise completion.  More explicitly, we can write any homogeneous element $x\in \widehat{C}^i $ as an infinite sum $\sum_{n=0}^\infty \alpha_n\otimes T^{a_n}$ with $\alpha_n\in\widehat{tel}(\mathcal{C}_{\mathbb{F}})^i$ and $a_n\in\mathbb{R}$ pairwise distinct such that $A(\alpha_n)+a_n\to \infty$. The valuation of such an element is the minimum of $A(\alpha_n)+a_n$. 

We have that $SH_M^i(K;\Lambda)$ is isomorphic to $$ker (\widehat{d}_i)/im (\widehat{d}_{i-1}),$$ and $SH^i_{M,\theta}(K;\mathbb{F})\hat{\otimes} \Lambda$ is nothing but $$\widehat{ker (d_i)/im (d_{i-1})}.$$ Our goal is to use Lemma \ref{lem-completion-quotient} to finish the proof.

Note that $ker (\widehat{d}_i)$ and $\widehat{ker (d_i)}$ are both canonically identified with the subset of $\widehat{C}^i$ consisting of elements $\sum_{n=0}^\infty \alpha_n\otimes T^{a_n}$ with $\alpha_n$ closed in $\widehat{tel}(\mathcal{C}_{\mathbb{F}})^i$. Hence all that is left to show is that  $im (\widehat{d}_{i-1})$ is the closure of $im ({d}_{i-1})$ inside $\widehat{C}^i $ (and hence inside $\widehat{ker (d_i)}$).

Continuity immediately implies that $im (\widehat{d}_{i-1})$ is in the closure. Conversely, we can write every limit point as $\sum_{n=0}^\infty \alpha_n\otimes T^{a_n}$ with each  $\alpha_n$ exact. The finite boundary depth assumption finishes the proof as it lets us construct a primitive with respect to $\widehat{d}_i$.
\end{proof}

\subsection{Liouville manifolds}\label{ss-liouville-review}

We now discuss the relation between relative $SH$ and the symplectic cohomology introduced by Viterbo mainly for Liouville domains and their completions. Here we refer in particular to the non-quantitative approach emphasised in Section 3e) of \cite{seidelbiased}\footnote{conventions in this reference are slightly different but we believe this will not cause confusion}. 

Let $(M,\theta)$ be a finite type complete Liouville manifold and denote by $V$ the Liouville vector field. Let $W\subset M$ be a compact domain with smooth boundary such  that the Liouville vector field $V$ is outward pointing on $\partial W$ and $V$ is non-zero outside of $W$. Call such a domain \emph{admissible}. Denoting by $Sk_\theta$ the skeleton of $M$ with respect to $\theta$, $\partial W$ and $V$ give rise to an exponentiated Liouville coordinate $$\rho: M\setminus Sk_\theta\to \mathbb{R}_{>0},$$ which is equal to $1$ on $\partial W$ and satisfies $V\cdot \rho=\rho.$   Let us denote by $\cL(W)$ the class of Hamiltonians on $M$ which  outside of a compact set are linear functions of $\rho$. Define a pre-order on $\cL(W)$ by $H_i\preceq H_{i+1}$ if there is a constant $C$ for which $H_1\leq H_2+C$.  

We then get an invariant $SH^*(M;\mathbb{F})$ which we refer to as \emph{Viterbo $SH$}. It is defined by considering  the \emph{non-completed} colimit of the Floer complexes for any  sequence of Hamiltonians $H_i\in \cL(W)$ with the slope going to infinity. Since the sequence $H_i$ is only required to satisfy that $H_{i+1}-H_i$ is bounded from below, the Viterbo $SH$ contains no quantitative information about the domain $W$. 

At first sight there is still some dependence on the domain $W$ because of the involvement of $\cL(W)$. However, as pointed out  in \cite{seidelbiased}, the Viterbo symplectic cohomologies for different admissible subdomains $W$ are canonically isomorphic. The reason is that we can squeeze a sequence in $\cL(W_1)$ into any sequence in $\cL(W_2)$ and vice versa. Moreover, given functions $H_1,H_2$ which at infinity are linear functions of  $\rho_1,\rho_2$ respectively, and satisfying $H_1\leq H_2+C$ there are well defined continuation maps between them \footnote{The well definedness relies on a maximum principle developed in \cite{seidelbiased}. Alternatively, one can rely on \cite{groman} that $H_1,H_2$ are dissipative. There are thus well defined continuation maps which agree with the ones constructed relying on maximum principles.}.

It is convenient to push this discussion somewhat further. Denote by  $\cL_a(W)$ the set of dissipative functions $H$ on $M$  for which there exists a $c$ so that $\frac1{c}\rho<H<c\rho$ outside of a compact set. The subscript $a$ here stands for asymptotic linearity.  Since for $W_1,W_2$ we have constants  $\frac1{c}\rho_1<\rho_2<c\rho_1$ there is actually an equality $\cL_a(W_1)=\cL_a(W_2)$. We thus drop the dependence on $W$ from the notation and write $\cL=\cL_a(W)$.  

Then $\cL(W)\subset \cL$ is $\preceq$-cofinal.  That is, for any element in $H\in \cL$ there is an element $H'\in \cL(W)$ such that $H\preceq H'$. It follows that \emph{the Viterbo $SH$ can be computed with any $\preceq$-cofinal sequence in $\cL$. }

\begin{remark}
The Viterbo $SH$  is a global invariant of $M$, but it is not naturally endowed with a norm. The interpretation of the Viterbo $SH$ for exact symplectic cluster manifolds is as the functions of the algebraic mirror defined over $\mathbb{F}$.

As observed in \cite{groman} the Viterbo $SH$ can also be defined for
$M$ which is not necessarily the completion of a Liouville domain (not even exact) provided one specifies an appropriate growth
condition at infinity akin to the set $\cL$. In the case we are considering, the integral affine structure
can be used to specify such a condition, namely, piece-wise linearity in integral affine coordinates. In particular this specifies a set of global algebraic functions on the mirror over the Novikov field.
\end{remark}

We now compare Viterbo $SH$ to relative $SH$ of an arbitrary compact domain $W\subset M$.
The universal mapping property induces a map
\begin{equation}\label{eqSHComparison}
SH^*(M;\mathbb{F})\to SH^*_{M,\theta}(W;\mathbb{F})   
\end{equation}
by  considering for the left hand side any cofinal sequence of Floer data $(H_i,J_i)$ that is both linear at infinity and $<0$ on $W$.  The choice of cofinal sequence has no affect on the map by standard arguments. 
\begin{lemma}\label{lmComparisonNaturality}
    For any inclusion $W'\subset W$ of compact sets the corresponding maps from \eqref{eqSHComparison} fit into a commutative triangle
$$
\xymatrix{
SH^*(M;\mathbb{F}) \ar[r] \ar[rd] &  SH^*_{M,\theta}(W;\mathbb{F})\ar[d] \\
& SH^*_{M,\theta}(W';\mathbb{F}) 
}
$$
\end{lemma}
\begin{proof}
    This is immediate from the commutativity of 
    $$
\xymatrix{
HF^*(H_i,J_i;\mathbb{F}) \ar[r] \ar[rd] &  SH^*_{M,\theta}(W;\mathbb{F})\ar[d] \\
& SH^*_{M,\theta}(W';\mathbb{F}) 
}
$$
\end{proof}

\begin{theorem}\label{tmSHComparison}

Suppose $W$ is the intersection of a descending sequence of admissible subdomains $W_i$. Then the map of \eqref{eqSHComparison} is an isomorphism of (non-normed) $BV$ algebras. 
 In particular, forgetting norms, we have, for any inclusion $W'\subset W$ of such domains, an isomorphism $SH^*_{M,\theta}(W;\mathbb{F})\to SH^*_{M,\theta}(W';\mathbb{F})$ as BV algebras.

\end{theorem}
\begin{proof}
 In computing relative $SH$ of $W$  we can use acceleration data of Viterbo type. Namely, if $W=\cap_{i=1}^{\infty}W_i$ the underlying $i$th Hamiltonian is a linear function of $\rho_i$ near infinity, a convex function of $\rho_i$ on $M\setminus W_i$ and $C^2$ small on $W_i$. Note that the sequence of Hamiltonians of this acceleration datum is $\preceq$ cofinal in $\cL$.

 For such acceleration data one can immediately see from the Viterbo $y$-intercept trick that the action functional really only takes negative values.  This means that the filtration map can only take non-positive values. Thus completion has no effect and the map \eqref{eqSHComparison} can be lifted to an isomorphism of complexes for a particular choice of acceleration data.  
\end{proof}

\begin{remark} These isomorphisms are a special feature of the exact case and even then only hold over a trivially valued field. For example the isomorphism in the second part does not have to be bounded with respect to the filtration.  Thus if we base change to a non-trivially valued field the completions will be different and the canonical restriction map will no longer be an isomorphism.    
\end{remark}
There are many techniques to compute the Viterbo symplectic cohomology as a BV algebra. Most important among them is Viterbo's isomorphism between symplectic cohomology and string homology in case of the cotangent bundle of a smooth manifold. The latter is fully computable for $T^n$, which we will use below (see Theorem \ref{thm-BV-torus}). On the other hand what is relevant for us is the completed version (assuming finite boundary depth as above) $$SH^*_{M,\theta}(W;\mathbb{F})\hat{\otimes} \Lambda= SH_M^*(W;\Lambda)$$ as we are planning to use our locality theorem in more global situations. Hence, it is important to be able to explicitly describe the norm on $SH^*(M;\mathbb{F})$ given by $W$ as above.

\section{Analysis of the local model for the regular fibers}\label{s-noray}

We denote the coordinate functions of $\bR^n$ by $q_1,\ldots q_n$ and the corresponding dual linear coordinates on $(\bR^n)^{\vee}$ by $p_1,\ldots p_n$. $\mathbb{Z}^n$ and $(\mathbb{Z}^n)^{\vee}$ are the  standard integer lattices in $\mathbb{R}^n$ and $(\mathbb{R}^n)^{\vee}.$  Let us denote the smooth manifold underlying $\bR^n$ by $B_0$ for clarity. Let $v$ be the vector field $\sum q_i\frac{\partial}{\partial q_i}$ in $B_0$. Note that this vector field is invariant under the action of linear isomorphisms of $
\bR^n$.

Let $M:=(\bR^n)^{\vee}/(\mathbb{Z}^n)^{\vee}\times \bR^n$ be equipped with the Liouville structure $\theta:=\sum -q_i dp_i$ and let $\omega:=d\theta$. Let $\pi:M\to B_0$ be the canonical projection and equip $\pi$ with the induced (trivial) horizontal subbundle. The horizontal lift $V$ of $v$ is the Liouville vector field. 

Let us also trivialize the canonical bundle $TM_\bC^{\wedge n}$ defined through the compatible almost complex structure $J\frac{\partial}{\partial q_i}=-\frac{\partial}{\partial p_i}$ with the trivialization $$\frac{\partial}{\partial q_1}\wedge\ldots\wedge \frac{\partial}{\partial q_n}.$$

We will now state a special case of Viterbo's theorem, Theorem 1.1 of \cite[Chapter 12]{latschev}. Note that $M$ is symplectomorphic to the cotangent bundle of $T^n$ $$\left(T^*((\bR^n)^{\vee}/(\mathbb{Z}^n)^{\vee})=(\bR^n)^{\vee}/(\mathbb{Z}^n)^{\vee}\times \bR^n, \sum dp_idq_i=-\omega\right) $$ via the map that negates the $q$ coordinates. Using the computation of the Chas-Sullivan string homology BV-algebra from \cite[Section 6.2]{tonkonogstring}, we know that the symplectic cohomology of $T^*T^n$ is isomorphic to 
\begin{equation}\label{eqTorusSH}
\mathbb{F}[H_1(T^n, \mathbb{Z})]\otimes \Lambda^*(H^1(T^n, \mathbb{Z}))
\end{equation}
with the BV operator given by taking the interior product of an element of $\Lambda^*(H^1(T^n, \mathbb{Z}))$ with an element of $H_1(T^n, \mathbb{Z}).$ Here $\mathbb{F}[\cdot]$ refers to taking the group algebra and $\Lambda^*(\cdot)$ denotes the exterior algebra functor.

 Note that Hamiltonian Floer chain complexes carry a grading by the first homology of the symplectic manifold obtained by defining the summand corresponding to $\alpha\in H_1(M;\mathbb{Z})$ to be the span of all the $1$-periodic orbits whose homology class is $\alpha.$  This induces a $H_1(M;\mathbb{Z})$ grading on symplectic cohomology and all other invariants derived from Hamiltonian Floer theory as well.

\begin{theorem}\label{thm-BV-torus}
The Viterbo symplectic cohomology $BV$-algebra $SH^*(M;\mathbb{F})$ is isomorphic to $$\mathcal{A}^*:=\mathbb{F}[(\mathbb{Z}^n)^{\vee}]\otimes \Lambda^*(\mathbb{Z}^n),$$ with its tensor product graded algebra structure and the $BV$ operator: $$\Delta(z^{\alpha}\otimes \beta)=  z^{\alpha}\otimes\iota_{\alpha}\beta,$$ for any  $\alpha\in (\mathbb{Z}^n)^{\vee}$ and $\beta\in \Lambda^*(\mathbb{Z}^n)$.

Moreover, if we identify $SH^*(M;\mathbb{F})$ with $\mathcal{A}^*$ with this isomorphism,  using the identification $H_1(M;\mathbb{Z})=(\mathbb{Z}^n)^\vee$ and the extra $H_1(M;\mathbb{Z})$-grading of $SH^*(M;\mathbb{F})$ obtained using the homology classes of orbits, we have the following properties: \begin{itemize} \item the homogeneous summand of $SH^0(M;\mathbb{F})$ corresponding to the homology class $\alpha\in (\mathbb{Z}^n)^\vee$ is generated by $z^{\alpha}\otimes 1$ as an $\mathbb{F}$-module.
\item  the homogeneous summand $SH^*_0(M;\mathbb{F})$ of $SH^*(M;\mathbb{F})$ corresponding to the homology class $0\in (\mathbb{Z}^n)^\vee$ is generated by $1\otimes e_i, i=1,\ldots,n$ as an $\mathbb{F}$-algebra.
\item The homogeneous summand of $SH^*(M;\mathbb{F})$ corresponding to the homology class $\alpha\in (\mathbb{Z}^n)^\vee$ is generated by $z^{\alpha}\otimes 1$ as an $SH^*_0(M;\mathbb{F})$-module.

\end{itemize}
\end{theorem}

%

Recall that at the end of Section \ref{ss-ana-no-ray}, for an admissible convex polytope $P\subset B_0$, we defined the $BV$-algebra $\mathcal{A}^*(P)$. 

\begin{proposition}\label{prop-BV-torus-B} The natural map of $\Lambda$-algebras $$\mathcal{A}^*\otimes \Lambda\to \mathcal{A}^*(P)$$ given by $$z^{e_i^\vee}\otimes 1\otimes 1\to y_i^{-1}\text{ and } 1\otimes e_j\otimes 1\to y_j\frac{\partial}{\partial y_j}$$ respects the BV operator. Moreover, it is injective and has dense image. \end{proposition}

\begin{proof}
We need to check what happens to \begin{equation}\label{eq-BV-elt} z^{\alpha}\otimes e_{j(1)}\wedge\ldots \wedge\ldots  e_{j(k)}\otimes 1,\end{equation} for $\alpha\in (\mathbb{Z}^n)^\vee$ and $j:[k]\to [n]$ strictly order preserving. It's image under BV operator is $$z^{\alpha}\otimes\sum_{l=1}^k(-1)^{l-1}\alpha(e_{j(l)})\cdot e_{j(1)}\wedge\ldots  \widehat{e_{j(l)}}\wedge\ldots  e_{j(k)}\otimes 1.$$ Applying the natural map in turn, we get 
$$y^{-\alpha}\sum_{l=1}^k(-1)^{l+1}\alpha(e_{j(l)})\cdot y^{e_{j(1)}^\vee+\ldots+\widehat{e_{j(l)}^\vee}+\ldots +e_{j(k)}^\vee} \frac{\partial}{\partial y_{j(1)}}\ldots\widehat{\frac{\partial}{\partial y_{j(l)}}}\ldots \frac{\partial}{\partial y_{j(k)}},$$

Under the natural map \eqref{eq-BV-elt} is sent to $$y^{-\alpha+e_{j(1)}^\vee+\ldots +e_{j(k)}^\vee}\frac{\partial}{\partial y_{j(1)}}\ldots \frac{\partial}{\partial y_{j(k)}}.$$ The inner product with $\Omega_0$ returns $$(-1)^{s(j)}y^{-\alpha-e_{j'(1)}^\vee-\ldots -e_{j'(n-k)}^\vee}dy_{j'(1)}\ldots dy_{j'(n-k)},$$ where $$s(j):=\sum_{l=1}^kj(l)-l,$$ and $j':[n-k]\to [n]$ is the order preserving isomorphism to the complement of the image of $j$. 

Exterior differential of this is $$(-1)^{s(j)}y^{-\alpha-e_{j'(1)}^\vee-\ldots -e_{j'(n-k)}^\vee}\sum_{l=1}^k(-1)^{j(l)-l}(-\alpha(e_{j(l)}))\cdot y^{-e_{j(l)}^\vee}dy_{j'(1)}\ldots dy_{j(l)}\ldots dy_{j'(n-k)}.$$ This then goes back to $$\sum_{l=1}^k(-1)^{s(j)+s(j\setminus l)+j(l)+l+1}y^{-\alpha+e_{j(1)}^\vee+\ldots+\widehat{e_{j(l)}^\vee}+\ldots +e_{j(k)}^\vee}\alpha(e_{j(l)})\cdot \frac{\partial}{\partial y_{j(1)}}\ldots\widehat{\frac{\partial}{\partial y_{j(l)}}}\ldots \frac{\partial}{\partial y_{j(k)}},$$ where $j\setminus l$ denotes the map $[k-1]\to [k]\to [n]$ with the first map isomorphism on to the complement of $\{l\}$.

It is easy to see that $$s(j)-s(j\setminus l)=j(l)-l+(k-l).$$ Applying the final sign $(-1)^k$ we obtain the correct result.

For the last sentence note the definition of $\cA^0(P)$ is the completion of $\cA\otimes \Lambda$ with respect to an appropriate norm, and that $\cA^*(P)$, the algebra of continuous polyderivations of $\cA^0(P)$ is generated as an algebra over $\cA^0(P)$, by the vector fields $y_j\frac{\partial}{\partial y_j}$. The injectivity and density now follow.
\end{proof}
%
%

Let us recall some generalities on integrable Hamiltonian flows on $M$. Note that the vertical tangent space of $\pi$ at any $x\in M$, is identified with $T^*_{\pi(x)}B_0=(\mathbb{R}^n)^{\vee}$ by using $\omega$. The Hamiltonian vector field of any $H=h\circ \pi$ at $x$ (which is vertical) is equal to $dh_{\pi(x)}$ under this identification. Hence the Hamiltonian flow of such an $H$ preserves the fibers of $\pi$, and inside each torus the time $t$ map of the flow is given by translating the torus $\pi^{-1}(b)= (\mathbb{R}^n)^{\vee}/(\mathbb{Z}^n)^{\vee}$ by the vector $tdh_{\pi_b}.$ In particular, the time $1$ periodic orbits of the Hamiltonian flow of $H$ correspond to the $b\in B_0$ such that $dh_{b}$ is integral.

Let us denote the Hamiltonian flow of $H=h\circ \pi$ by $\Phi^t$. The flow $\Phi^t$ does not preserve the horizontal subspaces. The image of a horizontal vector $v\in T_xM$ under $d\Phi^t$ is given by the sum of the horizontal lift of $\pi_*v$ to $\Phi_t(x)$ and the vertical vector corresponding to $$t\nabla_{\pi_*v}dh\in T_{\pi(x)}^*B,$$ where the bilinear form $\nabla dh$ on $T_bB$ has matrix $$\left(\frac{\partial^2h}{\partial q_i\partial q_j}(b)\right)$$ with respect to basis $\frac{\partial}{\partial q_1},\ldots, \frac{\partial}{\partial q_n}.$

Hence the Morse-Bott non-degeneracy of a time-$1$ orbit of $\Phi^t$ corresponding to a point $b\in B$ with integral $dh_b$ is equivalent to the non-degeneracy of the symmetric bilinear form $(\nabla dh)_b$. The Maslov index \cite{robbin} of this path is equal to $-s/2$, where $s$ is  the number of positive eigenvalues minus the number of negative eigenvalues of $(\nabla dh)_b$ using the Normalization property listed in page 17 of \cite{robbin}. Our convention is to add $n$ to the Maslov index to define the degree in our Floer complexes. After controlled perturbations to $H$ that make it non-degenerate, such a $b$ therefore contributes generators to the Hamiltonian Floer complex whose degrees are in the interval $[n-\frac{s}{2}, 3n-\frac{s}{2}]$. 

The following generalization of Viterbo's $y$-intercept formula for actions will be useful. We state it more generally than needed here using the same notation.

\begin{proposition}\label{prop-action}
Let $M$ be an exact symplectic manifold and $\rho: M\to \mathbb{R}^k$ be an  smooth map whose components pairwise Poisson commmute $\{\rho_i,\rho_j\}=0,$ for $i,j=1,\ldots k,$ i.e., $\rho$ is an involutive map as defined in Lemma \ref{lem-lag-inv}. Assume that the Liouville vector field $V$ on $M$ and the Euler vector field $v$ in the base $\mathbb{R}^k$ are $\rho$-related. Let $h:\mathbb{R}^k\to \mathbb{R}$ be smooth and $H:=h\circ \rho$. Then the action of a $1$-periodic orbit $\gamma: S^1\to M_k$ of $H$ living over $b\in \mathbb{R}^k$ is $$\int \gamma^*\theta +\int \gamma^*H=h(b)+\int_{0}^{1} \theta(X_H(\gamma(t))dt.$$ Using $$\theta(X_H)=\omega(V,X_H)=-dH(V)=-dh(v),$$ we obtain precisely that the action is given by the height axis intercept of the tangent space to the graph of $h$ over $b$.\end{proposition}

For the next proposition we need a lemma concerning dissipativity.
\begin{lemma}\label{lmDssipCrit}
    If $M$ is geometrically bounded and $(H,J)$ is a Floer datum on $M$ such that $J$ is geometrically bounded, $H$ is Lipschitz with respect to $J$ and there is a compact set $K$ and a $\delta>0$ such that for all $x\in M\setminus K$ we have $d(\psi(x),x)\geq\delta$ where $\psi$ is the time $1$ flow of $H$, then $(H,J)$ is dissipative. 
\end{lemma}
\begin{proof}
    This is the combination of Lemma 5.11 and Corollary 6.19 of \cite{groman}.
\end{proof}

For $\alpha\in (\mathbb{Z}^n)^{\vee}$ and $Q\subset B_0^n$ compact, we define 
$$
F(Q,\alpha):=\{x\in Q|\alpha(x)=\max_{y\in Q}\alpha(y)\}.
$$ 

\begin{proposition}\label{prop-BV-torus-A}
Let $P\subset B_0^n$ be an admissible convex polytope containing the origin. Then,

\begin{enumerate}
\item There exists an acceleration datum for $\pi^{-1}(P)\subset M$ leading to the Floer $1$-ray $\mathcal{C}_{\mathbb{F}}:=CF^*(H_1;\mathbb{F})\to CF^*(H_2;\mathbb{F})\ldots$ such that the bounded action property is satisfied and $tel(\mathcal{C}_{\mathbb{F}})$ has finite boundary depth at all degrees.
\item The map obtained by Theorem \ref{tmSHComparison} and Theorem \ref{thm-BV-torus}  $$\mathbb{F}[(\mathbb{Z}^n)^{\vee}]\otimes \Lambda^*(\mathbb{Z}^n)\to SH^*_{M,\theta}(\pi^{-1}(P),\mathbb{F}) $$ is so that the induced filtration map (via Theorem \ref{SH-exact}) on $\mathbb{F}[(\mathbb{Z}^n)^{\vee}]\otimes \Lambda^*(\mathbb{Z}^n)$ is $$\sum a_{i}z^{\alpha_{i}}\otimes \beta_i\mapsto  \min_{i}-\alpha_i(F(P,\alpha_i))=\min_{i}\min_{p\in P}-\alpha_i(p).$$

\end{enumerate}

\end{proposition}

\begin{proof}
Since $P$ is rational and contains the origin, $P$ can be written as the intersection of rational half-spaces: $$\{p\mid\nu_i(p)\geq b_i\}\subset B_0^n,$$where $\nu_i$ are primitive elements of $(\mathbb{Z}^n)^{\vee}$ and $b_i\leq 0$ for $i=1,\ldots k$. Note we allow  $b_i=0$ to include the possibility of degenerate polygons. 

We also define the admissible convex polygons $$P_\eta:=\bigcap_{i=1}^k\{\nu_i(p)\geq b_i-\eta\}\subset B_0^n$$ for $\eta>0$, which contain the origin in their interior. Without loss of generality, we can assume that $F(P_\eta,\nu_i)$ is a codimension $1$ face of $P_\eta$ for all $\eta>0.$

Let us define for $\eta>0$ the functions $$
l_{i,\eta}(x)=1-\frac{\nu_i(x)}{b_i-\eta},
$$ 
for $x\in B_0^n$. Then $P_{\eta}$ is defined by the equations $l_{i,\eta}\leq 0$. For any $\lambda\in\mathbb{R}_{>0},$ we define $$h_{{\lambda},\eta}:= \max{(0, \lambda l_{1,\eta},\ldots, \lambda l_{k,\eta})}$$and also let $\nu_{{\lambda},\eta}^-$ be an arbitrary function $\mathbb{R}^n\to (\mathbb{R}^n)^\vee$ equal to $dh_{{\lambda},\eta}$ whenever the latter makes sense. Note $P_{\eta}$ is defined by $h_{{\lambda},\eta}\leq 0$.


\begin{figure}
\includegraphics[width=0.6\textwidth]{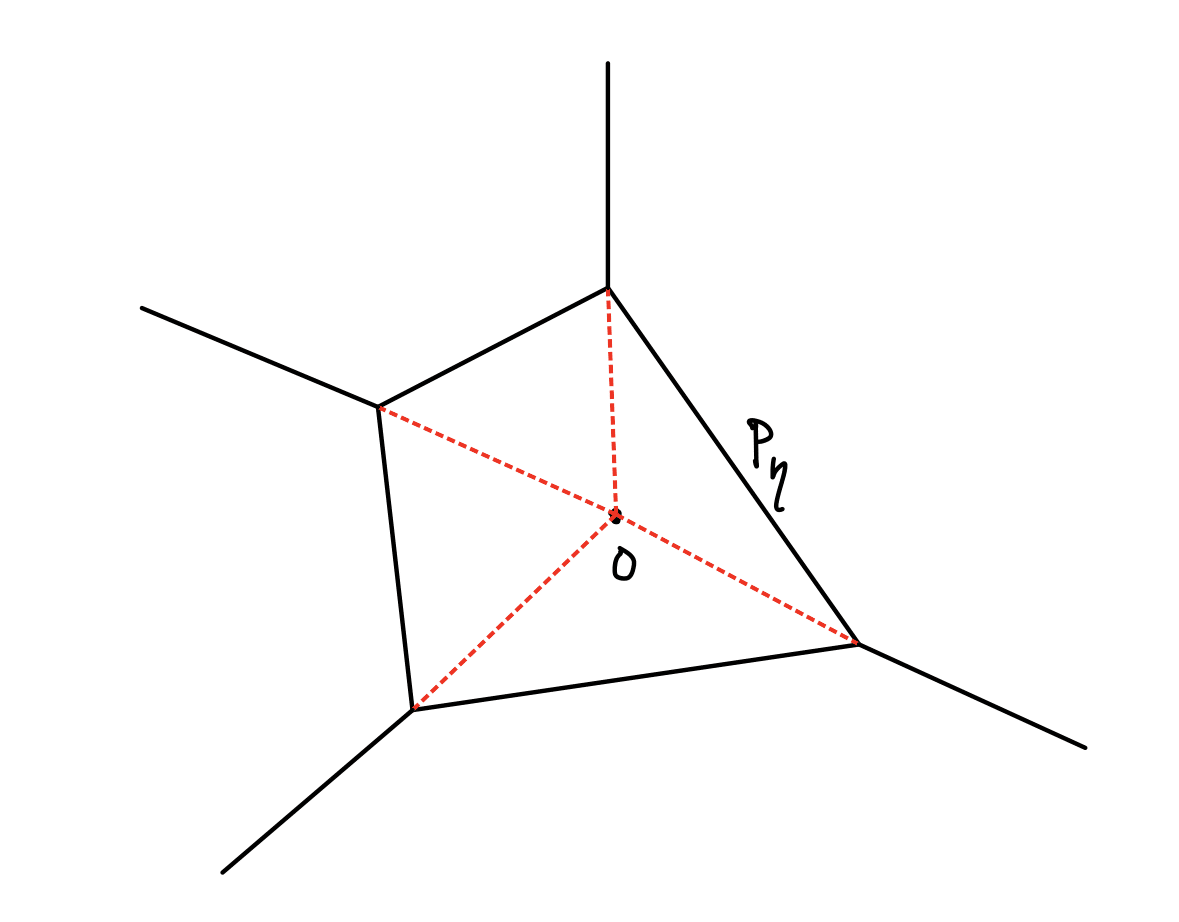}
\caption{An example in two dimensions. The subset $D_\eta$ is the union of the solid line segments that form the boundary of $P_\eta$ and the solid part of the rays emanating from the origin towards the vertices of $P_\eta.$}
\label{fig-deta}
\end{figure}

Let $\rho_{\epsilon}:\mathbb{R}^n\to \mathbb{R}$, $\epsilon\geq 0$ be a choice of mollifiers which approximate the Dirac delta function such that $\rho_{\epsilon}$ is supported inside the $\epsilon$-ball centered at $0$. We also consider a parameter $\delta\geq 0$ and define  $$h_{\lambda,\eta,\epsilon,\delta}(p):= \rho_{\epsilon}\star h_{\lambda,\eta}(p)-\delta=-\delta+\int_{\mathbb{R}^n}\rho_{\epsilon}(p-y)h_{\lambda,\eta}(y)dy.$$ 
For each $\delta\geq 0$, the $0$ sublevel set of $h_{\lambda,\eta,\epsilon,\delta}$ as $\epsilon\to 0$ converges to $\{h_{\lambda,\eta}\leq\delta\}$ which contains $P_{\eta}$.
Note that we have $$(dh_{\lambda,\eta,\epsilon,\delta})_p=\int_{\mathbb{R}^n}\rho_{\epsilon}(p-y)\nu_{\lambda,\eta}^-(y)dy.$$


For each $\lambda$ consider the map $\mathbb{R}^n\to(\mathbb{R}^n)^\vee$ which sends $$p\mapsto (-dh_{\lambda,\eta,\epsilon,\delta})_p.$$ Denote the image by $I_{\lambda}$. Note that every point that is more than $\epsilon$ away from $P_\eta$ is mapped to the boundary of $I_\lambda.$ Indeed, outside the $\epsilon$ neighborhood of $P_{\eta}$, the function $h_{\lambda,\eta,\epsilon,\delta}$ is homogeneous of degree $1$.

Let us denote by $X_{\alpha}$ the preimage of $\alpha\in (\mathbb{Z}^n)^\vee\cap I_{\lambda}$ under the map $$p\mapsto (-dh_{\lambda,\eta,\epsilon,\delta})_p.$$ Then all the $1$-periodic orbits of $\pi^*h_{\lambda,\eta,\epsilon,\delta}$ in homology class $\alpha\in (\mathbb{Z}^n)^\vee\cap I_{\lambda}$ lie inside $\pi^{-1}(X_\alpha)$. By Proposition \ref{prop-action}, the actions $\mathcal{A}_{\lambda,\eta,\epsilon,\delta}(\alpha)$ of these $1$-periodic orbits are given precisely by the value of the Legendre transform $\mathcal{L}(h_{\lambda,\eta,\epsilon,\delta})$ at $\alpha$.

We can also talk about the discrete Legendre transform of $h_{\lambda,\eta}$. Denoting this also by $\mathcal{L}(h_{\lambda,\eta}),$ we have that $$\mathcal{L}(h_{\lambda,\eta,\epsilon,\delta})\to \mathcal{L}(h_{\lambda,\eta})$$pointwise as $\epsilon,\delta\to 0$. This implies by looking at the point $\alpha$ that $$\mathcal{A}_{\lambda,\eta,\epsilon,\delta}(\alpha)\to -\alpha(F(P_\eta,\alpha))$$ as $\epsilon,\delta\to 0$.

Let us call $\lambda\in\mathbb{R}_{>0}$ sufficiently irrational, if $\partial I_{\lambda}$ does not contain any integral points. The set of such $\lambda$ is dense in $\mathbb{R}$. Indeed the family $\partial I_{\lambda}=\lambda \partial I_1$ foliates the complement of a compact set in $(\mathbb{R}^n)^\vee$. If $\lambda$ is sufficiently irrational (which we assume from now on) then by compactness there is a an $\epsilon'>0$, depending on $\lambda$, so the distance of any point on $\partial I_{\lambda}$ is at least $\epsilon'$ away from any integer. This has the consequence that outside an $\epsilon$ neighborhood of the pre-image of $P$ we have a uniform bound  from below on the distance of any point to its time $1$ flow under  $\pi^*h_{\lambda,\eta,\epsilon,\delta}$. Here the distance is measured with respect to the standard $J$ coming from the flat metric on $\mathbb{T}^n$. On the other hand these Hamiltonians are Lipschitz with respect to this $J$. Lemma \ref{lmDssipCrit} implies these Floer data are dissipative. 

We pick a sequence $\lambda_i,\eta_i, \epsilon_i,\delta_i$ such that $\lambda_i\to\infty,$ $\eta_i, \epsilon_i,\delta_i\to 0$, and let $h_i:=h_{\lambda_i,\eta_i, \epsilon_i,\delta_i}$. Then the sequence $h_i$ converges pointwise to $0$ on $P$ and to $\infty$ on the complement of $P$. We may assume the sequence $h_i$ is monotone increasing.   Moreover, we can make sure that $|\mathcal{A}_{\lambda,\eta,\epsilon,\delta}(\alpha)+\alpha(F(P,\alpha))|<1/i$ for all homology classes $\alpha\in (\mathbb{Z}^n)^\vee$.

We perturb $\pi^*h_i$ to a non-degenerate $H_i$ by adding a sufficiently $C^2$ small function supported in a sufficiently small neighborhood of $$\pi^{-1}\left(\bigcup_{\alpha\in (\mathbb{Z}^n)^\vee\cap I_{\lambda}} X_\alpha\right)$$ such that the action of any $1$-periodic orbit in the homology class $\alpha\in (\mathbb{Z}^n)^\vee\cap I_{\lambda}$ still converges to $-\alpha(F(P,\alpha))$. This is standard, for example see \cite[Lemma 4.17 (1)]{borman}. The functions $H_i$ are non-degenerate and dissipative, and moreover, making the perturbation even smaller, the sequence $H_i$ is monotone. We can extend to an acceleration datum by picking for each $i$ an arbitrary non-degenerate monotone dissipative path connecting $H_i$ and $H_{i+1}$.  To prove part (1) we adjust  our Hamiltonians  so that the actions of the $1$-periodic orbits of $H_i$ in homology class $\alpha\in (\mathbb{Z}^n)^\vee$ are always in the interval $(-\alpha(F(P,\alpha))-2,-\alpha(F(P,\alpha))+2)$.



For part (2), first note that due to the splitting of the entire telescope according to the homology class of the $1-$periodic orbits, it suffices to show the property for sums with only one term. Any homogenous class in $SH^*_{M,\theta}(\pi^{-1}(P),\mathbb{F})$ with homology grading $\alpha$ can be represented by elements lying in the subcomplex  $$\left(\bigoplus_{i=N}^\infty CF^*(H_i;\mathbb{F})[q]\right)$$ for arbitrary $N>0$. Combined with the fact that all the maps in the telescope are action non-decreasing, the proof is complete.

\end{proof}

Let us summarize what we have established so far. For any inclusion $Q\subset P$ of admissible convex polytopes we obtain a commutative diagram 
\begin{equation}\label{eqDiagramLocalModel}
\xymatrix{
\cA^*\otimes\Lambda \ar[r] \ar[rd] &  SH^*_{M,\theta}(\pi^{-1}(P);\mathbb{F})\hat{\otimes}\Lambda\ar[d]\ar[r]& SH_M^*(\pi^{-1}(P),\Lambda)\ar[d]\\
& SH^*_{M,\theta}(\pi^{-1}(Q);\mathbb{F})\hat{\otimes}\Lambda\ar[r]&SH_M^*(\pi^{-1}(Q),\Lambda).
}
\end{equation}
The left side of the diagram comes the composition of the maps in Theorem  \ref{thm-BV-torus} and \eqref{eqSHComparison} (see Lemma  \ref{lmComparisonNaturality}) together with functoriality of $\otimes\Lambda$ and completion. The right hand side is the commutative square of the rescaling map \eqref{eqRescaling}. Recall now the sheaf $P\mapsto \cA^*(P)$ discussed in Section \ref{Sec-Yk-analyt}. According to \eqref{eq-BV-comp-B-side}, when $P$ is a convex polytope, $\cA^*(P)$ is the completion of $\cA^*\otimes\Lambda$ with respect to a certain norm. We can now state the outcome of this section.

\begin{theorem}\label{thm-full-BV-comp}
Let $P\subset \bR^n$ be an admissible convex  polytope, then the map $\cA^*\otimes\Lambda\to SH_M^*(\pi^{-1}(P),\Lambda)$ from the composition of the horizontal arrows at the top of diagram \eqref{eqDiagramLocalModel}
extends uniquely to an isomorphism 
\begin{equation}\label{eqAtoSH}
\cA^*(P)\to SH_M^*(\pi^{-1}(P),\Lambda)
\end{equation}
of normed $\Lambda$-BV algebras. These isomorphisms are compatible with the restriction maps.
\end{theorem}
\begin{proof}
If $P$ contains the origin, by Proposition  \ref{prpBoundaryDepthComparison}, Theorem \ref{tmSHComparison} and Proposition \ref{prop-BV-torus-A}, $SH_M^*(\pi^{-1}(P),\Lambda)$ is  the completion of the image of $\mathcal{A}^*\otimes \Lambda$ with respect to the filtration map $$\sum z^{\alpha_{i}}\otimes \beta_i\otimes A_i\mapsto  \min_{i}(val(A_i)+\min_{b\in P}-\alpha_i(b)).$$  This is also seen to be true even when $P$ does not contain the origin by choosing a different base point on $B_0^n$. 
$\cA^*(P)$ is isomorphic to the same completion of $\mathcal{A}^*\otimes \Lambda$ by Equation \eqref{eq-BV-comp-B-side} and Proposition \ref{prop-BV-torus-B} proving the first part of the claim. Naturality with respect to restriction maps follows from the diagram \eqref{eqDiagramLocalModel}.

\end{proof}

 Let us note the degree $0$ portion separately. Recall that we have a sheaf on $\mathbb{R}^n$ with the G-topology of admissible (not-necessarily convex) polytopes defined by $$\mathcal{F}_0(P):=SH^0_M(\pi^{-1}(P),\Lambda).$$
 Using Lemma \ref{lmSheafGluing} we conclude.

\begin{corollary}\label{cor-no-ray-degree-0}
Equation \eqref{eqAtoSH} induces an isomorphism of sheaves  of normed algebras on the G-topology of admissible polytopes $\mathcal{O}_0\to \mathcal{F}_0$ on $\mathbb{R}^n$. \end{corollary}

\section{Analysis of the local models for the nodal fibers I}\label{s-6}
\subsection{Symplectic local models}\label{ss-6.1}
For an integer $k>0$ the integral affine structure on $B^{reg}_k$ induces a lattice in each cotangent fiber. Let $X(B^{reg}_k)$ be the quotient of $T^*B^{reg}_k$ by these lattices. Let  $\pi_k: X(B^{reg}_k)\to B^{reg}_k$ be the induced Lagrangian torus fibration.  By construction, the integral affine structure on $B^{reg}_k$ induced by the Arnold-Liouville theorem is the one that we started with.

By gluing in an explicit \emph{local model}, one can then extend 
\begin{itemize}
\item the smooth structure on $B^{reg}_k$ to a smooth structure on $B_k:=B^{reg}_k\cup \{0\}$ with its natural topology, 
\item the symplectic manifold $X(B^{reg}_k)$ to a symplectic manifold $M_k$ and,
\item  the fibration $\pi_k:X(B^{reg})\to B^{reg}$ to a nodal Lagrangian torus fibration
$\pi_k:M_k\to B_k$ with $k$ focus-focus singularities. 
\end{itemize}
For details, see Section 7.1 of \cite{locality}.

We now turn to the computation of the sheaf of $\Lambda$-algebras $$\cF_k(P):=SH^0_{M_k}(\pi^{-1}_k(P);\Lambda)$$ over the $G$-topology of $B_k$. This will occupy most of the remainder of the paper. Our strategy will be to first compute the restriction of $\cF_k$ to $B_{k}^{reg}$, i.e., we will prove
\begin{theorem}\label{thm-F_k,reg}
There exists an isomorphism of sheaves $\cF_k$ and $\cO_k$ over $B_k\setminus \{0\}$.
\end{theorem}

We then prove a Hartogs property in Section \ref{s-hartogs} which shows that $\cF_k$ is completely determined by its restriction to $B_{k}^{reg}$. Since by Theorem \ref{prop-Hartogs-B} $\cO_k$ also satisfies the Hartogs property the isomorphism of Theorem \ref{thm-F_k,reg} will extend over all of $B_k$.

Let us define the  eigenray $l_\pm$ for $\pi_k:M_k\to B_k$ (including $k=0$) to be the complement of the image of the canonical embeddings $E^{\pm}_k: B^{\pm}\to B_k$ in the notation introduced in the beginning of Section \S\ref{ss-integral-affine-local}.

A \emph{Lagrangian tail} emanating from one of the focus-focus points $p$ is a properly embedded Lagrangian plane in $M_k$ which surjects under $\pi_k$ onto one of the eigenrays with fiber over the node being $p$ and with all the other fibers diffeomorphic to circles. A $\pm$ Lagrangian tail is one that lives above the $\pm$ eigenray.

Our strategy for constructing an isomorphism as in  Theorem  \ref{thm-F_k,reg} is to first construct isomorphisms between the restrictions of $\cF_k$ and $\cO_k$ to each of the sets $B_k\setminus l_{+}$ and $B_k\setminus l_{-}$. Then we will prove, using wall crossing analysis, that these two isomorphism glue to an isomorphism of global sections over $B_k$. For this we will use the two main Theorems of \cite{locality} which we recall below. We first recall a definition:

\begin{definition}
A symplectic embedding $\iota:X\to Y$ of equidimensional symplectic manifolds is called a \emph{complete embedding} if $X$ and $Y$ are both geometrically bounded. 
\end{definition}
We refer to \cite{locality} for a detailed discussion of this notion. It is shown in \cite[Proposition 7.16]{locality} that the symplectic manifolds $M_k$ for all $k\geq 0$ are geometrically of finite type.

The following theorem is our main tool and we refer to the isomorphisms mentioned in its statement as \emph{locality isomorphisms}.

\begin{theorem}[{\cite[Theorem 1.1]{locality}}]\label{thm-local}
Let $\iota:X\to Y$ be a complete embedding of symplectic manifolds. Assume that $Y$ is graded and the grading of $X$ is obtained by restriction.

Then, if $X$ is geometrically of finite type and $SH^i_X(K)$ has homologically finite torsion in degree $i\in \mathbb{Z}$ (c.f. Corollary \ref{cor-fin-bd-tor}), we have an isomorphism $\iota_*:SH^i_{X}(K)\to SH^i_{Y}(\iota(K))$ functorially with respect to $K$.
\end{theorem}

\begin{remark}\label{rmkTruncated}
    This result is deduced from the locality results for \emph{truncated symplectic cohomology} \cite[Theorem 3.5 and 3.6]{locality} using the homologically finite torsion assumption. Truncated symplectic cohomology is reviewed thoroughly in Sections 2.2 and 2.3 of \cite{locality}. See also Proposition 1.2 of \cite{locality}.
\end{remark}

\begin{lemma}\label{lmLocalityHamIso}
Let $\iota:X\to Y$ be a complete embedding and let Let $K\subset X$ be compact with the same assumptions as in Theorem \ref{thm-local}. Let $\phi_t:X\to X$ and  $\psi_t:Y\to Y$ be a Hamiltonian isotopy satisfying $\phi_t(K)=K$ and $\psi_t(\iota(K))=\iota(K)$ for all $t$ then for all $t$ we have an equality $(\psi_t\circ\iota\circ\phi_t)_*=\iota_*$ of morphisms $SH^i_{X}(K)\to SH^i_{Y}(\iota(K))$.
\end{lemma}
\begin{proof}
The locality morphisms are functorial with respect to complete embeddings. The claim now follows from the isotopy invariance property \cite[Theorem 4.0.1]{VarolgunesThesis} (only the $K=K'$ case is needed). We just remark that for a Hamiltonian isotopy $\psi:M\to M$ the locality isomorphism is the same as the relabelling isomorphism of \cite[Theorem 4.0.1]{VarolgunesThesis}. 
\end{proof}

The following Proposition is a particular case of  \cite[Theorem 7.33]{locality}, except for the last clause which is also straightforward from the proof of that statement.

\begin{proposition}\label{prpCompEmb}
Let $p$ be one of the critical points of $\pi_k:M_k\to B_k$, let $L_\pm$ be a Lagrangian tail from $p$ lying over $l_{\pm}$. Then, there are symplectomorphisms $\Phi_\pm: M_{k-1}\to M_k\setminus L_{\pm}$. 

Moreover, let $\epsilon>0$ and let $U_\pm^\epsilon$ be open neighborhoods of $l_\pm$ whose union is $U^\epsilon:=\{-\epsilon<u<\epsilon\}\subset B_k$ for $u$ the primitive integral affine function vanishing on the eigenline. Then $\Phi_\pm$ can be chosen to fit into commutative diagrams 
\begin{align}\xymatrix
   {
     M_{k-1}\setminus \pi_{k-1}^{-1}(U_\pm^\epsilon)\ar[d]_{\pi_{k-1}}\ar[rr]^{\Phi_\pm}& & M_k\ar[d]^{\pi_k}\\
       B_{k-1}\setminus U_\pm^\epsilon\ar[rr]& & B_k
   }
\end{align}
where the lower horizontal embedding is the one that is compatible with the canonical embeddings $E^\pm_i$ with $i=k-1,k$, and the induced map on $B_k\setminus U^\epsilon$ is the identity map on the upper half plane and the shear
$\begin{pmatrix}
        1& 1\\
        0& 1\\
        \end{pmatrix}$ on the lower half plane. 
        
      The maps $\Phi_{\pm}$ are determined up to Hamiltonian automorphisms of $M_{k-1}$ which preserve the fibration $\pi_{k-1}$ over $B_k\setminus U^{\epsilon}_{\pm}$.  For any $0<\epsilon'<\epsilon$ we can modify the maps $\Phi_\pm$ via Hamiltonian isotopies supported in the region $\Phi_\pm^{-1}(U_{\epsilon})$ so that the diagram above commutes upon replacing $\epsilon$ by $\epsilon'$. 

\end{proposition}

 We omitted the precise statement about the independence of $\Phi_\pm$ up to Hamiltonian isotopy of $M_k$ on the choice of Lagrangian tails and how the Hamiltonian isotopy interacts with the Lagrangian fibrations. This boils down to constructing careful Hamiltonian isotopies taking one choice of Lagrangian tail to any other, which we leave to the reader.
\\

The symplectic structure on $\pi_k^{-1}(B_k^{reg})\subset M_k$ admits a canonical primitive, which we now describe. There is an \emph{Euler vector field} $V$ on $B_k^{reg}$, which is defined as follows. At every point $b\in  B_k^{reg}$, there is a unique vector $v\in T_bB_k^{reg}$ such that the affine geodesic $\gamma: [0,1]\to B_k^{reg}$ that starts at $b$ with velocity vector $v$ satisfies $\gamma(t)\to$ the node as $t\to 1$. The Euler vector field at $b$ is defined by taking the negative of this vector $-v$. Using the Ehresmann connection on $\pi_k^{-1}(B_k^{reg})\to B_k^{reg}$ induced from the Gauss-Manin connection defined by the period lattice, we can lift the Euler vector field to a vector field $Z$ upstairs, which is easily checked to be a Liouville vector field. Let us call the resulting primitive $\tilde{\theta}_k\in \Omega^1(\pi_k^{-1}(B_k^{reg}))$. 

If $k=0$, we take ${\theta}_0=\tilde{\theta}_0$ on $B_0=B_0^{reg}$. For $k>0$, we first note that the relative deRham cohomology class $[(\omega_{M_k},\tilde{\theta}_k)]\in H_{dR}(M_k,\pi_k^{-1}(B_k^{reg}))$ is zero. This is because, by the relative deRham isomorphism, it suffices to show that this class pairs trivially with the fiber class and the classes of the Lagrangian tails in $H_2(M_k, \pi^{-1}_k(B_k^{reg})).$ This is easy to check. Now taking an arbitrary primitive $(\alpha, f)$ of $(\omega_{M_k},\tilde{\theta}_k)$, we can construct a primitive ${\theta}_k$ on $M_k$, which agrees with $\tilde{\theta}_k$ outside of an arbitrarily small neighborhood of the singular fiber, by patching together $\alpha$ and $\tilde{\theta}_k$.  All the computations of actions $$A(\gamma)=\int \gamma^*\theta +\int \gamma^*H dt, \text{ for }\gamma: S^1\to M_k$$ will be done using $\theta_k$ for which the neighborhood of the singular fiber  where $\theta_k\neq \tilde{\theta}_k$ is chosen to be sufficiently small so as not to meet any of the non-constant periodic orbits under consideration. Note that if $\gamma:S^1\to M_k$ is a $1$-periodic orbit contained inside a regular torus fiber $\pi_k^{-1}(b)$ generated by the integral covector $\alpha \in T_b^*B^{reg}_k$, then \begin{equation}\label{ref-action}
\int \gamma^*\theta_k=-\alpha(V_b).
\end{equation}


For $M=M_k$ and $W$ the preimage under $\pi_k$ of an admissible convex polygon $P\subset B_k$ containing $0$ in its interior, there is a nice way to construct an acceleration datum as in the proof of Proposition \ref{prop-BV-torus-A}. Each edge of $P$ belongs to a rational line which divides $B_k$ into two parts. We consider non-negative PL functions on $B_k$ which are zero on the part that contains $P$ and affine (with slopes chosen sufficiently irrational) on the other. Taking the maximum of all these PL functions, smoothing and perturbing we obtain functions that are dissipative with respect to certain standard almost complex structures. The actions of generators can be computed explicitly in this case using Proposition \ref{prop-action} and branch cuts in the base. 


\subsection{The wall-crossing map}\label{subsec-wall-cross-map}

Consider embedding $M_0$ into $M_k$ by choosing $k$ Lagrangian tails either all to the left ($-$ embedding) or all to the right ($+$ embedding). As in Proposition \ref{prpCompEmb} we can choose the $\pm$ embeddings so that away from $\epsilon$ neighborhoods of the eigenrays they cover integral affine embeddings  
\begin{equation}\label{eqLowerEpm}
E_\pm: B_0\setminus \{(v,u)\mid u=0, \pm v\geq 0\}\to B_k\setminus \psi_k(\{(v,u)\mid u=0, \pm v\geq 0\})
\end{equation}
such that $E_+\circ E_-^{-1}$ is the identity map on the upper half plane and the shear $\begin{pmatrix}
        1& k\\
        0& 1\\
        \end{pmatrix}$ on the lower half plane.

\begin{proposition}\label{prp-local}
The locality isomorphism of Theorem \ref{thm-local} associated with the complete embeddings $E_{\pm}$ induce isomorphisms of sheaves  
\begin{equation}
\cO_k|_{B_k\setminus l_{\pm}}\simeq  \cF_k|_{B_k\setminus l_{\pm}}.
\end{equation}
\end{proposition}
\begin{proof} Let us consider the $+$ case.
Let $P\subset B_k\setminus U^\epsilon_{+}$ be an admissible convex polygon. From Proposition \ref{prpCompEmb} (only the part that concerns the $+$ embedding),  Theorem \ref{thm-local}, Corollary \ref{cor-no-ray-degree-0}, and the diagrams \eqref{eqT+diagram} and \eqref{eqT-diagram}  we get an isomorphism \begin{equation}\label{eqLocalityIso}
\cO_k(P)\simeq  \cF_k(P).
\end{equation} for each $\epsilon$. These isomorphisms are compatible with restriction maps by the same statements. Then, we use Lemma \ref{lmSheafGluing} to extend to all admissible polygons inside $B_k\setminus U^\epsilon_{+}$. To extend from $B_k\setminus U^\epsilon_{+}$ to $B_k\setminus l_{+}$, we use Lemma \ref{lmLocalityHamIso}. Note this statement implies that the isomorphisms are independent of the choice of tails too. The $-$ case is handled exactly in the same way.
\end{proof}

Our proof of Theorem \ref{thm-F_k,reg} proceeds by showing the compatibility of the above maps for admissible polygons $P\subset B_k\setminus l_-\cap B_k\setminus l_+$. In other words, let us denote the monodromy invariant line in $B_k$ by $R$ for all $k>0$. For any admissible polygon $P$ in the complement of $R$ inside $B_k$, we obtain the following diagram
\begin{align}\label{eqWallCrossDiagram}
\xymatrix
   { & \mathcal{F}_k(P) \ar[dl] \ar[dr]& \\
     \mathcal{F}_0(E_+^{-1}(P)) \ar[d] & & \mathcal{F}_0(E_-^{-1}(P))\ar[d]\\
       \mathcal{O}_0(E_+^{-1}(P)) \ar[rr]& & \mathcal{O}_0(E_-^{-1}(P)).\\
   }
\end{align}
The arrows are as follows. The diagonal maps are the locality isomorphisms  of Theorem \ref{thm-local}. The vertical maps are the inverses to \eqref{eqAtoSH} from Theorem \ref{thm-full-BV-comp}. Since all the other arrows in the diagram are filtered isomorphisms we obtain the horizontal isomorphism, which is also filtered. We call the horizontal morphism the \emph{$A$-side wall-crossing isomorphisms} for the purposes of this document. The diagrams \eqref{eqT+diagram} and \eqref{eqT-diagram} induce a diagram
\begin{align}
\xymatrix
   { & \mathcal{O}_k(P) \ar[dl] \ar[dr]& \\
       \mathcal{O}_0(E_+^{-1}(P)) \ar[rr]& & \mathcal{O}_0(E_-^{-1}(P))\\
   }
\end{align}

and we call the corresponding map the \emph{B-side wall crossing map}. Evidently, Theorem \ref{thm-F_k,reg} will follow once we show the A- and B-side wall crossing maps are equal. To show this we first need to prove a general lemma about the leading term of the A-side wall crossing. This will be done in the next section. 


\begin{remark}
It might be possible to use Yu-Shen Lin's results from \cite{lin} (for example Theorem 6.18) along with the unexplored relationship between Family Floer theory and relative symplectic cohomology that was mentioned in Section \ref{sss-ff} to get an enumerative calculation of wall-crossing isomorphisms (as we defined them).
\end{remark}

\section{Locality via complete embeddings and wall-crossing}\label{s-7}
For a complete embedding $\iota:X\to Y$, Theorem \ref{thm-local} introduces for appropriate compact $K\subset X$  a locality isomorphism $\iota_*:SH^*_K(X)\to SH^*_Y(X)$. This isomorphism depends very much on the embedding $\iota$. Given a pair of such embeddings $\iota_1,\iota_2$ which agree on a neighborhood of $K$, we obtain an induced automorphism of $SH^*_K(X)$ by considering the composition $\iota_{2,*}^{-1}\circ\iota_{1,*}$. Let us refer to this map is the \emph{abstract wall crossing map} to distinguish it from the wall crossing map defined for a particular case in the diagram \eqref{eqWallCrossDiagram} and whose source and target involve some additional identifications. 
The aim of this section is to prove Lemma \ref{lmWallCrossing} which shows the abstract wall crossing map is a small perturbation of the identity.

We try to be as self-contained as  possible within the length limits, but some familiarity with Sections 3,4 and 5 of \cite{locality} would be helpful.

Recall from Section \ref{ss-filt-map} that a Novikov ring $\Lambda_{\geq 0}$ module $A$ is naturally equipped with a filtration map that we in this section denote by $$val_A(a):=\sup\{r\in \mathbb{R}_{\geq 0}\mid a\in T^rA\}.$$ We also define $|a|=e^{-val_A(a)}$ and call it the module semi-norm - note that if $A$ has torsion the scalar multiplication is only sub-multiplicative.

 \begin{remark} Given a chain complex $C$ over $\Lambda_{\geq 0}$, it is customary to define a semi-norm on each homology module $H^i(C)$ by taking the infimum over the module semi-norms of its representatives. The norm obtained in this way agrees with the module semi-norms of the homology modules. Also note that the natural inclusion $$T^rH^*(C)\subset \ker{(H^*(C)\to H^*(C/T^rC))}$$ is an equality for all $r$. \end{remark}

An extremely important fact is that a Novikov ring module map does not decrease valuations, or, equivalently, does not increase norms.

\begin{lemma}\label{lmWallCrossing}
Let $X,Y$ be geometrically bounded graded symplectic manifolds of the same dimension.  Let $K\subset X$ be compact and let $V_0\subset X$ be an open neighbourhood of $K$ with the property that there is an admissible function on $X$ with no critical points outside of $V_0$. In particular, $X$ is geometrically of finite type. Suppose further that $K$ has homologically finite torsion in degree $i\in \mathbb{Z}$. 

Let $\iota_1,\iota_2: X\to Y$ be grading compatible symplectic embeddings such that $\iota_1|_{V_0}=\iota_2|_{V_0}$. Denote by $\iota_{i,*}$ the locality isomorphisms of Theorem \ref{thm-local} and let $$WC_X:=\iota_{2,*}^{-1}\circ\iota_{1,*}: SH_X^*(K)\to SH_X^*(K).$$ We refer to it as \emph{the wall crossing map}. Then there is a $\delta>0$ so that for all $x\in SH^0_X(K)$ we have $|WC_X(x)-x|<e^{-\delta}|x|$. Moreover, if the symplectic form of $Y$ is exact, all eigenvalues of  $WC_X\otimes_{\Lambda_{\geq 0}}\Lambda$ are $1$. 
\end{lemma}


Even though we stated the result in terms of $WC_X,$ for the proof it will be psychologically more convenient to consider $$WC=\iota_{2,*}\circ\iota_{1,*}^{-1}: SH_Y(\iota_1(K))\to SH_Y(\iota_2(K)),$$ noting that $\iota_1(K)=\iota_2(K)$, and then prove the same statements for $WC$.

%
%
%
%

The analysis below employs the natural truncation maps $\tau_{\lambda}:SH^*_{M}(K)\to SH^*_{M,\lambda}(K)$. For their definition recall equation \eqref{EqDefSH} defining $SH^*_{M}(K)$ as the pushout over a contractible diagram $\mathcal{G}_K$ of the homologies of certain chain complexes $\widehat{tel}(\cC)$ defined over the Novikov ring $\Lambda_{\geq 0}$. The truncation map $\tau_{\lambda}$ arises by considering the obvious natural transformation arising from the truncation maps  $\widehat{tel}(\cC)\to\widehat{tel}(\cC)/T^{\lambda}\widehat{tel}(\cC)$. See Remark \ref{rmkTruncated} for further references. 

For the next two lemmas, we do not need any extra assumptions on $K\subset M$.

\begin{lemma}\label{corTruncatedNorm}

For  $a\in SH^*_M(K)$ and any $\lambda_0$ such that $\tau_{\lambda_0}(a)\neq 0$ we have $|a|=|\tau_{\lambda_0}(a)|$. 
\end{lemma}
\begin{proof}
For a chain complex $C$ over $\Lambda$ and $a\in H^*(C)$, we have
\begin{equation}
-\log\left|a\right|=\sup\left\{\lambda :a\in T^\lambda H^*(C)= \ker{(H^*(C)\to H^*(C/T^\lambda C))}\right\}. 
\end{equation}

For $\lambda_0\geq \lambda$ we denote by $\tau^{\lambda_0}_{\lambda}:SH^*_{\lambda_0}\to SH^*_{\lambda}$ the truncation map. The claim now follows by the functoriality $\tau_{\lambda}=\tau^{\lambda_0}_{\lambda}\circ\tau_{\lambda_0}$. 
\end{proof}

\begin{lemma}
Let $\{M_\gamma\}$ be a filtered\footnote{This word here is not related to filtration maps.} direct system of Novikov ring modules (see \cite[Section A.5.3]{Weibel}).
 Then for any $a\in \varinjlim_\gamma M_\gamma$
\begin{equation}
|a|=\inf_{\lambda, b\in M_{\lambda}:\kappa(b)=a}|b|.
\end{equation}
Here we denote by $\kappa:M_{\lambda}\to  \varinjlim_\gamma M_\gamma$ the structural maps. 
\end{lemma}
\begin{proof}
It is immediate that the RHS is at least as large as the LHS. For the inequality in the other direction, which is also very easy, let $x$ and $y$ be elements in the direct limit such that $T^rx=y$. We know that there exists $x_0$ in some $M_{\gamma_0}$ such that $\kappa(x_0)=x$.  Note also that $\kappa(T^rx_0)=y$. This leads to the proof.
\end{proof}

\begin{corollary}\label{CorNormCompatibility}
Let $a\in SH^*_{M,\lambda_0}(K)$.
 Then 
\begin{equation}
|a|=\inf_{H\in\cH_K, b\in HF^*_{\lambda_0}(H):\kappa(b)=a}|b|.
\end{equation}
Here for any $H\in\cH_K$ we denote by $\kappa:HF^*_{\lambda}(H)\to SH^*_{M,\lambda}(K)$ the structural map. 
\end{corollary}
\begin{proof}Follows from the filteredness of $\cH_K$, see \cite[Theorem 6.10]{groman}.

\end{proof}


One final ingredient we shall need is the following monotonicity estimate taken from \cite[Proposition 3.2]{Hein2010}.

\begin{lemma}\label{prpHein}
Let $M$ be a symplectic manifold and consider $\Sigma:=S^1\times \mathbb{R}$ with coordinates $(t,s)$ as usual. Let $s\mapsto (H_s,J_s)$ be a family of Floer data so that $\partial_sH_s,\partial _sJ_s$ is compactly supported on $\bR$. Let $U\subset V$ be precompact open subsets of $M$. Assume that in the region  $\overline{V}\setminus U$ we have that $J_s$ is independent of $s$, that $H_s-H_{s'}\equiv const$, and that $H_s$ is $t$-independent. Moreover, assume that $\partial V$ and $\partial U$ are level sets of $H_s$ for some, hence any, $s\in\bR$. Finally assume there $H_s$ has no $1$-periodic orbits in the region $\overline{V}\setminus U$. Then there is a constant $\delta>0$ depending only on the restriction of $(H_s,J_s)$ to $\overline{V}\setminus U$ so that any solution $u:\bR\times S^1\to M$ to the parametrized Floer equation 
\begin{equation}
\partial_su+J_s(\partial_tu-X_{H_s})=0
\end{equation}
which meets both $\partial U$ and $\partial V$ satisfies
\begin{equation}
E^{geo}(u):=\int_{-\infty}^{\infty}\|\partial_su(s,\cdot)\|_{L^2}^2ds\geq\delta.
\end{equation}

If we additionally assume that the Floer data is monotone, i.e., $\partial_sH_s\geq 0$, then \begin{equation}
E^{top}(u):=\int (u^*\omega+dH\wedge dt)\geq\delta.
\end{equation}
\end{lemma}

 Strictly speaking the statement \cite[Proposition 3.2]{Hein2010} doesn't mention the $s$-dependent case. The proof however requires virtually no adjustment as we assumed that $s$-dependence is trivial in the region $\overline{U}\setminus V$.

\begin{proof}[Proof of Lemma \ref{lmWallCrossing}]
It suffices to prove the claim for the induced map $$WC:SH^*_{Y,\lambda_0}(K)\to SH^*_{Y,\lambda_0}(K).$$ Indeed, under the assumption on torsion, $$SH^*_{Y}(K)=\varprojlim_{\lambda}SH^*_{Y,\lambda}(K).$$ Moreover, by Lemma \ref{corTruncatedNorm}, the norm on the left hand side of the last equation is the one induced by the norms on the right. 

We recall the construction of the locality isomorphism. Let $\iota$ be either one of $\iota_1$ or $\iota_2$. We denote $\iota(K)$ by $K$ by an abuse of notation. For each $\lambda>0,$ the locality isomorphism $\iota_*:SH^*_{X,\lambda}(K)\to SH^*_{Y,\lambda}(K)$ is constructed as follows. Denote by $\cH_{K,X}$ the set of Floer data $(H,J)$ on $X$ such that $H|_K<0$. We define $\cH_{K,Y}$ as the set of Floer data $(H,J)$ on $Y$ such that $H|_K<0$. Then $SH^*_{X,\lambda}(K)=\varinjlim_{(H,J)\in \cH_{K,X}}HF^*_{\lambda}(H,J)$. We show in \cite[Section 5]{locality} that we can find a cofinal set $\cH_{\iota,K,Y}\subset \cH_{K,Y}$ so that 
\begin{enumerate}
\item the associated $\lambda$-truncated Floer complexes $CF^*_{\lambda}(H,J)$ split as  $CF^*_{\lambda}(H,J)=CF^*_{\lambda,inner}(H,J)\oplus CF^*_{\lambda,outer}(H,J)$, where the first complex is generated by periodic orbits lying in $V_0$,    
\item 
the splitting is functorial at the homology level with respect to continuation maps,
\item the induced map $HF^*_{\lambda,outer}\to SH^*_{Y,\lambda}(K)$ is trivial, 
\item the complex $CF^*_{\lambda,inner}(H,J)$ is local. This means that there is a fixed open neighborhood $V_{\lambda}\subset X$ such that all Floer solutions of energy $\leq \lambda$ connecting the generators in $\iota(V_0)$ are contained in $\iota(V_{\lambda})$ and are unaffected by the values of $(H,J)$ outside of $\iota(V_{\lambda})$.  A similar statement holds for continuation maps. 
\end{enumerate}

The takeaway from this is that at the truncation level $\lambda$ we can concretely realize the locality map at the homology level in the following way. 



Fix the open set $V_{\lambda}\subset X$ independently of whether $\iota=\iota_1$ or $\iota=\iota_2$. Let  $[x]\in SH^*_{X,\lambda}(K)$. Let $H_{0,X}\in\cH_{K,X}$ so that $[x]$ is the image of an element $[x_0]\in HF^*_{\lambda}(H_{0,X})$. For $i=1,2$ let $H_{0,\iota_i}$ be a Hamiltonian on $Y$ so that $H_{0,\iota_i}\circ\iota_i=H_{0,X}|_{V_{\lambda}}$.  Let $H_{1,X}\in\cH_{K,X}$ and $H_{1,\iota_2}\in \cH_{\iota_2,K,Y}$ be Hamiltonians on $X$  and $Y$ respectively so that $H_{1,X}\geq H_{0,X}$, $H_{1,\iota_2}\geq \max\{H_{0,\iota_1}, H_{0,\iota_2}\}$ and so that $H_{1,\iota_2}\circ \iota_2=H_{1,X}|_K$. The existence of such a pair of Hamiltonians is justified by considering that the set $\cH_{\iota_2,K,Y}$ is cofinal. Observe we can pick first the Hamiltonian $H_{1,\iota_2}\in \cH_{\iota_2,K,Y}$ so that $H_{1,\iota_2}\geq \max\{H_{0,\iota_1}, H_{0,\iota_2}\}$ and then the existence of an appropriate $H_{1,X}$ is guaranteed by the properties listed for $\cH_{\iota_2,K,Y}$. Here and for the remainder of the proof we abusively omit $J$ from the notation even though it plays an important role in the locality isomorphism.


We then have two different maps  $HF^*_\lambda(H_{0,\iota_1})\to HF^*_\lambda(H_{1,\iota_2})$. We have the Floer theoretic continuation map $f_Y=\kappa_{01,Y}$ associated with a monotone interpolating datum from $H_{0,\iota_1}$ to $H_{1,\iota_2}$. The other is $f_X$ as defined by the diagram 
\begin{align}\xymatrix
   { HF^*_{\lambda}(H_{0,X})\ar[r]^{\kappa_{01,X}}\ar[d]^{\iota_{1,*}}&HF^*_{\lambda}(H_{1,X}) \ar[d]^{\iota_{2,*}} \\
     HF^*_\lambda(H_{0,\iota_1}) \ar[r]^{f_X} & HF^*_\lambda(H_{1,\iota_2})}.
\end{align}
Denote by $\kappa_0:HF^*_\lambda(H_{0,\iota_1})\to SH^*_{Y,\lambda}(K)$ and $\kappa_1:HF^*_\lambda(H_{1,\iota_2})\to SH^*_{Y,\lambda}(K)$ the structural maps. Let $[y]=\iota_{1,*}([x])\in SH^*_{Y,\lambda}(K)$ and  $[y_0]=\iota_{1,*}([x_0])\in HF^*_\lambda(H_{0,\iota_1})$.  Then $\kappa_1\circ f_Y([y_0])=\kappa_0([y_0])=[y]$, and $\kappa_1\circ f_X([y_0])=WC([y])$. 

Therefore, to prove the first part of the claim, we need to prove the inequality
\begin{equation}\label{eqLocalityContinuationEstimate}
\left|\kappa_1\circ f_Y([y_0])\right|>\left|\kappa_1\circ (f_Y-f_X)([y_0])\right|.
\end{equation}
Note that the expressions on both sides are independent of any of the choices made. Thus it suffices to prove the estimate for carefully chosen Hamiltonians. Using monotonicity, we may assume our Floer data are chosen so that there is a $\delta>0$ such that any local continuation Floer trajectory of energy $<\delta$ and connecting orbits inside $V_0$  is contained inside $V_0$  and a similar claim for the continuation trajectories in $Y$. Indeed, we may pick the Hamiltonians $H_{0,X}\leq H_{1,X}$ so that the following are satisfied
\begin{itemize}
\item $V_0=H^{-1}_{0,X}(-\infty,t)$ for some real $t$. 
\item There is a $t'<t$ such that all periodic orbits of either $H_{0,X}$ or $H_{1,X}$ that are contained in $V_0$ are contained inside $U:=H^{-1}_{0,X}(-\infty,t')$.
\item On the region $\overline{V}_0\setminus U$ we have that $H_{1,X}-H_{0,X}$ is constant. 
\end{itemize}
We are then in the setting of Lemma \ref{prpHein}.

Fixing such a $\delta$ we further assume by Corollary \ref{CorNormCompatibility} that $H_{0,X}$ is chosen so that $|[x_0]|<e^{\delta/2}|[x]|$. Note that $H_{0,X}$ can be adjusted to satisfy this in such a way that outside of $V_0\setminus U$ it is only shifted by a large constant. Thus the the constant $\delta$ guaranteed by Lemma \ref{prpHein} is not changed by this adjustment. The left hand side of \eqref{eqLocalityContinuationEstimate} is then $|[x]|$ while the right hand side is $<|f_Y-f_X|e^{\delta/2}|[x]|$. So to conclude it now suffices to show $|f_Y-f_X|\leq e^{-\delta}$ . For this note that  our Hamiltonians agree with the local ones on $V_0$ since we assumed $\iota_1=\iota_2$ on $V_0$. Therefore, the contribution to $f_Y-f_X$ comes from the trajectories which leave the region $V_0$. By the choice of $\delta$ these all have energy $\geq\delta$. The first part of the claim follows.


It remains to prove the claim concerning eigenvalues. For this note that, after base  change to $\Lambda$, we can take our underlying chain level models for relative $SH$ to be of the form $tel(\mathcal{C}_{\mathbb{F}})\hat{\otimes} \Lambda$ as in the proof of Proposition \ref{prpBoundaryDepthComparison}. The locality maps and therefore the wall crossing maps are then defined over the trivially valued field $\mathbb{F}$.  Relying on Proposition \ref{prop-tor-eq-bdry},\footnote{Note that the finiteness/infiniteness of boundary depth is invariant under completion and under tensoring with $\Lambda$.} the conclusion of Proposition \ref{prpBoundaryDepthComparison} holds. Thus, if $[x]$ is an eigenvector then after scalar multiplication, we may assume that it is of the form $a\otimes 1$ for $a\in SH^*_{M,\theta}(K;\mathbb{F})$. The inequality $|WC_X(x)-x|<|x|$  implies $WC_X(x)=(1+c)x$ for $c\in\mathbb{F}$ of norm $<1$.  Since $\mathbb{F}$ is trivially valued, this is only possible for $c=0$.
\end{proof}

\section{Analysis of the local models for the nodal fibers II}\label{s-8}
Armed with the result of the previous section we turn to compute the wall crossing maps. Let $\eta$ and $\xi$ be the monomials corresponding to $-e_2^{\vee}$ and $e_1^{\vee}$ in $\Lambda[(\mathbb{Z}^2)^{\vee}]=\Lambda[H_1(M_0;\bZ)]\subset \cO_0(Q)$ for any admissible polygon $Q\subset B_0$. Recall that the wall crossing isomorphisms were defined as the horizontal arrow in the diagram \eqref{eqWallCrossDiagram}. We remind the reader that the wall crossing maps are maps
$$\mathcal{O}_0(E_+^{-1}(P)) \to\mathcal{O}_0(E_-^{-1}(P)),$$
and that both $E_+^{-1}(P)$ and $E_-^{-1}(P)$ are polygons in $B_0$. On the other hand for $Q=E_{\pm}^{-1}(P)$ connected, $\cO_0(Q)$ is a certain completion of the Laurent polynomials in the variables  $\eta,\xi$. Thus to compute the wall crossing isomorphisms all we need is to derive formulas for the wall crossing map applied to global functions $\eta,\xi$. 

In deriving a formula for the wall crossing we will  have to deal separately with the case of $P\subset B_k\setminus R$ in the upper and lower half plane. On the other hand, in each half plane, the formulas, as formal Laurent series in $\eta,\xi$, will not depend on the choice of the polygon $P$. To see this observe that wall-crossing isomorphisms are compatible with restriction maps.  Moreover, the restriction maps in each half plane are intertwined with those of $\cO_0$ and we know these to be injective as stated in Proposition \ref{prpUnique}. From now on let us assume that $P\subset B_k\setminus R$ is a convex admissible polygon, which we will vary to prove certain restrictions.



Let us start listing the results that eventually lead to the computation of the wall-crossing isomorphisms. 

\begin{proposition}\label{prppi1grading}
There are formal Laurent series $h_1,h_2,h_3,h_4\in\Lambda[[\eta,\eta^{-1}]]$ such that
the wall-crossing isomorphisms 
$$\mathcal{O}_0(E_+^{-1}(P)) \to \mathcal{O}_0(E_-^{-1}(P))$$
for $P$ lying in the upper (resp. lower) side of $R$ send any monomial $\xi^a\eta^b$, with $a,b$ non-negative integers, to an element with formal expression $(\xi^a\eta^b)\cdot h_1^ah_2^b$ (resp. $(\xi^a\eta^b)\cdot h_3^ah_4^b$).
\end{proposition}
\begin{proof}

For any $K\subset M$ we have a grading of $SH^*_{M}(K)$ by $H_1(M;\bZ)$ which assigns to a periodic orbit its homology class. For the complete embeddings $E_{\pm}$ we get an induced $H_1(M_k;\mathbb{Z})=\bZ$ grading on $SH_{M_0}(E_{\pm}^{-1}(P))$.  This grading is preserved by the wall crossing maps. Indeed, $M_0$ deformation retracts to a torus fiber of $M_0$ on which each of the maps $E_{\pm}$ is the inclusion as a torus fiber of $M_k$. Moreover  the grading is determined by the projection map  $E_{\pm*}:H_1(M_0;\mathbb{Z})\to H_1(M_k;\mathbb{Z})=\mathbb{Z}$. The kernel of $E_{\pm*}$ is generated by the class of the vanishing cycle in the torus fiber of $M_k$ under the above identifications.  By definition of $\eta$ in the beginning of this section the  vanishing cycle corresponds to either $\eta$ or $\eta^{-1}$ under the identification $\Lambda[(\mathbb{Z}^2)^{\vee}]=\Lambda[H_1(M_0;\bZ)]\subset \cO_0(E_+^{-1}(P))$ in \eqref{eqTorusSH}. Thus the wall crossing map lies in the subalgebra generated by $\eta$.
\end{proof}
\begin{proposition}\label{prop-WC-BV}
Both wall-crossing isomorphisms send $\eta$ to $\eta$, i.e., $h_2=h_4=1$.
\end{proposition}
\begin{proof}
We first consider the upper wall crossing map. The wall-crossing isomorphism is in fact defined in all degrees, since the locality isomorphisms are defined in all degrees. Moreover, the wall-crossing isomorphism respects the $BV$ algebra structure by the same property of locality isomorphisms  (cf. \cite[Proposition 5.7]{mandel}). We know that $SH^*_{M_0}(\pi^{-1}(Q),\Lambda)$ is isomorphic as a BV algebra to $\mathcal{A}^*(Q)$ for any convex admissible polygon $Q\subset B_0$ by Theorem \ref{thm-full-BV-comp}. It thus follows from Proposition \ref{prpCYBV} that up to multiplication by a constant $c\in\Lambda^*$ the CY form $\Omega$ is preserved under the wall-crossing isomorphisms. Moreover, by the last clause in Lemma \ref{lmWallCrossing} we have $c=1$. 

We show how this implies the claim. Preservation of the CY form together with Proposition \ref{prppi1grading} give us the equation
\begin{align}
d\log{\eta}\wedge d\log{\xi}&= d\log\left(h_1(\eta,\eta^{-1})\eta\right)\wedge d\log\left(h_2(\eta,\eta^{-1})\xi\right)\notag\\
&= d\log{\eta}\wedge d\log{\xi}+d\log{h_1(\eta,\eta^{-1})}\wedge d\log{\xi}\notag.
\end{align}
From which we deduce $d\log{h_1(\eta,\eta^{-1})}$ and $d\log{\xi}$ are linearly dependent which is only possible if $d\log h_1=0$. That is, $h_1$ has only a constant term. In particular $\eta$ is an eigenvector of the wall crossing map. Thus again by the last clause in Lemma \ref{lmWallCrossing} we get that $h_1= 1$.

We now treat the lower wall crossing map. Our choice in Proposition \ref{prpCompEmb} does not allow us to directly apply Lemma \ref{lmWallCrossing} since we do not have  $\iota_1|_{V_0}=\iota_2|_{V_0}$ for $V_0$ a neighborhood of $\pi_0^{-1}(P)$. Rather we have $\iota_1|_{V_0}=\iota_2\circ \psi|_{V_0}$ for $\psi$ the symplectomorphism of $M_0$ induced by the shear map. To conclude, it suffices to observe that the induced action of the symplectomorphism $\psi$ is $\eta\mapsto\eta$ and $\xi\mapsto\eta^k\xi$.


\end{proof}

\begin{proposition}\label{PrpUpperLowerWC}
The upper wall-crossing isomorphism sends $$\xi\mapsto \xi(1+a_1\eta+a_2\eta^2+\ldots)$$ and the lower wall-crossing isomorphism sends $$\xi\mapsto \xi\eta^k(1+b_1\eta^{-1}+b_2\eta^{-2}+\ldots),$$ for some $a_i,b_i\in \Lambda.$
\end{proposition}
\begin{proof}
The wall crossing maps preserve the action filtration. Under the vertical isomorphisms of Diagram \eqref{eqWallCrossDiagram}, the action filtration with respect to the primitive on $M_0$ used in Section \ref{s-noray} corresponds to the valuation on $\mathcal{O}_0(Q)$ with $Q=E_{\pm}^{-1}(P)$ given by \begin{align}val_Q\left( \sum a_j\eta^{n_j}\xi^{m_j}\right):= \inf_{(v,u)\in Q}\inf_j \left(val(a_j)- m_jv+n_ju\right)\end{align} as noted in the proof of Theorem \ref{thm-full-BV-comp}.  

For the upper wall crossing map no negative  power of $\eta$ can appear in $h_2$ because we can take $P$ to be arbitrarily far away from $R$ and obtain a contradiction to the preservation of the valuation. Furthermore, using Lemma \ref{lmWallCrossing}, we obtain that the constant term of $h_2$ has to be $1$.

For the lower wall crossing map the same reasoning applies by first modifying the locality embedding with a symplectomorphism as in the last paragraph of the  proof of Proposition \ref{prop-WC-BV}.
\end{proof}

Finally, we come to the step which relates the upper and lower wall-crossing isomorphisms.

\begin{proposition}\label{thm-monodromy} We have the equality 
\begin{equation}\label{eq-thm-monodromy} 
1+a_1\eta+a_2\eta^2+\ldots=\eta^k(1+b_1\eta^{-1}+b_2\eta^{-2}+\ldots).
\end{equation}
\end{proposition}

The proof of this proposition relies on a \emph{monodromy argument} which we formulate now. Consider an admissible convex polygon $S\subset B_k$ as shown in Figure \ref{fig-P-old} for definiteness.

Let us call $X$ (resp. $Y$) the union of the upper, lower and left (resp. right) edges  of $S$. See Figure \ref{fig-P-old}. Recall the maps $E_{\pm}$ defined in \eqref{eqLowerEpm}. Note that $X$ (resp. $Y$) lies in the image of $E_+$ (resp. $E_-$). Therefore we can use the locality isomorphisms with respect to the $+$ (resp. $-$) embedding for $X$ (resp. $Y$) as in Diagram \eqref{eqWallCrossDiagram}.
We use this along with Proposition \ref{prop-basic-hartogs} to identify $$\cF_k(X)\simeq\cF_0(E_+^{-1}(X))\simeq \cO_0(E_+^{-1}(X))\simeq  \cO_0(hull(E_+^{-1}(X))),$$ where $hull(E_+^{-1}(X))$ is the convex hull of $E_+^{-1}(X)$ inside $\bR^2$ - an admissible convex polygon. The first two isomorphisms are the maps of the left hand of diagram \eqref{eqWallCrossDiagram} while the last one is from Proposition \ref{prop-basic-hartogs}. Of course the same analysis can be made for $Y$ as well. 

Note that unlike what happened in Diagram \eqref{eqWallCrossDiagram} for admissible polygons contained in the upper or lower half of the eigenline, we have exactly one identification for $X$ and $Y$ depending on which eigenray they intersect.

\begin{figure}
\includegraphics[width=0.6\textwidth, trim= 0 700 0 100]{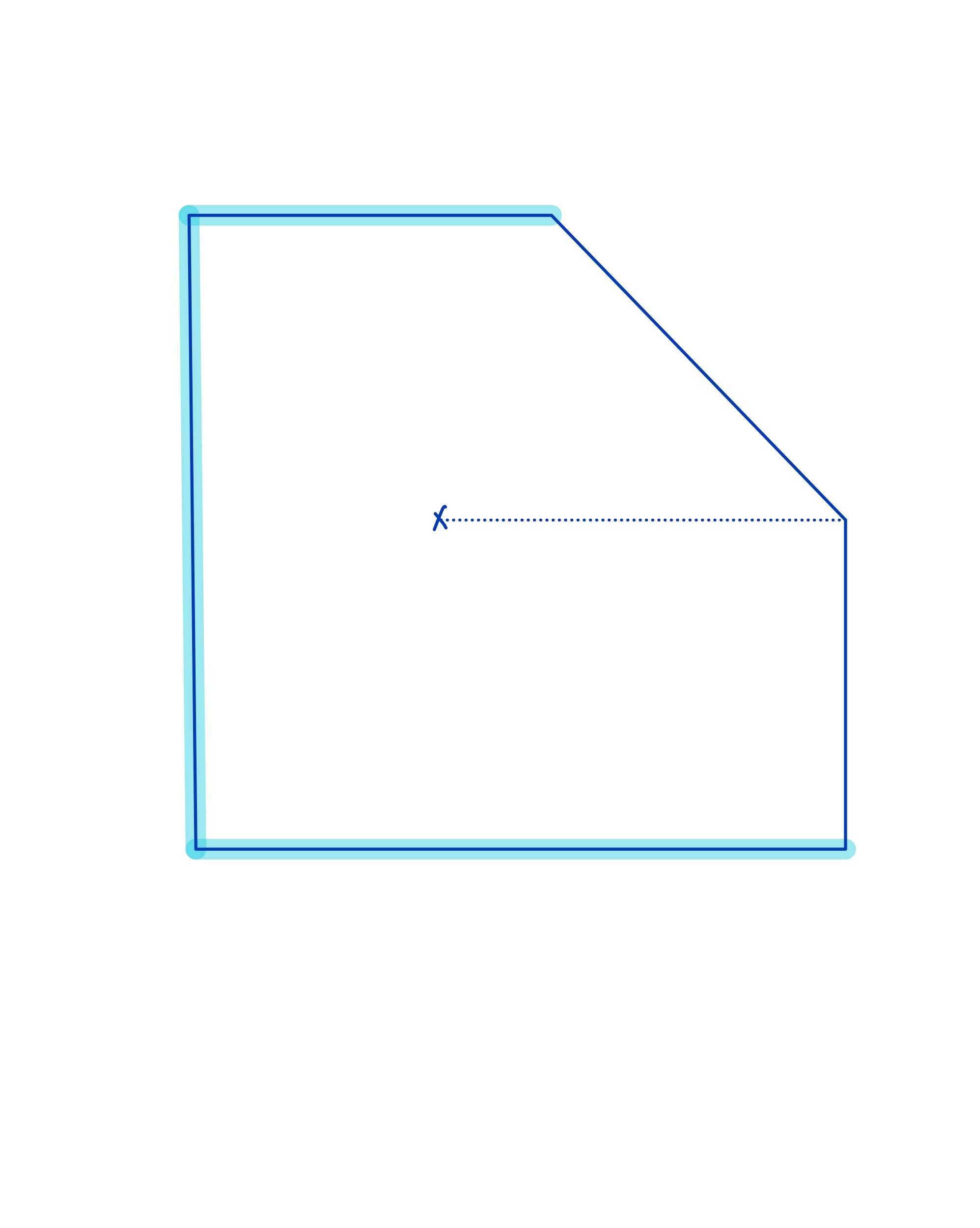}
\caption{The polygon $S$ drawn inside $B_k$ with the dashed ray representing the defining monodromy invariant ray.  The set $X$ is highlighted.}
 
 \label{fig-P-old}
\end{figure}



\begin{proposition}\label{prop-max-action}
The restriction map $$\cF_k(S)\to \cF_k(X)=\cO_0(hull(E_+^{-1}(X)))$$ has in its image the element with formal expression $\xi.$
\end{proposition}

We first finish the proof of Proposition \ref{thm-monodromy} using this statement.

\begin{proof}[Proof of Proposition \ref{thm-monodromy}]
Take an element $A\in \cF_k(S)$ which has image $A_X=\xi$ under the identification $\cF_k(X)=\cO_0(hull(E_+^{-1}(X)))$ as in the statement of Proposition \ref{prop-max-action}.

Restrict $A$ to either the upper edge $U$ or the lower edge $L$ of $P$. Using the left side of Diagram \eqref{eqWallCrossDiagram}, we obtain elements $A_U\in \cO_0(E_+^{-1}(U))$ and $A_L\in \cO_0(E_+^{-1}(L))$. The formal expression for both of these elements is $\xi$.

Now use the upper and the lower wall-crossing isomorphisms to get elements in $A_U'\in \cO_0(hull(E_-^{-1}(U)))$ and $A_L'\in\cO_0(hull(E_-^{-1}(L)))$. We a priori know that both of these elements have to be restrictions of $A_Y\in \cO_0(hull(E_-^{-1}(Y)))$ and in particular they have to have the same formal expression. 

Writing this down using Proposition \ref{PrpUpperLowerWC} we obtain $$\xi\cdot (1+a_1\eta+a_2\eta^2+\ldots)=\xi\cdot \eta^k(1+b_1\eta^{-1}+b_2\eta^{-2}+\ldots).$$ It is easy to see that one can cancel $\xi$ from this relation. Indeed, we can restrict to a polygon where $\xi$ is invertible. The claim follows.

\end{proof}

\begin{corollary}\label{cor-WC-final}
The wall-crossing isomorphisms (both!) are given by $$\eta\mapsto \eta, \xi\mapsto \xi(1+\eta)^k.$$
\end{corollary}
\begin{proof}
 For $k=1$, we uniquely determine the wall-crossing transformations from equation \eqref{eq-thm-monodromy}. For $k>1$, \eqref{eq-thm-monodromy} does not have a unique solution, but notice that we can resolve the the $k$-fold singularity into $k$ simple nodal singularities lying on one monodromy invariant line by nodal sliding (or moving of worms, see Proposition 7.8 of \cite{locality} and discussion preceeding it).  This can be achieved by modifying $\pi_k$ in a small neighborhood of $\pi_k^{-1}(0)$, and in particular does not modify  $\pi_k^{-1}(S)\subset M_k$ as a set or its relative $SH$. For this resolved singularity the wall crossing around all $k$ singularities is the same as the wall crossing around the unresolved one, and it factors as a $k$-fold iteration of the wall crossing isomorphism for the case $k=1$.
 
\end{proof}
Finally, we can prove Theorem \ref{thm-F_k,reg}.
\begin{proof}[Proof of Theorem \ref{thm-F_k,reg}]
As commented in Section \ref{subsec-wall-cross-map} this amounts to matching the A- and B- side wall crossing. The formula for the B-side wall crossing is given in Proposition \ref{prop-KS-cover}, which matches with Corollary \ref{cor-WC-final}.
\end{proof}

\subsection{Proof of Proposition \ref{prop-max-action}}

Let us now prove the intermediary Proposition \ref{prop-max-action}
\begin{proof}[Proof of Proposition \ref{prop-max-action}]

Let $L$ be the edge of $X$ on the left.  By functoriality with respect to inclusions and our knowledge of the restriction map $\cO_0(hull(E_+^{-1}(X)))\to \cO_0(E_+^{-1}(L))$ it suffices to prove that the restriction map to $L$ contains $\xi$, where we identified $\mathcal{F}_k(L)$ with $\cO_0(E_+^{-1}(L))$ via the locality isomorphism as well. 

A Hamiltonian $H$ is called \emph{$Q$-admissible} for a subset $Q\subset B_k$ if $H<0$ on $\pi_k^{-1}(Q)$.  A dissipative $H$ is called \emph{$\xi$-full} if $\xi$ is in the image of the natural map $HF^0(H;\Lambda)\to SH^0_{M_k}(\pi_k^{-1}(L);\Lambda)=\cF_k(L)$.

It suffices to show that there is an $S$-admissible Hamiltonian $H$ which is $\xi$-full.  Indeed, the map $HF^0(H;\Lambda)\to \cF_k(L)$ factors through the restriction map $\cF_k(S)\to\cF_k(L)$. 

Let $B_+$ be the image of $E_+$ and $M_+:= \pi_k^{-1}(B_+)$. A function $h:B_k\to \bR$ is said to be \emph{convex on the left}  if $dh(e_1)\geq 0$ on the eigenray $B_k\setminus B_+$ and on $B_+$ any co-vector $\nu$ with $\nu(e_1)<0$ is obtained at most once as $dh_b$ for some $b\in B_+$, using the canonical identification of the cotangent spaces of points in $B_+$

For a function $h$ which is convex on the left we show now that if the level sets are sufficiently close to being rectangular in a neighbourhood of $L$ we have the following property.  Equip $M$ with a $S$-compatible primitive (given by Euler primitive outside an open subset whose closure is contained in the interior of $S$). Consider the homology class $a$ in $H_1(M_+,\mathbb{Z})$ given by the primitive covector (called $-e_1^\vee$) of the left edge pointing out of $S$. Let $\tilde{a}$ be its image in $H_1(M,\mathbb{Z})$. Then if $h$ has a $1$-periodic orbit $\gamma$ corresponding to $-e_1^\vee$, this orbit maximizes action among all orbits in the class $\tilde{a}$. To see this note the action of an orbit representing $e_1+ne_2^\vee$ is arbitrarily close to $-\lambda_1-|n|\lambda_2$ where $\lambda_1$ is the distance from the origin to $L$ and $\lambda_2$ is half the length of $L$ by Proposition \ref{prop-action} (using the flat metric in the domain of $E^+$). We will refer to this as the action maximization property in the later parts of the proof.

Call a convex-on-the-left $h$ which has sufficiently large slope to the left of $L$ and with level sets sufficiently close to rectangular for the property of the previous paragraph to hold \emph{well-behaved}. It is straightforward that we can construct acceleration data for $L$ which consist of well behaved Hamiltonians.

Given $F:M_k\times S^1\to \mathbb{R}$ that is an appropriate small time dependent perturbation of $f\circ \pi_k$ for a well behaved $f$, let $\gamma^F$ be the $1$-periodic orbit corresponding to the negative of $e_1^\vee$ in degree $0$. By standard Morse-Bott considerations, the only possible  contributions to the differential of $\gamma^F$ are those connecting it to degree $1$ periodic orbits which don't arise from the perturbation of  $-e_1^\vee$. By the action maximization property we conclude that $\gamma^F$ is a cycle.  

Fix $H, G:M_k\times S^1\to \mathbb{R}$ as the $F$ of the last paragraph so that $H\leq G,$ $H$ is $S$-admissible and $G$ is $L$-admissible. Assume  that a neighborhood of $\gamma^H$ coincides with a neighborhood of $\gamma^G$ and that $H$ and $G$ coincide there. Then a monotone continuation map from the Floer complex of $H$ to that of $G$  maps $\gamma^H$ precisely to $\gamma^G$. Indeed, we have the constant solution, and we can have no other by the action maximization property. 

To conclude, we turn to show that at least for carefully chosen $G$  and almost complex structures, the generator $\gamma^G$ maps to a non-zero scalar multiple of $\xi$. By nodal sliding to the right (similar to the argument in locality for complete embeddings) we can assume that any relevant Floer trajectory with $\gamma^G$ as input with energy $\leq 1$ is contained in a fixed compact neighbourhood $K$ of $\pi^{-1}_k(L)$ inside $M_+$. 

We first claim that up to $O(1)$ we can guarantee that $\gamma^G$ maps to $\pm T^{\epsilon}\xi$ for some $\epsilon<1$. To see this, let $G_{loc}$ be an extension of $G|_K$ to $M_0$ with the property that outside of $K$ it does not have any $1$-periodic orbits in the same homology class as $\gamma^G$ inside $H_1(M_0; \mathbb{Z})$. We know that for some $G':M_0\times S^1\to \mathbb{R}$, $\xi$ is in the image of the map $HF^*(G';\Lambda)\to SH^*_{M_0}(\pi^{-1}(L);\Lambda).$ For $G_{loc}$ larger than $G'$ outside of a compact set (which can also be arranged) we therefore obtain that $\xi$ is in the image of the map $HF^*(G_{loc};\Lambda)\to SH^*_{M_0}(\pi^{-1}(L);\Lambda)$. But the only thing that can map to the span of $\xi$ is the span of the generator $\gamma^G$. Since the topological energy of the Floer solutions is given by difference of actions, which can be made to be smaller than $1$ in all the relevant cases, all the Floer solutions lie inside $K$ and have topological energy smaller than $1$.  This proves that $\gamma^G$ maps to $mT^{\epsilon}\xi$ for some $\epsilon<1$ and non-zero integer $m$. Working over the integers we have $m=\pm1$ for otherwise $\xi$ cannot be in the image, only a multiple of it. 

We now rule out the possibility of an $O(1)$ correction. Such a correction would amount to a continuation trajectory from $\gamma^G$ to some orbit $\delta$ of some well behaved Hamiltonian $G'>G$ so that $\delta$ is in the class $\tilde{a}$ in $M$ but not in $M^+$. By definition of well behavedness, we have the action maximization property. Therefore such a $\delta$ has lower action and such a trajectory is impossible.  

\end{proof}

\begin{remark}
Let us explain what happened a posteriori. The $A$-side wall crossing maps as defined by Diagram \eqref{eqWallCrossDiagram} gives us an isomorphism of the rigid analytic spaces $$(\Lambda^{*})^2\setminus \{val(\eta^-)=0\}\to (\Lambda^{*})^2\setminus \{val(\eta^+)=0\}.$$ This is a morphism between two disconnected spaces and it appears that we have two entirely independent maps. But it turns out that in fact this map extends to an isomorphism $$(\Lambda^{*})^2\setminus \{\eta^-=-1\}\to (\Lambda^{*})^2\setminus \{\eta^+=-1\}.$$ If we knew this fact a priori, then Proposition \ref{thm-monodromy} would follow as both maps (and in particular the pullback of the regular function $\xi^+$) are determined by this one map. In fact our results prove that it is the restriction of the transition map from Proposition \ref{prop-KS-cover}. 

Our monodromy argument is a replacement of this point. Its intuition is the knowledge that the function $\xi^+$ extends to a function $Y_1^{an}$ and hence maybe we should use the mirror of this global function to obtain the desired relationship between how the mirror of $\xi^+$ transforms under the two wall crossing transformations. 

%
\end{remark}

\section{Hartogs's property}\label{s-hartogs}


Our goal in this section is to prove the following statement, which will allow us to extend the isomorphism of Theorem \ref{thm-F_k,reg} to all of $B_k$.

\begin{proposition}\label{prop-Hartogs-A}
Let $P\subset  B_k$ be an admissible convex polygon containing a node in its interior. Then, the restriction map $$\mathcal{F}_k(P)\to \mathcal{F}_k(\partial P)$$is an isomorphism. 
\end{proposition}

Recall that the same statement holds if we replace $\mathcal{F}_k$ with $\cO_k$, see Proposition \ref{prop-Hartogs-B}. The rough idea of the proof is that one can choose the acceleration data for $\partial P$ and $P$ so that in a large region where the degree $0$ orbits of the Hamiltonians are contained the acceleration data are exactly the same. We start with an abstract lemma.

\begin{lemma}\label{lem-deg0-hom}
Let $C\to D$ be a chain map of non-negatively graded chain complexes. Assume that $C^0\to D^0$ is an isomorphism and $C^1\to D^1$ is injective. Then, $H^0(C)\to H^0(D)$ is an isomorphism.
\end{lemma}

\begin{proof}
Injectivity is immediate. For surjectivity, let $d\in D^0$ be a closed element. There exists a $c\in C^0$ that maps to it. We know that the differential of $c$ maps to $0.$ By injectivity on degree $1,$ this implies that $c$ is closed as well, finishing the proof.
\end{proof}

\begin{lemma}\label{lem-alg-hartogs}
Let us take a map of $1$-rays over the Novikov ring $\Lambda_{\geq 0}$
\begin{align}
\xymatrix{ 
\mathcal{C}_1\ar[r]^{\iota_1}\ar[d]_{f_1}\ar[dr]_{h_1}& \mathcal{C}_2\ar[d]_{f_2}\ar[r]^{\iota_2}\ar[dr]_{h_2} &\mathcal{C}_3\ar[r]^{\iota_3}\ar[d]_{f_3}\ar[dr]_{h_3}& \ldots \\ \mathcal{C}_{1}'\ar[r]_{\iota_1'} &\mathcal{C}_2'\ar[r]_{\iota_2'}&\mathcal{C}_3'\ar[r]_{\iota_3'}&\ldots}
\end{align}

Assume that each of these chain complexes are non-negatively graded, in degrees $0$ and $1$ they are free of finite rank, $$f_i^0: \mathcal{C}_i^0\otimes \Lambda_{\geq 0}/\Lambda_{> 0}\to \mathcal{C}_i^{'0}\otimes \Lambda_{\geq 0}/\Lambda_{> 0}$$ is an isomorphism, $$f_i^1: \mathcal{C}_i^1\otimes \Lambda_{\geq 0}/\Lambda_{> 0}\to \mathcal{C}_i^{'1}\otimes \Lambda_{\geq 0}/\Lambda_{> 0}$$ is injective  and both $ H^0(\widehat{tel}(\mathcal{C}))$ and $ H^0(\widehat{tel}(\mathcal{C}'))$ have finite torsion.

Then, the induced map $$ H^0(\widehat{tel}(\mathcal{C}))\to H^0(\widehat{tel}(\mathcal{C}'))$$ is an isomorphism.
\end{lemma}
\begin{proof}

We have the commutative diagram: 
\begin{align}
\xymatrix{ 
H^0(\widehat{tel}(\mathcal{C}))\ar[d]\ar[r]& H^0(\widehat{tel}(\mathcal{C}'))\ar[d]\\ \varprojlim_r(\varinjlim_i(H^0(\mathcal{C}_i/T^r\mathcal{C}_i)))\ar[r] &\varprojlim_r(\varinjlim_i(H^0(\mathcal{C}_i'/T^r\mathcal{C}_i')))}
\end{align} and, under the finite torsion assumption,  the vertical maps are isomorphisms by \cite[Proposition 1.2]{locality}.

Therefore, it suffices to show that $$H^0(\mathcal{C}_i/T^r\mathcal{C}_i)\to H^0(\mathcal{C}_i'/T^r\mathcal{C}_i')$$ is an isomorphism for every $r\geq 0$ and $i=1,2,\ldots$. Fix an $i$ and $r$ for the rest of the argument.

Now, by Nakayama's lemma, we see that $$f_i^0: \mathcal{C}_i^0\otimes \Lambda_{\geq 0}/\Lambda_{> r}\to \mathcal{C}_i^{'0}\otimes \Lambda_{\geq 0}/\Lambda_{> r}$$ is an isomorphism.

Note that, moreover, $$f_i^1: \mathcal{C}_i^1\otimes \Lambda_{\geq 0}/\Lambda_{> r}\to \mathcal{C}_i^{'1}\otimes \Lambda_{\geq 0}/\Lambda_{> r}$$ preserves the module norms and hence is injective for all $r\geq 0$. Here we are also using that $\mathcal{C}_i^1$ is free to conclude that the module semi-norms are in fact norms.

Finally, we use Lemma \ref{lem-deg0-hom} to finish the proof.

\end{proof}

\begin{proof}[Proof of Proposition \ref{prop-Hartogs-A}]

We construct acceleration data for $P$ and $\partial P$  and homotopy data for the restriction map so that the hypotheses of the previous lemma hold. 

Let us call a smooth function $f:B_k\to \mathbb{R}$ strongly non-degenerate if \begin{itemize}
\item in a neighborhood of the node $f$ only depends on the monodromy invariant integral affine coordinate $u$ and $0 < \frac{\partial f}{\partial u}(0)<1$,
\item the Hessian \begin{equation}\label{eq-hess} \left(\frac{\partial^2f}{\partial q_i\partial q_j}(b)\right) \end{equation} with respect to an (and hence any) integral affine coordinate system $q_1,q_2$ is non-degenerate for every $b\in B^{reg}_k$ such that $df_b\in T_{\mathbb{Z}}^{*}B$, and,
\item it is dissipative for some geometrically bounded almost complex structure.
\end{itemize}
Note that the first bullet point implies that the only $1$-periodic orbits of the Hamiltonian $f\circ\pi_k$ which lie on the nodal fiber $\pi_k^{-1}(0)$ are the constant orbits at the focus-focus singularities. Moreover, the Morse index of each of these critical points of $f\circ\pi_k$ is $2$. To see this note that in a neighborhood of  the singular point there are complex valued coordinates $z_1,z_2$ such that $u\circ \pi_1$ is the map $(z_1,z_2)\mapsto |z_1|^2-|z_2|^2$.

Let $V$ denote the Euler vector field in $B_k^{reg}$. For any strictly convex domain $D$ containing the node in its interior, we can define a smooth function $r_D: B_k^{reg}\to \mathbb{R}$ which satisfies $V(r_D)=r_D$ and $r_D(\partial D)=1.$ 

We can choose a sequence of strictly convex domains $D_1\supset D_2\supset\ldots$ inside $B_k$ such that 
\begin{itemize}
\item $D_i$ contains $P$ in its interior,
\item $\bigcap D_i =P$, and,
\item there exists an increasing sequence $c_i\to 1$ as $i\to \infty$ such that $c_iD_i\subset int(P)$ and  $c_1D_1\subset c_2D_2\subset\ldots$ .
\end{itemize}
Here by strict convexity we mean that in affine coordinates near each point of the $\partial D$ we have strict convexity in the usual sense.

We proceed to construct acceleration data for the pre-images of $P$ and $\partial P$ respectively so that they have the same periodic orbits in degree $0$, and the degree $1$ periodic orbits for $P$ are included in those for $\partial P$.

For $P$, consider a monotone sequence of strongly non-degenerate smooth functions $f_1<f_2<\ldots$ on $B_k$ such that
\begin{itemize}
\item 
the sequence $f_i$ converges pointwise to $0$ on $P$ and to $\infty$ outside of $P$,
\item there are real functions $h_i$ such that $f_i=h_i\circ r_{D_i}$ outside of $r_{D_i}<0.5,$
\item $h_i$ is a convex function attaining its minimum at $1$,
\item $h_i$ is linear near infinity and $C^2$-small for $r_{D_i}<1$, and,
\item in the interior of $D_i$ the only integral slope of $f_i$ is at a maximum.
\end{itemize}

For $\partial P$, we consider a monotone sequence of strongly non-degenerate smooth functions $g_1<g_2<\ldots$ on $B_k$ such that

\begin{itemize}
\item 
the sequence $g_i$ converges pointwise to $0$ on $\partial P$ and to $\infty$ outside of $\partial P$,
\item there are real functions $p_i$ such that $g_i=p_i\circ r_{D_i}$ outside of $r_{D_i}<0.5$,
\item $p_i\geq h_i$ for every $i\geq 1$, 
\item $p_i=h_i$ on $[c_i,\infty)$ for every $i\geq 1$, and,
\item $g_i$ is so that if $(dg_i)_b\in T_{\mathbb{Z}}^{*}B_k^{reg}$ for $b\in int(D_i)$, then the Hessian from \eqref{eq-hess} is not positive definite.
\end{itemize}

\begin{figure}
\includegraphics[width=0.6\textwidth]{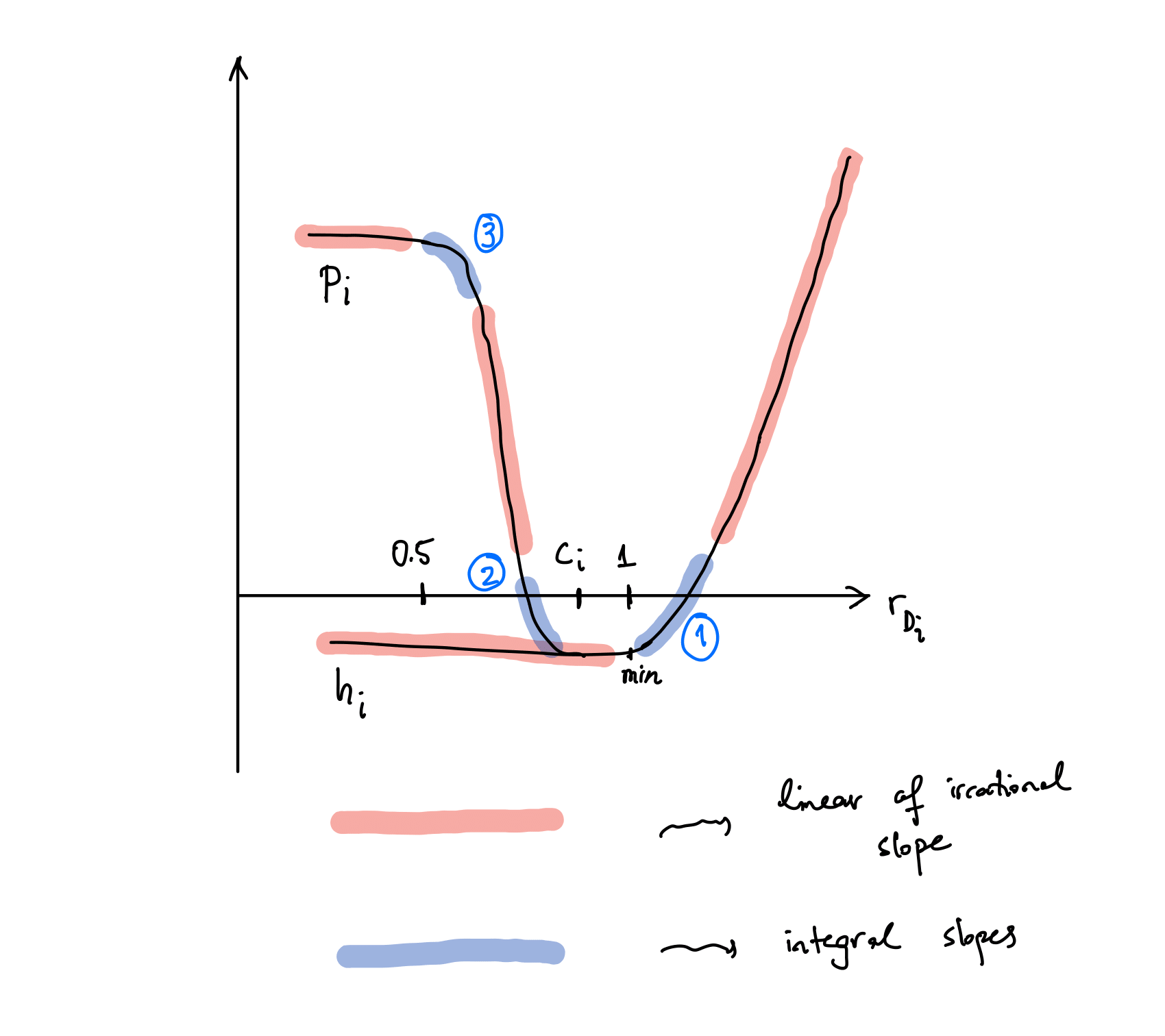}
\caption{The graphs of $h_i$ and $p_i$. Note that to the right of $c_i$ the two graphs agree. }
\label{fig-hartogs}
\end{figure}

The graphs of $h_i$ and $p_i$ are depicted in Figure \ref{fig-hartogs}. To see the last bullet point note that for $(dg_i)_b\in T_{\mathbb{Z}}^{*}B_k^{reg}$ the Hessian from \eqref{eq-hess} is positive definite if and only if the graph of $g_i$ locally lies above the tangent plane to the graph at the point above $b$.

We carefully perturb the Morse-Bott families of orbits with the condition that for all the orbits occurring outside of $c_iD_i$, which is the region where $f_i$ agrees with $g_i$, we use the same perturbations for both $f_i$ and $g_i$. 

To complete the construction of acceleration data for $P$ and $\partial P$, we choose arbitrary monotone interpolations.

To verify conditions of Lemma \ref{lem-alg-hartogs} we claim that there are no generators of negative degree and that all the degree $0$ and $1$ generators of the acceleration datum for $P$  lie in the region where $f_i=g_i$. Moreover, in the region where $g_i>f_i$ all the generators have degrees at least $1.$ See the discussion  in Section \ref{s-noray} right before Proposition \ref{prop-action} regarding indices. The upshot is 
\begin{itemize}
\item all the periodic tori contribute only in degree $\geq 0$,  
\item periodic tori occurring with non-positive definite Hessian contribute only in degrees $\geq1$, 
\item our assumptions regarding strong non-degeneracy similarly imply the critical points occurring in the interior where $f_i\neq g_i$ have index $\geq 2$.
\end{itemize}

The condition that we have an isomorphism in degree $0$ and an injection in degree $1$ now follows. The finite torsion condition follows because rather than using the telescope we can  compute these homologies using the model which completes direct limit at the chain level \cite[Lemma 2.3.7]{varolgunes}. This model is non-negatively graded and it is torsion free by Lemma \ref{lemma-torsion-free} that we prove below. Since zeroth homology simply involves taking a submodule, it does not have any torsion.


\end{proof}

\begin{lemma}\label{lemma-torsion-free}
\begin{enumerate}
\item Let $A_1\to A_2\to\ldots $ be a diagram of $T$-torsion free Novikov ring modules. Then, the direct limit is also $T$-torsion free.
\item If $A$ is a $T$-torsion free Novikov ring module, then its $T$-adic completion is also $T$-torsion free.
\end{enumerate}

\end{lemma}

\begin{proof}
Let's prove (1) first. Take a torsion element in the direct limit. It must come from some $a\in A_n$. For some $R>0$, $T^Ra$ maps to zero in the direct limit. This means that it maps to zero in some $A_N$ with $n>N$. By the torsion free assumption, this implies that in fact $a$ maps to zero in $A_N$, finishing the proof.

Let's now move to (2). Note first that by the torsion free assumption, for every $r>0$ every element in $A/T^rA$ has torsion exactly $r$. Now, let $a$ be an element of torsion $R$ in the completion. This means that the truncation of $a$ in $A/T^rA$ for $r>R$ has torsion $R$ which is a contradiction.

\end{proof}

We finally prove Theorem \ref{thm-local-comp}.

\begin{proof}[Proof of Theorem \ref{thm-local-comp}]
We want to show that we can extend the isomorphism of Theorem \ref{thm-F_k,reg}. By another application of Lemma \ref{lmSheafGluing}, it suffices to extend it to admissible convex polygons containing the node. We define these isomorphisms by the diagram \begin{align}
\xymatrix{ 
\cF_k(P)\ar[d]\ar[r]& \cO_k(P)\ar[d]\\ \cF_k(\partial P)\ar[r] &\cO_k(\partial P),}
\end{align} where the vertical maps are given by Propositions \ref{prop-Hartogs-A} and \ref{prop-Hartogs-B}, for any such $P\subset B_k.$

We need to check that, for an admissible convex polygon $Q\subset P\subset B_k$, \begin{align}
\xymatrix{ 
\cF_k(P)\ar[d]\ar[r]& \cO_k(P)\ar[d]\\ \cF_k(Q)\ar[r] &\cO_k(Q)}
\end{align} commutes.

First, note that when $Q\subset \partial P$, this is immediate as the restriction maps factor through $\partial P$. Now assume that $Q\subset P$ intersects $\partial P$ and let $q\in Q\cap \partial P$. Then, if $Q$ contains the node, $q\in \partial Q$. Consider the diagram
\begin{align}
\xymatrix{ 
\cF_k(P)\ar[d]\ar[r]& \cF_k(Q)\ar[d]\ar[r] &\cF_k(\{q\})\ar[d]\\ \cO_k(P)\ar[r] &\cO_k(Q)\ar[r] &\cO_k(\{q\}).}
\end{align}
We already know that the outer rectangle and the square on the right commute. Noting that $\cO_k(Q)\to\cO_k(\{q\})$ is injective finishes the proof. This injectivity was proved in Propositions \ref{prpUnique} and \ref{prop-analytic-cont-B}.

Finally, assume that $Q$ is arbitrary. We can find an admissible convex polygon $Q'\subset P$ with non-empty interior, which intersects both $Q$ and $\partial P$. Let $q'\in \partial Q'\cap Q$. Morever, if $Q$ contains the node, we also make sure that $q'\in \partial Q.$ Both squares in the diagram
\begin{align}
\xymatrix{ 
\cF_k(P)\ar[d]\ar[r]& \cF_k(Q')\ar[d]\ar[r] &\cF_k(\{q'\})\ar[d]\\ \cO_k(P)\ar[r] &\cO_k(Q')\ar[r] &\cO_k(\{q'\}),}
\end{align} commute so the composition square also commutes. We finish this time using the injectivity of $\cO_k(Q)\to\cO_k(\{q'\}).$

\end{proof}


\section{Construction of mirrors}\label{s-cons}

\subsection{Symplectic cluster manifolds}
Let $\cR$ be an eigenray diagram. As explained in Section 7.1 of \cite{locality} (to be more specific, see the discussion after Definition 7.4),  one can associate to $\cR$ a symplectic manifold $M_{\cR}$ together with a nodal Lagrangian fibration $\pi_{\cR}:M_{\cR}\to B_{\cR}$. The pre-image of an admissible polygon containing a node with multiplicity $k$ is modeled on $\pi_k:M_k\to B_k$ that we have previously analyzed. There are some choices involved in the construction of $\pi_{\cR}$ near each node of $B_\cR$, that we called \emph{fine data} in \cite{locality}, but by abuse of notation we will shall not distinguish them. 

\begin{remark}
    Let us be a little more clear about these choices and their effect on the construction. The data of the integral affine manifold $B^{reg}_{\cR}$ gives rise to a Lagrangian torus fibration $M^{reg}_{\cR}\to B^{reg}_{\cR}$ with a Lagrangian section (the \emph{zero section}). Extending this to a Lagrangian torus fibration with focus-focus singularities $\pi_{\cR}: M_{\cR}\to B_{\cR}$ involves making a choice near each node of a local model with a specified Lagrangian section and a nodal integral affine embedding of its base into $B_\cR$. A priori there is the possibility that the sheaf $\cF_{\cR}$ and hence the rigid analytic space $\cY_{\cR}$ we construct below depend not just on $\cR$ but also on this fine data. 
    
    Let us show this is not the case. Let us temporarily write $\pi_{\cR}: M_{\cR}\to B_{\cR}$ and $\pi_{\cR}': M_{\cR}'\to B_{\cR}$ for the nodal Lagrangian torus fibrations obtained by different choices of such filling. These give rise to sheaves $\cF_{\cR}, \cF'_{\cR}$ over $B_{\cR}.$ If we fix an arbitrary open neighborhood of the nodes $U\subset B_{\cR}$, there is a symplectomorphism $\phi_U:M_{\cR}\to M'_{\cR}$ which is the identity in an open neighborhood of $B_{\cR}\setminus U,$ but not even fiber preserving over $U$ \cite{symington}. Fixing such $U$ and $\phi_U$ we obtain induced isomorphisms $\phi_*(P):\cF_{\cR}(P)\to\cF'_{\cR}(P)$ for $P$ an admissible open that is either disjoint from $U$ or contains it in its interior by the functoriality of relative symplectic cohomology under symplectomorphisms. These isomorphisms are compatible with restriction maps.
    
    We can find $U_1\supset U_2\supset\ldots $ with intersection $N_\cR$ and symplectomorphisms $\phi_{U_i}$ such that there is a Hamiltonian isotopy from $\phi_{U_i}$ to $\phi_{U_{i+1}}$ that is constant in the complement of $\pi_\cR^{-1}(U_i)$. Note that every admissible open in $B_\cR$ that contains a node contains it as an interior point. For each admissible open $P$ and $i$ sufficiently big, we can define an isomorphism $(\phi_{U_{i+1}})_*(P):\cF_{\cR}(P)\to\cF'_{\cR}(P)$. By the Hamiltonian isotopy invariance property of restriction maps \cite{varolgunes}, these isomorphisms are independent of $i$ and they define the isomorphism of sheaves  $\cF_{\cR}\to\cF'_{\cR}$ that we are after.
    
    The reader might also find it enlightening to think about the Hartogs property of Theorem \ref{prop-Hartogs-A} in this light.

\end{remark}

A \emph{symplectic cluster manifold} is a symplectic manifold which is symplectomorphic to $M_\cR$ for some eigenray diagram $\cR$.  A given symplectic cluster manifold may have multiple eigenray diagram representations as a result of nodal slide and branch move operations on eigenray diagrams. See Section 7.2 of \cite{locality}. 

The following is Proposition 7.17 of \cite{locality}.

\begin{proposition}\label{prpClstGeoBd}
Symplectic cluster manifolds are geometrically of finite type.\end{proposition}

If $M$ is a symplectic cluster manifold, then $c_1(M)=0$ and there is a preferred trivialization up to homotopy of $\Lambda^nT_\mathbb{C}M,$ which makes the regular fibers of any associated $\pi_{\mathcal{R}}: M_{\mathcal{R}}\to B_{\mathcal{R}}$ have the Maslov homomorphism $\pi_1(T^2,\star)\to \mathbb{R}$ zero. We use this grading datum in defining relative $SH$ without further mention. From now on we consider $\mathcal{F}(\cdot):=SH^0_{M_\cR}(\pi_\cR^{-1}(\cdot);\Lambda)$ as a sheaf over the $G$-topology of $B_\mathcal{R}$ from Section \ref{ss-eigenray}, and denote it by $\mathcal{F}_{\mathcal{R}}$.

For a small admissible polygon as constructed in Proposition \ref{prpSmall} we say it has multiplicity $k>0$ if it contains a singular value of multiplicity $k$. If it contains no singular value we say it has multiplicity $0$.
\begin{proposition}\label{prpLocalSmall}
Let $P\subset B_{\cR}$ be a small admissible polygon of multiplicity $k$. Then there exists a complete embedding $\iota:M_k\to M_{\cR}$, a convex admissible polygon $P_0\subset B_k$, and a nodal integral affine embedding $f:U\to B_{\cR}$ from an open neighborhood $U$ of $P_0$ so that \begin{itemize}
\item $f(P_0)=P$
\item $\pi_{\cR}^{-1}(P)\subset \iota(M_k)$
\item $\pi_{\cR}\circ\iota|_{\pi_k^{-1}(P_0)}=f\circ \pi_k$.
\end{itemize} 
\end{proposition}
\begin{proof}
Abusing notation, we shall use the term \emph{ray} to refer to the image $\psi_{\cR}(l)\subset B_{\cR}$ of a ray $l\subset \bR^2$ of the eigenray diagram $\cR$. By definition $P$ meets at most one ray. Moreover, if $P$ meets a ray, that ray has total multiplicity $k'\geq k$. Here the total multiplicity of a ray is the sum of the multiplicities of all the nodes on the ray. We can choose Lagrangian tails lying over each eigenray not met by $P$ for each critical point over each node. Since the rays are pairwise disjoint we can pick these tails to be pairwise disjoint. The complement of these tails is symplectomorphic to $M_{\cR'}$ according to \cite[Theorem 1.8]{locality} where $\cR'$ is the eigenray diagram obtained by omitting all the rays not met by $P$. Moreover, this symplectomorphism can be taken to intertwine the Lagrangian fibrations away from arbitrarily small neighborhoods of the removed rays. In case $\cR'=B_k$ we are done. Otherwise the diagram $\cR'$ has a single ray of total multiplicity $k'$ with multiple nodes. The nodes are ordered along that ray. 
For the critical points of $\pi_{\cR'}$ which is not over the node in $P$ we can pick pairwise disjoint Lagrangian tails whose projection to the base does not meet $P$. 
Removing these tails we again obtain according to \cite[Theorem 1.8]{locality} a symplectic manifold which is symplectomorphic to $M_k$ by a symplectomorphism that intertwines the fibration $\pi_{\cR}$ in a neighborhood of pre-image of $P$ with the fibration $\pi_k$.

\end{proof}

We can now prove the five properties listed in Theorem \ref{thm-four-prop}.

\begin{proof}[Proof of Theorem \ref{thm-four-prop}]
The locality isomorphism of Theorem \ref{thm-local} and Proposition \ref{prpLocalSmall} reduce the result to the same statements for $\mathcal{F}_k$ on $B_k$. Theorem \ref{thm-local-comp} reduces it to  $\mathcal{O}_k$ on $B_k,$ which is covered by Corollary \ref{cy4PropertiesOk}.
\end{proof}

\subsection{Gluing}
For $P$ a small admissible polygon in $B_\cR$, we consider the affinoid domains $$M(P):=M(\mathcal{F}_\cR(P)).$$ For $Q\subset P$ small admissible polygons, we have the restriction map $\mathcal{F}_\cR(P)\to \mathcal{F}_\cR(Q)$. Let us denote by $$\iota_{Q\subset P}: M(Q)\to M(P)$$  the maps of affinoid domains induced by the restriction maps. The underlying maps of sets are injective by \cite[\S3.3, Lemma 10]{bosch}.

We denote the collection of all small admissible polygons in $B_\cR$ by $\{P_i\}_{i\in I}$. Let us define,  for each pair $i,j\in I,$ the affinoid subdomains
$$U_{ij}:=\iota_{P_i\cap P_j\subset P_i} (M(P_i\cap P_j))\subset M(P_i).$$
Then, for all $i,j\in I$, we automatically obtain the isomorphism of affinoid domains
$$\psi_{ij}:= \iota_{P_i\cap P_j\subset P_j}\circ \iota_{P_i\cap P_j\subset P_i}^{-1}: U_{ij}\to U_{ji}.$$

\begin{lemma}
For $i,j,k\in I,$ we have $$
\psi_{ij}(U_{ij}\cap U_{ik})\subset U_{ji}\cap U_{jk}.
$$
\end{lemma}
\begin{proof}
Note that we have $$U_{ij}\cap U_{ik}=im({M}(P_i\cap P_j\cap P_k))$$ inside ${M}(P_i)$ by the strong cocycle condition, and the map $\iota_{P_i\cap P_j\cap P_k\subset P_i}$ is the same as $$\iota_{P_i\cap P_j\subset P_i}\circ \iota_{P_i\cap P_j\cap P_k\subset P_i\cap P_j}$$ by the presheaf property of restriction maps.

Therefore, we have a canonical commutative diagram \begin{align}
\xymatrix{ 
&&U_{ji}\\ M(P_i\cap P_j\cap P_k)\ar[r]  &M(P_i\cap P_j)\ar[r]\ar[ur]&U_{ij}\ar[u]_{\psi_{ij}},}
\end{align}where the image of $M(P_i\cap P_j\cap P_k)$ in $U_{ij}$ and  $U_{ji}$ is equal to $U_{ij}\cap U_{ik}$ and  $U_{ji}\cap U_{jk},$ respectively. The proof is immediate.
\end{proof}

For $i,j,k\in I,$ let us give the name 
$$
\psi_{ijk}:U_{ij}\cap U_{ik}\to U_{ji}\cap U_{jk}
$$
to the map obtained by restricting $\psi_{ij}$.

\begin{lemma}
The collection of maps $\psi_{ij}$ satisfies $\psi_{ij}\circ\psi_{ji}=id$, $\psi_{ii}=id$ and the cocycle condition $$\psi_{ijk}=\psi_{kji}\circ \psi_{ikj}$$ as maps of rigid analytic spaces.
\end{lemma}
\begin{proof}
First we verify the identities as maps of sets.
The first two are obvious. The cocycle condition follows from the commutativity of the diagram of isomorphisms:\begin{align}
\xymatrix{ 
&U_{ji}\cap U_{jk}\\ M(P_i\cap P_j\cap P_k)\ar[r]  \ar[r]\ar[ur]&U_{ij}\cap U_{ik}\ar[u]_{\psi_{ijk}}.}
\end{align}

To conclude that these identities are in fact identities of rigid analytic spaces, we observe that the $U_{ij}$ are all affinoid domains and that the morphisms $\psi_{ij}$ and $\psi_{ijk}$ are morphisms of affinoid domains by construction. So it remains to verify they are isomorphims, which amounts to the corresponding algebra maps being isomorphisms. Since all involved spaces are in fact reduced affinoid domains, the last claim is a standard consequence of Hilbert's Nullstellensatz \cite[\S3.2, Theorem 4]{bosch}. Namely, given a pair of reduced affinoid algebras $A,B$ and a pair $f,g:A\to B$ of affinoid algebra maps inducing identical maps $M(B)\to M(A)$
we must have $f=g$. Indeed for $x\in A$ and any maximal ideal  $\mathfrak{m}$ in $B$  we have $f(x)=g(x)\in B/\mathfrak{m}$ since they are both the image of $x$ under the map  $A/f^{-1}(\mathfrak{m})=A/g^{-1}(\mathfrak{m})\to B/\mathfrak{m}.$ Thus $f-g$ vanishes at every maximal ideal, so $f-g$ is nilpotent which by reducedness implies $f-g=0$. 

\end{proof}

As in \cite[\S5.3, Proposition 5]{bosch} we now define a rigid analytic space by gluing $M(P_i)$ along $U_{ij}$'s using $\psi_{ij}$'s. As a set 
\begin{equation}
\mathcal{Y}_{\mathcal{R}}=\coprod_{i\in I}M(P_i)/\sim
\end{equation}
where $\sim$ is the relation $x\in U_{ij} \sim \psi_{ij}(x)\in U_{ji}$.  

Note that by construction there are canonical embeddings of each ${M}(P_i)$ into $\mathcal{Y}_\cR$ and we call their images (which are admissible opens) $X_i\subset \mathcal{Y}_\cR.$ If $P_i\subset P_j$, then $X_i\subset X_j$. Moreover $X_i\cap X_j$ is the image of $M(P_i\cap P_j).$

%
%

We now want to construct a map $p_\cR:\mathcal{Y}_\cR\to B_\cR$ using the properties listed in Theorem \ref{thm-four-prop}.

\begin{proposition}\label{propProjInt}Let $y\in \mathcal{Y}_{\mathcal{R}}$. Denote by $I_y\subset I$ the set of indices for which $P_i$ is a small admissible polygon such that $y\in X_i$. Then the set $$\bigcap_{i\in I_y} P_i\subset B_\cR$$ is a singleton $\{b_y\}$.
\end{proposition}
\begin{proof}


We first claim that there is a sequence of small admissible polytopes $$P_{i(1)}\supset P_{i(2)}\supset \ldots $$ such that \begin{itemize}
\item $i(k)\in I_y$ for all $k\geq 1$.
\item $\bigcap_{k\geq 1} P_{i(k)}$ is a singleton.
\end{itemize}

Consider the explicit PL isomorphism $\psi_{\cR}:\mathbb{R}^2\simeq B_\cR$ introduced in \S \ref{ss-eigenray}. Let us use the flat metric on $\mathbb{R}^2$ to equip $B_\cR$ with a metric space structure.

By construction $I_y$ is non-empty. Let $i(1)\in I_y$. Partition $P_{i(1)}=\bigcup_{k=1}^n P_{j(k)}$ into admissible convex polygons (which are automatically small) with half the diameter. Using the independence property we know that $y\in X_{j(k)}$ for some $k$. We let $i(2)=j(k)$ and thus continue inductively. The polygons $P_{i(k)}$ are compact in the standard topology on $B_{\cR}$ so we deduce that the intersection  $\bigcap_{k\geq 1} P_{i(k)}$ is non-empty and of $0$ diameter. That is, a singleton.

It follows, first, that the intersection $\bigcap_{i\in I_y} P_i$ has at most one element. It remains to show non-emptyness. Let $b_y$ be the unique element in $\bigcap_{k\geq 1} P_{i(k)}$.  Let $i\in I_y$ and assume by contradiction that $b_y\not\in P_i$. Then for $k$ large enough, $P_i\cap P_{i(k)}=\emptyset$. But then we cannot have $y\in X_i\cap X_{i(k)}$ by the gluing construction.  This contradiction completes the proof.
\end{proof}


We thus define a canonical map $$p_\cR: \mathcal{Y}_{\mathcal{R}}\to B_{\mathcal{R}}$$ by $y\to b_y$.

\begin{lemma}\label{lmRecoverXi}
For any small admissible polygon $P_i$  we have $p_{\cR}^{-1}(P_i)= X_i$ .
\end{lemma}
\begin{proof}
For $y\in X_i$ we have $b_y\in P_i$ by construction. So $p_{\cR}(X_i)\subset P_i$. Conversely, suppose $b\in P_i$. We want to show that $p_\cR^{-1}(b)\subset X_i$. Let $y\in p_\cR^{-1}(b),$ that is $b_y=b.$

Let $P_j$ be any small admissible polygon such that $y\in X_j$. Let $P_k=P_i\cap P_j$. Note that $b$ is contained in $P_k$. Then we claim that $y\in X_k$. Assuming $y\notin X_{k}$, by the separation property, there is a $P_l\subset P_j$ which does not contain $b$ so that $y\in X_l$. This contradicts the construction of $b=b_y$. Therefore, we have that $y\in X_k\subset X_i$.

%
%

\end{proof}

%
%
\begin{proposition}\label{prppcrlocalform}
If $P\subset B_{\cR}$ is a small admissible polygon of multiplicity $k$ there is a commutative diagram of sets
\begin{align}\xymatrix
   {
     p_{\cR}^{-1}(P)\ar[d]^{p_{\cR}}\ar[r]& Y_k\ar[d]^{p_k}\\
      P \ar[r] & B_k
   }
\end{align}   
where the upper horizontal arrow is an analytic embedding and the lower arrow is a nodal integral affine embedding.

\end{proposition}

\begin{proof}
By Lemma \ref{lmRecoverXi}  we have $p_{\cR}^{-1}(P_i)= X_i\simeq M(P_i)$. The upper horizontal arrow is the one induced by the locality isomorphism and the lower arrow is the inverse of the map $f$ from Proposition \ref{prpLocalSmall} restricted to $f^{-1}(P)$. It remains to show commutativity. 


For this, note that for $P_i\subset P_j$ admissible polygons, locality and the isomorphism of Theorem \ref{thm-local-comp}  intertwine the inclusion $M(P_i)\to M(P_j)$ with the inclusion $p_k^{-1}(P_i)\to p_k^{-1}(P_j)$. Second, for any $y\in Y_k$ we have tautologically  that $p_k(y)=\cap_{i:y\in p_k^{-1}(P_i)}P_i$  where $i$ runs over all small admissible polygons in $B_k$. The commutativity thus follows by construction of $p_{\cR}$.  
\end{proof}

\begin{remark}
    The horizontal arrows in Theorem \ref{prppcrlocalform} are not claimed to be unique.
\end{remark}

So far we have worked with an index set $I$ indexing all the small admissible polygons. We now consider the possibility of using proper subsets of $I$ to carry out the construction. Let $J\subset I$ be any subset. Then  we can define a rigid analytic space 
$$
\cY_{\cR}(J)= \coprod_{j\in J}M(P_j)/\sim 
$$
where for $j_1,j_2\in J$ we identify points of the image of $M(P_{j_1}\cap P_{j_2})\to M(P_{j_1})$ with the corresponding points of the image of $M(P_{j_1}\cap P_{j_2})\to M(P_{j_2}).$ For each $j\in J$ we let $X_j(J)$ be the image of $M(P_j)$ in $\cY_{\cR}(J)$. Then the $X_j(J)$ form an admissible cover in the strong Grothendieck topology on $\cY_{\cR}(J)$. The intersection $X_{j_1}(J)\cap X_{j_2}(J)$ is the image of the embedding $M(P_{j_1}\cap P_{j_2})\to \cY_{\cR}(J)$.   The rigid analytic space $\cY_{\cR}(J)$ is endowed
with a canonical map of rigid analytic spaces $h_J: \cY_{\cR}(J)\to\cY_{\cR}$. Indeed, for each $j\in J$ we have a map $h_j:X_j(J)=M(P_j)\to \cY_{\cR}$ , and for any $j_1,j_2\in J$ we have $h_{j_1}=h_{j_2}$ on overlaps. Thus the $h_j$ glue to a morphism $h_J$ by \cite[\S5.3, Proposition 6]{bosch}

We now investigate for what subsets $J\subset I$ the canonical map $h_J$ is an isomorphism.  We call $J\subset I$ an \emph{admissible index set} if for any $i\in I$ we have that $P_i$ is contained in a finite union  $\cup_{j\in S}P_j$ where $S\subset J$ is finite.

\begin{proposition}\label{prpSubcover}
    When $J\subset I$ is an admissible index set the map 
     $ h_J: \cY_{\cR}(J)\to\cY_{\cR}$ is an isomorphism.
 \end{proposition}
\begin{proof}
    We first construct a map $f_J:\cY_{\cR}\to \cY_{\cR}(J)$. For any $i\in I$  we have a map of rigid analytic spaces $f_i:M(P_i)\to \cY_{\cR}(J)$ defined as follows. Fix a finite cover $P_i=\cup Q_k$ by small admissible convex  so that for each $k$ there is a $j(k)\in J$ with $Q_k\subset P_{j(k)}$. This is possible by admissibility of $J$. Then we have the maps of rigid analytic spaces $M(Q_k)\to P_{j(k)}\subset \cY_{\cR}(J)$ which agree on overlaps, giving rise to a map $M(P_i)\to \cY_{\cR}(J)$. This map is easily seen to be independent of the cover of $P_i$ or the assignment $k\mapsto j(k)$. Moreover, given $i_1,i_2\in I$ the restrictions of $f_{i_1},f_{i_2}$ to $M(P_{i_1}\cap P_{i_2})$ agree. We can thus glue to obtain a map $f_J$ as required. It is straightforward, but tedious, to verify that $f_J$ is indeed inverse to $h_J$. 
\end{proof}

We now investigate when the projection map $p_{\cR}:\cY_{\cR}\to B_{\cR}$ can be reconstructed from a smaller index set $J\subset I$. We say that $J\subset I$ is \emph{sufficiently fine} if for every point of $b\in B_\cR$ there is a descending chain $$P_{j_1}\supset P_{j_2}\supset \ldots $$ with $j_i\in J$ such that $\bigcap_{k\geq 1} P_{j_k}=\{b\}$.

\begin{proposition}\label{prpFine}
    Suppose $J\subset I$ is a sufficiently fine admissible index set. For $y\in \cY_{\cR}(J)$ denote by $J_y\subset J$ the set of indices $j$ for which $y\in X_j(J)$. Then the set $$\bigcap_{j\in J_y} P_j\subset B_\cR$$ is a singleton $\{b_y(J)\}$. Let $p_{\cR,J}:\cY_{\cR}(J)\to B_{\cR}$ be the map $y\mapsto b_y(J)$ then we have a commutative diagram 
    $$
    \xymatrix{\cY_{\cR}(J)\ar[rd]^{p_{\cR,J}}\ar[r]^{h_J}&\cY_{\cR}\ar[d]^{p_{\cR}}\\
    &B_{\cR}}.
    $$
\end{proposition}
 \begin{proof}
    First observe that if $y\in X_j(J)$ then $h_J(y)\in X_j=h_J(X_j(J))$. So, $J_y\subset I_{h_J(y)}$ and we have an inclusion $$\bigcap_{i\in I_{h_J(y)}} P_i\subset \bigcap_{j\in J_y} P_j.$$ Thus the right hand side is non-empty by Proposition \ref{propProjInt}. If $y\in X_{j_1}(J)\cap X_{j_2(J)}$ it must lie in the image of the natural map $M(P_{j_1}\cap P_{j_2})\to \cY_{\cR}(J)$ and thus $P_i\cap P_j\neq\emptyset$. Since $J$ is sufficiently fine, it follows $\bigcap_{j\in J_y} P_j$ has at most one element. The commutativity of the diagram is just unwinding of definitions. 
 \end{proof}   


    


Recall the notion of small polygons was introduced in \S\ref{ss-sheaf-G} and involved some choices of strips and piecewise affine segments in ${B_\cR}$. As a corollary we can now establish the independence of our construction on these choices.  Let us introduce a superscript  $\cY_{\cR}^S$, $p_{\cR}^S$ to emphasize this choice. Let $S'$ be any other such choice so that the set of $S'$-small admissible  polygons satisfies the axioms listed in Proposition \ref{prpSmall}. 

\begin{corollary}\label{CorSmallIndep}
  There is a canonical isomorphism $\cY^{S}_{\cR}\simeq\cY^{S'}_{\cR}$ intertwining $p^S_{\cR}$ with $p^{S'}_{\cR}$. 
\end{corollary}
\begin{proof}
 Using the covering axiom we can consider a sufficiently fine admissible index set $J\subset I^S$ consisting of $i$  such that $P_i$ is also $S'$-small. $J$ can then be embedded into the index set $I^{S'}$ of all $S'$-admissible polygons as an admissible index set which is sufficiently fine.  We thus get a chain of isomorphisms $\cY^{S'}_{\cR}\to\cY^{S'}_{\cR}(J)=\cY^{S}_{\cR}(J)\to\cY^S_{\cR}$ which intertwine the projection maps to $B_{\cR}$. 
\end{proof}

\begin{proof}[Proof of Theorem \ref{thm-mirror-cons}]
We have already constructed $\cY_{\cR}$ with its map $p_{\cR}:\cY_{\cR}\to B_{\cR}$. That $p_{\cR}$ is a Stein continuous map follows from Lemma \ref{lmRecoverXi} and the properties of the gluing construction of analytic spaces. Let us denote the structure sheaf of $\cY_{\cR}$ by $\cO_{\cR}.$ 

By construction and Lemma \ref{lmRecoverXi}, for a small admissible polygon $P$, $$(p_\cR)_*\cO_{\cR}(P)=\mathcal{F}_\cR(P).$$ Invoking Lemma \ref{lmSheafGluing} we obtain the desired isomorphism of sheaves of algebras $\cF_{\cR}\simeq p_{\cR}^*\cO_{\cR}$.

%

Let us now start proving the numbered assertions. The items (1) and (2) follow immediately from Proposition \ref{prppcrlocalform} and the discussion in Section \ref{Sec-Yk-analyt} for $p_k:Y_k\to B_k$. For item (3), we in addition need Proposition \ref{prop-hol-vol-form} and the fact that $\Omega_0$ is invariant under the action of the integral affine transformations $SL(2,\mathbb{Z})\ltimes \mathbb{R}^2$ on $Y_0^{an}=(\Lambda^*)^2$.

For (4), we note that while a branch move modifies the eigenray presentation, it does not affect the integral affine structure. That is, $B_{\cR_1}$ and $B_{\cR_2}$ are canonically identified with each other as nodal integral affine manifolds. The eigenray presentation does affect the definition of the notion of small polygon which plays an auxiliary role in our construction. To address this, identify $B\simeq B_{\cR_1}\simeq B_{\cR_2}$. Consider the index set $J\subset I$ of polygons in $B_{\cR_1}$ which are also small admissible with respect to $B_{\cR_2}$ under the identification. Then $J$ is a sufficiently fine admissible index set. The claim now follows by Propositions \ref{prpSubcover} and \ref{prpFine}. 

For (5), we note that a nodal slide does modify the integral affine structure. This time  we have a PL isomorphism  $\psi_{\cR_1}\circ\psi_{\cR_2}:B_{\cR_1}\to B_{\cR_2}$, where for an eigenray diagram the map $\psi_{\cR}:\bR^2\to B_{\cR}$ was introduced in \S \ref{ss-eigenray}. Hence we will identify both with $\bR^2$ via $\psi_{\cR_i}$. The sliding is obtained by translating a node (which may be the starting node) along a given ray $l$. Without loss of generality $\cR_2$ is obtained from $\cR_1$ by moving a node of $l\subset\cR_1$ in the direction of the non-compact end of $l$. Since we are not allowing any multiplicity changes, we can find a small admissible polygon $P$ of $B_{\cR_1}$ which covers the entire sliding locus. By possibly adjusting the choices of strips and segments in the definition of small polygon in for $B_{\cR_2}$ (as allowed by Corollary \ref{CorSmallIndep}) we can assume $P$ is also a small admissible polygon for $\cR_2$. Fix a small neighborhood $U$ contained in $P$ and containing the sliding locus. We can find a symplectomorphism $\phi:M_{\cR_1}\to M_{\cR_2}$ which intertwines the torus fibrations $\pi_{\cR_1},\pi_{\cR_2}$ away from the pre-image of $U$ \cite[Proposition 7.8]{locality}. We now consider the set $J\subset I$ of all small admissible polygons for $\cR_1$ which are also small admissible for $\cR_2$ and which either contain $U$ in their interior or are disjoint from $U$. Then $J$ is an admissible index set. For each $j$ in $J$ the symplectomorphism $\phi$ produces an isomorphism of BV algebras $\phi_*:SH^*_{M_{\cR_1}}(\pi^{-1}_{\cR_1}(P_j))\to SH^*_{M_{\cR_2}}(\pi^{-1}_{\cR_2}(P_j))$ which commutes with restriction maps involving polygons $\{P_j\}_{j\in J}$ and their finite intersections. We thus get an induced isomorphism of the corresponding rigid analytic spaces.  


Finally, for (6), it suffices to note that both sides can be defined as an inverse limit over the preimage of an exhaustion of $B_\cR$ by  admissible polygons.

\end{proof}

\begin{remark}Our proof of the first three bullet points of Theorem \ref{thm-four-prop} rely on a computation (not just properties) and is ad-hoc.
At least when all the multiplicities are $1$, we expect to be able to give a much more conceptual proof based on the fact that there is a Lagrangian section $L$ satisfying the local generation property after recasting the construction of $\cY_\cR$ using the relative Lagrangian Floer homology  of $L$. 

The intrinsic meanings of (2) and (3) of Theorem \ref{thm-mirror-cons} are clear. One might consider replacing (1) with the statement that  $p_\cR$ is the projection to the essential skeleton defined using the canonical volume form on $\cY_\cR$. 
\end{remark}

\section{Future work}\label{secFutureWork}

\subsection{Higher dimensions}

In this section we want to briefly explain how our methods would extend to higher dimensions. Let us start with a generalization that is essentially straightforward.

Consider a triple $(\mu, f, Z)$ where \begin{itemize}
    \item $\mu:\mathbb{R}^n\to \mathbb{R}$ is an integral affine function, 
    \item $f\in\mathbb{Z}^n\setminus\{0\}$ satisfies $d\mu(f)=0$, and,
    \item $Z$ is a properly embedded connected codimension $1$ smooth submanifold $Z$ of $$H:=\{x\in\mathbb{R}^n\mid \mu(x)=0\}$$ which is transverse to $f$ everywhere (inside $H$).
\end{itemize}
Define a subset $S(\mu, f, Z)$ of $\mathbb{R}^n$ by considering the closure of the $f$-positive side of $Z$ inside $H$. 

Given $(\mu, f, Z)$, we can define a nodal integral affine manifold with the singular locus $Z$ by regluing a connected neighborhood of $S:=S(\mu, f, Z)$ by the identity on the $\mu(x)> 0$ side  and by the transvection
\begin{equation} v\mapsto v-\mu(v)\cdot f\end{equation} on the other.

We define a shear-cut diagram $\mathcal{R}$ in $\mathbb{R}^n$ as a finite collection of such triples $(\mu_i, f_i, Z_i), i\in I$ such that $S_i:=S(\mu_i, f_i, Z_i)$ are pairwise disjoint. Given a shear cut locus $\mathcal{R}$, we can define a Lagrangian fibration $\pi_\mathcal{R}: M_\mathcal{R}\to B_\mathcal{R}$ with only focus-focus singularities such that the induced nodal integral affine structure is isomorphic to the one induced on $\mathbb{R}^n$ by $\mathcal{R}$ \cite{bernard}.

We have analogues of the nodal slide and branch cut operations. For nodal slides, we can isotope $Z_i$ inside $H_i=\{\mu_i(x)=0\}$ as a properly embedded smooth submanifold keeping it transverse to $f_i$. This does not change $M_\mathcal{R}$ but modifies the fibration $\pi_\mathcal{R}$. Assuming that $H_i$ is disjoint from all the other $S_j$'s, we can apply a branch move by replacing $f_i$ with $-f_i$. This only changes how we represent our Lagrangian fibration using a shear-cut diagram.

Let us identify $B_\mathcal{R}$ with $\mathbb{R}^n$ using the canonical PL homeomorphism and denote the singularity locus $\bigcup_{i\in I}Z_i$ by $Z$. We now start assuming that $Z_i$ are in fact affine. This means that there are neighborhoods of $Z_i$ which are for some $k_i$ isomorphic to a neighborhood of $\mathbb{R}^{n-2}\times \{\text{node}\}$ inside $\mathbb{R}^{n-2}\times B_{k_i}.$ Let us fix such isomorphisms and call this data a framing.

We define an admissible convex polygon in $B_\mathcal{R}$ to be a convex rational polygon if it does not intersect $Z$, and if it does intersect, as a product of admissible convex rational polygons inside $\mathbb{R}^{n-2}$ and  $B_{k_i}$ with respect to a fixed framing.

Using locality for complete embeddings, Kunneth theorem,  e.g.,  \cite[Theorem 10.6]{groman} and some torsion analysis  it should be easy to prove Theorem \ref{thm-four-prop} and then construct a mirror $p_\cR: \cY_\cR\to B_\cR$. 

We can go further and consider more complicated constructions. To illustrate the idea let us focus on dimension $3$. In that case we consider $(\mu,\Delta,Z)$ where $\mu:\bR^3\to\bR^2$ is an integral affine map, $\Delta\subset\bR^2$ is a tropical curve, and $Z\subset \mu^{-1}(\Delta)$ is a graph projecting bijectively onto $\Delta$ under $\mu$. We label the two components of $\mu^{-1}(\Delta)\setminus Z$ by $j\in\{\pm\}$, and denote the closure by $P_j(\mu,\Delta,Z)$. Assume that we are given $\cR$, a collection of $(\mu_i,\Delta_i,Z_i)$ and a choice of signs $j_i$ such that $P_{j_i}(\mu_i,\Delta_i,Z_i)$ are pairwise disjoint. Then, we can form by the same method as above an integral affine manifold with nodal singularities on the edges. At each vertex we get the singular integral affine structure of the \emph{positive vertex}. Accordingly we can form a symplectic manifold $M_{\cR}$ by gluing a neighborhood of the singular fibers of the Gross-fibration corresponding to $\Delta$ \cite{bernard}. See \cite{locality} for a more detailed discussion of the local models.   
We expect the methods of this paper to extend to allow the proof  of Theorem  \ref{thm-four-prop} and moreover to compute the local models. 

\bibliographystyle{plain}
\bibliography{nodalmirror}
\end{document}